\documentclass[a4paper,11pt]{amsart}

\usepackage{amsmath}
\usepackage{amsfonts,amssymb, color}

\usepackage{latexsym}
\usepackage{geometry}
\newcommand{\Pro}{\mbox{Proj}}

\usepackage{bbm}

\usepackage{latexsym}
\usepackage{enumitem}

\usepackage{amsfonts}
\usepackage{dsfont}
\usepackage{amsthm}
\usepackage{color}

\usepackage{xcolor}
\definecolor{apricot}{rgb}{0.98, 0.81, 0.69}
\definecolor{babyblue}{rgb}{0.54, 0.81, 0.94}
\definecolor{american}{rgb}{1, 0.1, 0.24}

\newcommand{\jott}{j}
\newcommand{\Jott}{J}

\newcommand{\nsH}{{\bf H}}
\newcommand{\nsV}{{\bf V}}

\newcommand{\YY}{\mathbb{Y}}

\theoremstyle{plain}
\newtheorem{theorem}{Theorem}[section]
\newtheorem{lemma}[theorem]{Lemma}
\newtheorem{corollary}[theorem]{Corollary}
\newtheorem{proposition}[theorem]{Proposition}

\newtheorem{claim}[theorem]{Claim}

\theoremstyle{definition}
\newtheorem{definition}[theorem]{Definition}
\newtheorem{assumption}[theorem]{Assumption}

\theoremstyle{remark}
\newtheorem{remark}[theorem]{Remark}

\newcounter{lil1}
\newenvironment{step}
{\begin{list} { \bf Step (\Roman{lil1})}
	{ \usecounter{lil1}
		\setlength{\leftmargin}{0.0cm}
		\setlength{\topsep}{0.2cm}
		\setlength{\itemsep}{0.0cm}
		\setlength{\parsep}{0.1cm}
		\setlength{\itemindent}{0.8cm}
		\setlength{\parskip}{0.0cm}}}
{\end{list}}

\newcounter{stepcounter}
\newenvironment{stepproof}
{\begin{list} { \bf Step (\alph{stepcounter}) }
	{ \usecounter{stepcounter}
		\setlength{\leftmargin}{0.0cm}
		\setlength{\topsep}{0.2cm}
		\setlength{\itemsep}{0.0cm}
		\setlength{\parsep}{0.1cm}
		\setlength{\itemindent}{0.8cm}
		\setlength{\parskip}{0.0cm}}}
{\end{list}}

\newcounter{lil2}

\newcounter{michicounter}
\newenvironment{numref}
{\begin{list} {(\alph{michicounter})}
	{ \usecounter{michicounter}
		\setlength{\leftmargin}{0.4cm}
		\setlength{\topsep}{0.0cm}
		\setlength{\itemsep}{0.0cm}
		\setlength{\parsep}{0.1cm}
		\setlength{\itemindent}{0.1cm}
		\setlength{\parskip}{0.2cm}}}
{\end{list}}

\newcounter{counteins}

\newcounter{lil33}

\newcounter{lil1q}

\renewcommand{\div}{\operatorname{div}}

\newcommand{\eps}{\varepsilon}

\newcommand{\MA}{{ \mathfrak{A}}}

\newcommand{\CX}{\mathcal{X}}
\newcommand{\di}{\mbox{div}}
\renewcommand{\v}{{\bf v}}
\newcommand{\ckk}{c_\kappa}

\newcommand{\ud}{\mathrm{d}}

\newcommand{\bw}{{\bf w}}
\newcommand{\ck}{c_\kappa}
\newcommand{\nk}{n_\kappa}
\newcommand{\ukk}{\bu_\kappa}
\newcommand{\kk}{\kappa}

\newcommand{\bu}{\mathbf{u}}

\newcommand{\bv}{\mathbf{v}}
\newcommand{\bo}{\mathcal{O}}

\newcommand{\abss}[1]{\left\|#1\right\|}
\newcommand{\abs}[1]{\left\vert#1\right\vert}

\newcommand{\Div}{\mbox{div}}

\newcommand{\baray}{\begin{array}{rcl}}
\newcommand{\earay}{\end{array}}
\newcommand{\barray}{\begin{array}{rcl}}
\newcommand{\earray}{\end{array}}

\newcommand\dela[1]{}

\newcommand{\bcase}{\begin{cases}}
\newcommand{\ecase}{\end{cases}}

\newcommand\BX{\mathbb{X}}

\newcommand{\DeltaA}{A}
\newcommand\del[1]{}

\del{

\newcommand\del[1]{}
}

\def\eps{\varepsilon}

\newcommand{\lk}{\left}
\newcommand{\lqq}{\lefteqn}
\newcommand{\rk}{\right}
\newcommand{\la}{\langle}
\newcommand{\ra}{\rangle}

\newcommand{\ep} {\varepsilon }

\newcommand{\CO}{{{ \mathcal O }}}

\newcommand{\CH}{{{ \mathcal H }}}

\newcommand{\CV}{{{ \mathcal V }}}

\newcommand{\CM}{{{ \mathcal M }}}

\newcommand{\BF}{{{ \mathbb{F} }}}
\newcommand{\CF}{{{ \mathcal F }}}

\newcommand{\CL}{{{ \mathcal L }}}

\newcommand{\RR}{{\mathbb{R}}}

\newcommand{\NN}{\mathbb{N}}

\newcommand{\kl}{\ell}

\newcommand{\PP}{{\mathbb{P}}}

\newcommand{\EE}{ \mathbb{E} }

\newcommand{\DEQS}{\begin{eqnarray*} }
\newcommand{\EEQS}{\end{eqnarray*} }
\newcommand{\DEQSZ}{\begin{eqnarray} }
\newcommand{\EEQSZ}{\end{eqnarray} }
\newcommand{\DEQ}{\begin{eqnarray}}
\newcommand{\EEQ}{\end{eqnarray}}

\newcommand{\Dcal} {{\mathcal D}}

\newcommand{\Fcal} {{\mathcal F}}

\newcommand{\Mcal} {{\mathcal M}}

\newcommand{\Ocal} {{\mathcal O}}

\newcommand{\Vcal} {{\mathcal V}}

\newcommand{\Xcal} {{\mathcal X}}

\newcommand{\Afrak} {{\mathfrak A}}

\newcommand{\R}{\mathbb{R}}
\newcommand{\N}{\mathbb{N}}

\renewcommand{\P}{\mathbb{P}}
\newcommand{\Eb}{\mathbb{E}}

\newcommand{\C}{\mathbb{C}}

\newcommand{\X}{\mathbb{X}}

\usepackage{geometry} 
\advance\textwidth by -0.3cm
\advance\textheight by -1.0cm
\advance\oddsidemargin by -0.0cm
\advance\evensidemargin by -0.0cm

\usepackage{stackengine,scalerel}

\usepackage{todonotes}

\AtBeginDocument{  }
\usepackage[bookmarksopen, colorlinks=false, pdfpagemode=UseNone, pdfstartview=Fit]{hyperref}
\usepackage{empheq}

\usepackage{mathrsfs}
\usepackage[modulo]{lineno}
\newif\ifshowmarks
\showmarksfalse
\definecolor{revcol}{RGB}{190,0,120}
\definecolor{fixcol}{RGB}{0,80,210}
\definecolor{chgcol}{RGB}{0,120,60}
\newcounter{revcount}
\ifshowmarks
\newcommand{\rev}[1]{\par\smallskip\noindent\refstepcounter{revcount}%
  {\color{revcol}\textbf{[Review \arabic{revcount}]} #1}\par\smallskip}
\newcommand{\revdone}[1]{\par\smallskip\noindent\refstepcounter{revcount}%
  {\color{chgcol}\small\textbf{[Review \arabic{revcount}: settled in this revision.]} #1}\par\smallskip}

\newenvironment{fixblock}{\par\medskip\begingroup\color{fixcol}}{\endgroup\par\medskip}
\newenvironment{chgblock}{\par\begingroup\color{chgcol}}{\endgroup\par}
\else
\newcommand{\rev}[1]{\refstepcounter{revcount}}
\newcommand{\revdone}[1]{\refstepcounter{revcount}}

\fi
\definecolor{notecol}{RGB}{215,95,0}
\definecolor{newcol}{RGB}{125,0,185}
\definecolor{oldcol}{RGB}{140,140,140}
\newcounter{notecount}
\ifshowmarks
\newcommand{\note}[1]{\par\smallskip\noindent\refstepcounter{notecount}%
  {\color{notecol}\textbf{[Note \arabic{notecount}]} #1}\par\smallskip}

\newenvironment{newtext}{\par\begingroup\color{newcol}}{\endgroup\par}
\newcommand{\old}[1]{\par\smallskip\noindent{\color{oldcol}\textbf{[Obsolete with the new operator -- delete from here]}}\par
  {\color{oldcol}#1\par\noindent\textbf{[end of obsolete part]}}\par\smallskip}
\newcommand{\oldi}[1]{{\color{oldcol}[\textit{delete:} #1]}}
\else
\newcommand{\note}[1]{\refstepcounter{notecount}}

\newcommand{\old}[1]{}
\newcommand{\oldi}[1]{}
\fi
\newcounter{checkcount}
\ifshowmarks
\newcommand{\revb}[1]{\par\smallskip\noindent\refstepcounter{checkcount}%
  {\color{revcol}\textbf{[Check \arabic{checkcount}]} #1}\par\smallskip}
\else
\newcommand{\revb}[1]{\refstepcounter{checkcount}}
\fi
\begin{document}

	\title[Stochastic Chemotaxis-Stokes system]{An Abstract Framework for Nonlinear Stochastic Parabolic Systems with Application to a Chemotaxis-Fluid Model}
	\author[E. Hausenblas]{Erika Hausenblas$^{1,\ast}$}
	\dedicatory{\vspace{-10pt}\normalsize{$^{1}$  Department of  Mathematics and Informationstechnology,	Franz Josef Stra{\ss}e 18,	8700 Leoben, Austria\\
			$^{\ast}$ erika.hausenblas@unileoben.ac.at 
		$^{\sharp}$ boris.jidjou-moghomye@unileoben.ac.at}}
	\author[B.\  Jidjou Moghomye]{Boris Jidjou Moghomye$^{1,\sharp}$}
	\date{\today}
\begin{abstract}
	The purpose of the paper is twofold.
Firstly, we want to present an abstract existence Theorem to show the existence of a martingale solution for coupled systems of non-linear stochastic partial differential equations.
The idea is first to split the system by rewriting the nonlinear part in a linear part acting on a given process $\xi$. This is done in such a way that the fixed point with respect to $\xi$ would be the solution of the original Stochastic Partial Differential Equation (SPDE). However, it may be necessary to show the well-posedness of the {\sl linearized} system using a cut-off argument, which may impose additional constraints. 
	We work out under which conditions on the underlying spaces and the cut-off argument one gets a martingale solution of the original SPDE. 
	A similar strategy was applied by Hausenblas and T\"olle  in \cite{jonas} to investigate the stochastic Klausmeier system.
	
	Secondly, we want to verify the full applicability of the existence Theorem by showing the existence of a martingale solution of a highly nonlinear chemotaxis system with underlying fluid dynamics. In particular, we consider
	the following coupled stochastic system
	\begin{align*}
		\begin{cases}
			d {n} +\bu\cdot \nabla ndt= \lk[ r_n\Delta   |n|^{q-1}n- \chi \Div( n\nabla c) \rk]\, dt+g_{\gamma_1}(n) \circ dW_1(t),
			\\
			d{c}+\bu\cdot\nabla cdt  =\lk[r_c \Delta c  -\zeta c+ \beta n \rk]\, dt+g_{\gamma_2}(c) \circ dW_2(t),
			\\
			d \bu +\nabla Pdt=\left[r_\bu\Delta \bu
			+n\star\Phi \right]dt+ \sigma_{\gamma_3}\, dW_3(t),
			\\
			\nabla \cdot \bu =0,
		\end{cases}
	\end{align*}
	with initial condition $(n_0,c_0,\bu_0)$ on a filtered probability space  $\mathfrak{A}$ and $W_1$, $W_2$, and $W_3$ be three  independent time homogeneous spatial Wiener processes over $\mathfrak{A}$. Here $n$ is the cell density,  $c$ is the concentration of the chemical signal, and $\bu$ is a vector field over $\mathcal{O}$.
	The positive  terms $r_n$ and $r_c$  are the diffusivity of the  cells and chemoattractant, respectively,
	the positive value
	$\chi$ is the chemotactic sensitivity. Applying the existence Theorem, we can show the existence of a martingale solution.
\end{abstract}


{
\thanks{  Boris Jidjou Moghomye was supported by the Austrian Science Foundation (FWF), Project P 32295.
	Moreover, we would like to thank Paul Razafamandimby, City University Dublin, Ireland, for his inspiring discussions and helpful remarks. All authors contributed equally.
}}

\maketitle

\section{Introduction}

\del{ While there are man different kinds of stochastically strong solutions, such as classical solutions \cite{DPR04},  mild solutions \cite{Wa86}, Dirichlet solutions \cite{AR91}, variational solutions \cite{DPR04}, or solutions via rough paths and regularity structure \cite{Ha14} constructed on a given probability space, stochastically weak solutions introduced in \cite{Vio76} only reflect the statistics (laws) of the respective processes as solutions of a respective martingale problem, see \cite{Me88b}.}

Coupled systems of nonlinear partial differential equations arise naturally in a wide range of applications, including population dynamics, ecology, biology, and materials science. 
In many situations, such systems are neither monotone nor do they satisfy a maximum principle, so that standard analytical techniques are no longer applicable. 
This difficulty is further amplified when stochastic perturbations are present.

A commonly used approach in the stochastic setting consists of constructing approximate solutions, typically via Galerkin schemes or time discretisation, and deriving uniform bounds by means of suitable Lyapunov functionals. 
Combined with weak compactness arguments, such as the Banach–Alaoglu theorem, one can extract subsequential limits. 
The final and often delicate step is then to verify that the limiting object indeed satisfies the original stochastic system.
Frequently, cut-off arguments are required to establish local well-posedness, and global solutions must subsequently be obtained by exploiting Lyapunov-type estimates.
This strategy has been successfully applied in many works to establish the existence of martingale (or probabilistic weak) solutions.
For nonlinear equations with unbounded nonlinearities driven by Gaussian noise, we refer, for instance, to the monograph \cite{roecknerwei}.

Despite its effectiveness, this approximation-based framework has a significant drawback: for each concrete system, one must design an approximation scheme that preserves the relevant Lyapunov structure. 
In the presence of strong nonlinear couplings or cross-diffusion terms, constructing such schemes can be highly nontrivial.

In this work, we take a different approach.
We first establish an \emph{abstract framework theorem} that specifies structural conditions on strongly coupled nonlinear stochastic parabolic systems, together with assumptions on local solvability and uniform estimates, under which the existence of a probabilistic weak solution can be proved {without relying on any approximation scheme}.
The proof is based on a fixed-point argument formulated at the level of probability laws.
To demonstrate the applicability of this framework, we apply it to a chemotaxis system in porous media coupled with the  Stokes equations and prove the existence of a martingale solution.

To explain the idea more precisely, nonlinear systems often describe the temporal and spatial evolution of a vector-valued quantity
\[
\boldsymbol{n} = (w_1,\dots,w_n), \qquad w_i = w_i(t,x), \quad t\ge0,\ x\in\mathcal{O}.
\]
As an example, a nonlinear reaction--diffusion system models the evolution of the concentrations of chemical species involved in a reaction within a $d$-dimensional domain $\mathcal{O}$:
\begin{equation}\label{e:deteq}
	\partial_t \boldsymbol{n} = A(\boldsymbol{n}) + f(\boldsymbol{n}),
	\qquad \boldsymbol{n}(0)=\boldsymbol{n}_0.
\end{equation}
Here, $A$ typically denotes a (possibly nonlinear) diffusion operator, which may incorporate cross-diffusion effects arising in applications such as population dynamics or materials science, while $f$ represents local reaction mechanisms.

In many applications, intrinsic noise originates from microscopic random motions of particles and environmental fluctuations.
Macroscopic models are often derived as scaling limits of such microscopic systems, where stochastic partial differential equations naturally appear.
For rigorous derivations based on mesoscopic limit theorems, we refer to \textsc{Kotelenez} \cite{Ko08}.
Moreover, biological systems are rarely isolated; they are exposed to random environmental influences such as fluctuating pollution levels or varying external conditions.
Stochastic perturbations are therefore frequently included to model these effects and to enhance the descriptive power of the model.
Examples can be found in the works of \textsc{Christiansen and Fenchel} \cite{fenchel}, \textsc{Nisbet and Gurney} \cite{nisbet}, and \textsc{Lewis et al.} \cite{lewis}.

When stochastic forcing is introduced in the form of a vector-valued Wiener process
\(\mathbf{W}=(\xi_1,\dots,\xi_n)\), where the components $\xi_i$ are independent, the system becomes a stochastic partial differential equation.
The purpose of this paper is twofold.
First, we develop an abstract framework for proving the existence of martingale solutions to coupled nonlinear stochastic partial differential equations.
The key idea is to reformulate the nonlinear terms so that the system can be split into a linear problem parametrised by an auxiliary process, and to characterise solutions of the original system as fixed points of the induced solution map. The fixed point is first constructed pathwise for a time-delayed approximation, in which the auxiliary process is replaced by its shifted Haar projection; a compactness argument in law then yields a martingale solution.
Under suitable assumptions on the underlying spaces and on the linearised problem, which may involve cut-off arguments, the fixed-point approach yields a probabilistic weak solution.
A related strategy was employed by \textsc{Hausenblas and T\"olle}  in \cite{jonas} for the stochastic Klausmeier system, the technique is based on a Schauder-Tychonoff fix point Theorem (see \cite{jonas}  and
\cite{HausenblasKumarToelle2026}).
In contrast to these works, no Schauder fixed-point theorem on the space of laws is needed here: the time delay makes the approximating fixed-point problem solvable by a finite induction over a dyadic time grid.
Also an abstract approach for nonlinear stochastic evolution equations, which also covers systems, was developed by
\textsc{Agresti and Veraar} \cite{AgrestiVeraar2022a,AgrestiVeraar2022b}; for the variational setting see
\cite{AgrestiVeraar2024}. There, the nonlinearities are truncated by a cut-off function, and the truncated problem is solved by
the Banach fixed-point theorem in spaces of stochastic maximal $L^p$-regularity. This requires the nonlinearities to be locally
Lipschitz continuous and yields unique local (maximal) strong solutions; global existence then follows from blow-up criteria
together with additional a priori estimates \cite{AgrestiVeraar2022b}. Our theorem requires neither Lipschitz continuity nor
pathwise uniqueness for the coupled system: only the cut-off system with a given input process has to be well posed, and a
global martingale solution is obtained from compactness and from estimates that are uniform in the cut-off parameter.
However, we only get a martingale (or probabilitic weak in the sense of probability) solution.

\newcommand{\WWW}{\mathbb{W}}
\newcommand{\xib}{{W}}

In the second part of the manuscript, we apply the main Theorem to a chemotaxis system coupled with the Stokes equation. 
Chemotaxis is a fundamental mechanism through which cells and organisms direct their movement in response to chemical cues—specifically, spatial gradients of signaling molecules in their environment. It plays a crucial role in various biological contexts, including fertilization and reproduction, bacterial motility and aggregation, axonal guidance during the development of the nervous system, as well as molecular signaling and directed movement in amoebae. 

In many biological and ecological contexts, chemotactic motion occurs in a dynamic environment rather than a static one. Cells or microorganisms are often suspended in a surrounding fluid, which influences their movement due to fluid transport. Conversely, these organisms can also affect the fluid through their own movements, variations in density, or the production and consumption of chemical substances. This interaction naturally leads to the study of chemotaxis in fluid models.

Such coupled systems arise in various contexts, including the fluid dynamics of mixtures, electrochemistry, cell biology, and biofilm modeling. For a comprehensive overview, we refer to the survey by Horstmann \cite{crossdiff03}.
This gives us the motivation to investigate a stochastic chemotaxis system coupled with fluid dynamics.

\medskip
\newcommand{\footfullcite}{\cite}

The analysis of stochastic chemotaxis--fluid systems has important
precedents. Zhai and Zhang \cite{ZhaiZhang2017} established existence and
uniqueness results for a two-dimensional stochastic
chemotaxis--Navier--Stokes system. Hausenblas, Jidjou Moghomye and
Razafimandimby \cite{HausenblasEtAl2024} likewise obtained probabilistically
strong well-posedness for a two-dimensional stochastic
chemotaxis--Navier--Stokes model. These results concern specific
stochastic models and do not, by themselves, provide an abstract
existence principle for nonlinear coupled stochastic systems.

For deterministic systems with porous-medium-type cell diffusion,
Di Francesco, Lorz and Markowich \cite{DiFrancescoEtAl2010} established
global existence and investigated asymptotic behaviour in a
chemotaxis--fluid model. Tao and Winkler \cite{TaoWinkler2012} proved
global existence and boundedness for a two-dimensional
Keller--Segel--Stokes model for every porous-medium exponent $m>1$.
Winkler \cite{Winkler2015} subsequently studied boundedness and large-time
behaviour in a three-dimensional chemotaxis--Stokes system with nonlinear
diffusion and general sensitivity. These deterministic results provide
an important reference point, but their conclusions cannot simply be
transferred to stochastic equations: stochastic forcing requires
additional estimates and a probabilistic treatment of compactness and
of the limiting stochastic integrals.

The principal contribution of the present paper is methodological:
we formulate and prove a meta-theorem for the existence of martingale
solutions to nonlinear coupled stochastic systems. The theorem makes
explicit the hypotheses under which the construction for auxiliary
systems can be carried through to a solution of the coupled problem.
We then verify these hypotheses for a stochastic porous-medium
chemotaxis--Stokes system. The application establishes existence for the
specific model and assumptions stated below; we do not claim an
improvement of the deterministic diffusion thresholds or of the
uniqueness results available for other stochastic models.

\medskip

The rest of the paper is organized as follows. In Section \ref{metatt}, we state and prove a meta-theorem for the existence result of nonlinear coupled systems. In Section \ref{sec:application}, we prove the existence of a martingale solution to a stochastic porous media chemotaxis-fluid by verifying that the hypothesis of the meta-theorem obtained in Section 2 is satisfied.
In Appendix \ref{Appendix_A_noise}, Appendix \ref{TP}, and Appendix \ref{app:multiplication}, we present some technical preliminaries needed for the proof.

\section{The abstract existence  Theorem} \label{metatt}%

Let us start with introducing the framework within we are working.
 Let $\Ocal\subset \R^d$ be an open domain, $d\ge 1$,  $J\in \mathbb{N}$ and
 $ E,E_1,...,E_\Jott\subset\Dcal^{\prime}(\Ocal)$ be the Hilbert spaces and $ U,U_1,...,U_J\subset\Dcal^{\prime}(\Ocal)$ be the separable Banach spaces\footnote{Here, $\Dcal^\prime(\Ocal)$ denotes the space of Schwartz distributions on $\Ocal$, that is, the topological dual space of smooth functions with compact support $\Dcal(\Ocal)=C_0^\infty(\Ocal)$.}, such that\footnote{The symbol $\hookrightarrow $ denotes a continuous embedding.}
 $$  E\hookrightarrow U,\text{ and } E_\jott\hookrightarrow U_j,\qquad j=1,...,J.$$
Let $\Afrak=(\Omega,\Fcal,\mathbb{F},\P)$ be a filtered probability space
with filtration $\mathbb{F}=(\Fcal_{t})_{t\in [0,T]}$ satisfying the usual conditions. Let $H$ be a separable Hilbert space and let $\mathbb{W}=(W_0,W_1,\ldots, W_\Jott)$, where $\{W_\jott(t)\}_{t\in [0,T]}$, $j=0,\ldots,\Jott$, are cylindrical Wiener processes
in $H$,
such that $W_\jott$ has the representation
$$
W_\jott(t)=\sum_{i\in\mathbb{I}}  \psi_i\beta^j_i(t),\quad t\in [0,T].
$$
Here $\{\psi_i:i\in \mathbb{I}\}$ is a complete orthonormal system in $H$, $\mathbb{I}$ a suitably chosen countable index set, and $\{\beta^j_i:i\in\mathbb{I}\}$ a family of independent real-valued standard Brownian motions on $[0,T]$ over $\Afrak$.
Due to \cite[Proposition 4.7, p. 85]{DaPrZa:2nd}, this representation does not pose a restriction.

\newcommand{\bww}{\textbf{w}}

Here, we are interested in   a solution  $\bww =(w,v^1,...,v^J)$ of the following system of  It\^o stochastic partial differential equations (SPDEs)
\DEQSZ\label{spdes1}
\qquad \lk\{ \barray
dw(t) &=&\lk(\DeltaA_0 w(t)+ F_0(w(t),v^1(t),\cdots,v^\Jott(t))\rk)\, dt +\Sigma_0(w(t))\,dW_0(t),
\\
dv^\jott(t) &=&\lk(\DeltaA_\jott v^\jott(t)+ F_\jott(w(t),v^{\jott}(t),v^{\jott+1}(t),\cdots,v^\Jott(t))\rk)\, dt +\Sigma_\jott(v^\jott(t))\,dW_\jott(t),
\\
&&\quad \jott=1,\ldots ,\Jott,
\\
w(0)&=&w_0,
\\
v^\jott(0)&=&v^\jott_0,\quad \jott=1,\ldots ,\Jott,
\earray\rk.
\EEQSZ
where the equalities hold respectively in $U_j$, $j=0,1,...,J$, (with $U_0:=U$). More precisely, they hold in the larger spaces $K_\jott\supset U_\jott$ of Assumption~\ref{ass:noise}, see Definition~\ref{Def:mart-sol}.
The triangular structure ($F_\jott$
depends only on $w,v^\jott,\dots,v^\Jott$) is essential for the construction of the solution.

We shall also assume that $\DeltaA_0:D(\DeltaA_0)\subset E\to E$ is a possibly nonlinear and measurable (single-valued) operator,
$\DeltaA_\jott:D(\DeltaA_\jott)\subset E_\jott\to E_\jott$, $\jott=1,\ldots,\Jott$, are  linear  (single-valued) operators. The mappings $\Sigma_\jott$, $\jott=0,\ldots,\Jott$, are densely defined
mappings from $E_\jott$ into $L_{HS}(H;E_\jott)$ (with $E_0:=E$) specified later on\footnote{For a separable Hilbert space $H$ and a any Hilbert space $E_j$ defined $L_{HS}(H;E_j)$  the Hilbert-Schmidt operators from $H$ to $E_j$.} and
$$F_0:E\times E_1 \times \cdots \times E_\Jott\longrightarrow U, \ \text{ and } \ F_\jott:E\times E_j \times \cdots \times E_\Jott\longrightarrow U_j,\quad j=1,...J,$$  are nonlinear mapping. 	Later on, we will show in our examples that a chemotaxis system coupled with a laminar fluid can be treated by the meta Theorem.

\begin{assumption}[The Assumption on the Wiener process.]\label{ass:noise}
There are separable Hilbert spaces $K_\jott$, $\jott=0,\ldots,\Jott$, with continuous embeddings $U_\jott\hookrightarrow K_\jott$
($U_0:=U$), and a constant $C_\Sigma>0$ such that every $\Sigma_\jott$, $\jott=0,\ldots,\Jott$, extends to a continuous map
$\Sigma_\jott:K_\jott\to L_{HS}(H;K_\jott)$ with
$$
|\Sigma_\jott(x)|_{L_{HS}(H;K_\jott)}\le C_\Sigma\,\big(1+|x|_{K_\jott}\big),\qquad x\in K_\jott .
$$
\end{assumption}
\noindent We put $\mathbb K:=K_0\times K_1\times\cdots\times K_\Jott$; $\langle\cdot,\cdot\rangle$ denotes the inner product of
$K_\jott$. The spaces $K_\jott$ may be much larger than $U_\jott$; they are used only in the proof of Theorem~\ref{meta}, to
pass to the limit in the stochastic integrals. If $\Sigma_\jott$ is linear, Assumption~\ref{ass:noise} means that
$\Sigma_\jott
$ is bounded from $K_\jott$ into $L_{HS}(H;K_\jott)$; for additive noise it means $\Sigma_\jott\in L_{HS}(H;K_\jott)$.

\begin{assumption}[The Assumptions on the operator $A_j$ and $F_j$.
	]\label{ass:coeff}
Let $\X\subset\{\eta:[0,T]\to U\}$ and $\X_\jott\subset\{\eta:[0,T]\to U_\jott\}$, $\jott=1,\ldots,\Jott$, be {separable} Banach spaces of
paths; $\X_\jott\cap C([0,T];K_\jott)$ carries the sum of the two norms. There are a separable reflexive Banach space $V_0$
with dense and continuous embeddings $V_0\hookrightarrow K_0\hookrightarrow V_0^\ast$ (a Gelfand triple, $K_0$ identified
with its dual), separable Banach spaces $D_i$ densely and continuously embedded in $K_i$, $i=1,\ldots,\Jott$, and numbers
$\alpha>1$, $\lambda_0\ge0$, $\theta\ge1$, $C>0$ and $C_\varphi>0$, $\varphi\in D_i$, with the following properties. We put
$D_0:=V_0$ and $\alpha':=\frac\alpha{\alpha-1}$; $\langle\cdot,\cdot\rangle$ is the duality between $D_i^\ast$ and $D_i$, which
extends the inner product of $K_i$.
\begin{itemize}
\item[(C0)] (The operator $\DeltaA_0$.) $\DeltaA_0:V_0\to V_0^\ast$ is hemicontinuous, i.e.\
$\lambda\mapsto\langle\DeltaA_0(u+\lambda v),x\rangle$ is continuous on $\RR$ for $u,v,x\in V_0$; weakly monotone, i.e.\
$\langle\DeltaA_0u-\DeltaA_0v,u-v\rangle\le\lambda_0|u-v|^2_{K_0}$; and $|\DeltaA_0u|_{V_0^\ast}\le C(1+|u|^{\alpha-1}_{V_0})$,
$u,v\in V_0$.
\item[(C1)] (The operators $\DeltaA_i$, $i\ge1$.) For $i=1,\ldots,\Jott$ and $\varphi\in D_i$, the map
$v\mapsto\langle\DeltaA_iv(\cdot),\varphi\rangle$ is continuous from $\X_i\cap C([0,T];K_i)$ into $L^1(0,T)$, and
$\|\langle\DeltaA_iv(\cdot),\varphi\rangle\|_{L^1(0,T)}\le C_\varphi\big(1+\|v\|^\theta_{\X_i}\big)$.
\item[(C2)] (The nonlinearities $F_i$.) The map $(\xi,v^1,\ldots,v^\Jott)\mapsto F_0(\xi(\cdot),v^1(\cdot),\ldots,v^\Jott(\cdot))$
is continuous from $\X\times\prod_{\jott=1}^\Jott\big(\X_\jott\cap C([0,T];K_\jott)\big)$ into $L^{\alpha'}(0,T;V_0^\ast)$, and its
norm there is bounded by $C\big(1+\|\xi\|^\theta_\X+\sum_{\jott=1}^\Jott\|v^\jott\|^\theta_{\X_\jott}\big)$. For $i=1,\ldots,\Jott$
and $\varphi\in D_i$, the map $(\xi,v^i,\ldots,v^\Jott)\mapsto\langle F_i(\xi(\cdot),v^i(\cdot),\ldots,v^\Jott(\cdot)),\varphi\rangle$
is continuous from $\X\times\prod_{\jott=i}^\Jott\big(\X_\jott\cap C([0,T];K_\jott)\big)$ into $L^1(0,T)$, and its norm in
$L^1(0,T)$ is bounded by $C_\varphi\big(1+\|\xi\|^\theta_\X+\sum_{\jott=i}^\Jott\|v^\jott\|^\theta_{\X_\jott}\big)$.
\item[(C3)] (The cut-off.) For $\jott=1,\ldots,\Jott$ and $s\in[0,T]$, the function $h_\jott(\cdot,s)$ defined after
(\ref{2.5}) is continuous on $\X_\jott\cap C([0,T];K_\jott)$.
\end{itemize}
\end{assumption}
\noindent Below, $\X=L^m(0,T;X)$, see (\ref{eq:MMdef}), and $\X_1,\ldots,\X_\Jott$ are the path spaces of the solutions in
Theorem~\ref{meta}~(\ref{metawellposed}). The constants in Assumption~\ref{ass:coeff} do not depend on $\kappa$.

 For clarity, we define the notion of a martingale solution (or probabilistic weak) to the system (\ref{spdes1}).
\begin{definition}\label{Def:mart-sol}
Fix $T>0$ and initial conditions $\bww _0=(w_0,v^1_0,\ldots,v^\Jott _0)$.	A \textnormal{martingale solution} or \textnormal{weak solution} to the problem
	(\ref{spdes1})
	is given by the tuple
	$$
	\left(\Omega ,{{\mathcal{F}}},\mathbb{P},{\mathbb{F}},
	\mathbb{W}
	, \bww
	\right)
	$$
	such that
	\begin{itemize}
		\item  $\mathfrak{A}:=(\Omega ,{{\mathcal{F}}},{\mathbb{F}},\mathbb{P})$ is a complete filtered
		probability space with a filtration ${\mathbb{F}}=\{{{\mathcal{F}}}_t:t\in [0,T]\}$ satisfying the usual conditions,
		\item the vector-valued Wiener process
		$\mathbb{W}$ is given by $\mathbb{W} =(W_0,W_1,\ldots, W_\Jott )$, where $W_0,W_1,$ $\ldots, W_\Jott $  are mutually independent  cylindrical Wiener processes on $H$ over the probability space
		\item the initial condition $\bww(0)$ is $\CF_0$-measurable on the probability space  with a \emph{prescribed
			law} $\mu_0$ being equal to the law of $\bww_0$;
		$\mathfrak{A}$, and
		\item the tuple $\bww =(w,v^1,\ldots,v^\Jott )$ is such that  $w:[0,T] \times \Omega \rightarrow U$, and $v^\jott:[0,T] \times \Omega \rightarrow
		U_\jott$, $\jott=1,\ldots,\Jott $
		are $\BF$--progressively  measurable processes over $\mathfrak{A}$, such that $\PP$-a.s., $$(w,v^1,\ldots,v^\Jott )\in C([0,T];U)\times C([0,T];U_1)\times\cdot\cdot\cdot\times C([0,T];U_J), $$
		and satisfy  for all $t\in [0,T]$ the following  stochastic integral equations $\PP$--a.s.
		\begin{equation}\label{eq:nStratonovich}
			w(t)=w_0+\int_0^t \lk(\DeltaA_0w(s)+F_0(w(s),v^1(s),\ldots,v^\Jott(s))\rk)\,ds+\int_0^t \Sigma_0(w(s))\, dW_0(s),
		\end{equation}
		in $K_0$ and for  $\jott =1,\ldots, \Jott $,
		\begin{align*}%
		&	v^\jott(t)=v^\jott_0+\int_0^t \lk(\DeltaA_\jott v^\jott(s)+F_\jott (w(s),v^\jott(s),v^{\jott+1}(s),\ldots,v^\Jott (s))\rk)\,ds
\\ &\qquad {}+\int_0^t \Sigma_\jott(v^\jott(s))\, dW_\jott(s),\quad \text{in $K_\jott$}.
		\end{align*}
	\end{itemize}
\end{definition}

The main objective of this section is to prove an existence result for a solution of the abstract stochastic coupled system (\ref{spdes1}) in the sense of Definition \ref{Def:mart-sol}. 
Let $\mathfrak{A}$ denote a  fixed filtered  probability space, i.e., $\mathfrak{A}=(\Omega,\CF,(\CF_t)_{t\in[0,T]},\PP)$.
We find a solution of this nonlinear system as a fixed point of a solution map. The fixed point is constructed pathwise for a shifted dyadic Haar projection  of this map and then passes to the limit by a compactness argument. In this way,  for a certain fixed $m\ge1$, we define the collection of processes
\begin{equation}\label{eq:MMdef}
	\begin{aligned}  \Mcal_{\Afrak}^{m}(\X)
		:= & \Big\{ \xi:\Omega\times[0,T]\to  X\;\colon\;\\
		&\qquad\text{\ensuremath{\xi} is \ensuremath{\mathbb{F}}-progressively measurable}\;\text{and}\;\Eb|\xi|_{\X}^{m}<\infty\Big\}
	\end{aligned}
\end{equation}
equipped with the norm
$$
|\xi|^m_{\Mcal_{\Afrak}^{m}(\X)}:=\Eb|\xi|_{\X}^{m},\quad\xi\in\Mcal_{\Afrak}^{m}(\X),
$$
where $$\X:=L^m(0,T;X)\subset\{ \eta: [0,T]\to U \}$$
 with $X$ a given separable UMD Banach space.

 To overcome the difficulties caused by the non-linearities of $F_j$, we first introduce a  cut-off function. Let $\phi \in C^{\infty}(\RR)$ be a smooth cut-off function that satisfies
$$
\phi(x) \bcase =0, &\mbox{ if } |x|\ge 2,
\\ \in [0,1], &\mbox{ if } 1<|x|<2,
\\=1, &\mbox{ if } |x|\le 1,
\ecase
$$
and let
\begin{equation}
\phi_\kk(x):= \phi(x/\kk), \ x\in\RR,\,\kk \in \mathbb{N}. \label{2.5}
\end{equation}
For  $\jott=1,\ldots, \Jott $, let  $\YY_\jott$ be a function space, e.g.\ $\YY_\jott=L^{m_\jott}(0,T;Y_\jott)$   or $\YY _\jott=C(0,T;Y_\jott)$,  with $1\leq m_j\leq \infty$ and  $Y_\jott$ being some given  UMD Banach spaces of type $2$. Let $h_\jott  :[0,T]\times\YY _\jott\to\RR_0^+$ be given by $$h_\jott (\eta_j,t)=\big\|\eta_j|_{[0,t]}\big\|_{\YY _\jott(0,t)},\ \
  \eta_j\in \YY _\jott, \quad t\in [0,T].$$
Let
$$
\Theta^\jott_\kappa(\eta_j,t):= \phi_\kappa(h_\jott  (\eta_j,t  ))\ \text{ and } \ \Xi^\jott_\kappa(t):=\prod_{i=j}^\Jott \Theta^i_\kappa(v^i_\kappa,t),\quad t\in [0,T].
$$
Here $h_\jott(\eta_j,t)$ is the norm of the restriction of $\eta_j$ to $[0,t]$; it is nondecreasing in $t$ and does
not increase under restriction to a subinterval. The factor $\Xi^\jott_\kappa(t)\in[0,1]$ is the product of the cut-offs of
the components $v^\jott_\kappa,\ldots,v^\Jott_\kappa$, on which $F_\jott$ depends; it is equal to $1$ as long as
$h_i(v^i_\kappa,t)\le\kappa$ for $i=\jott,\ldots,\Jott$. More generally, the cut-off may multiply only a part of $F_\jott$,
i.e.\ $F_\jott=F^{(1)}_\jott+\Xi^\jott_\kappa F^{(2)}_\jott$; this is the case in Section~\ref{sec:application}, see Remark~\ref{rem:partialcut}.
For any $\xi\in \Mcal_{\Afrak}^{m}(\X)$ and $\kappa\in \mathbb{N}$, we consider on the  fixed filtered  probability space $\mathfrak{A}=(\Omega,\CF,(\CF_t)_{t\in[0,T]},\PP)$ given above, the following system

\begin{align}
\label{spdeskappa} \tag{{\arabic{equation}$_\kappa$}}
&\lk\{\barray
dw_\kappa(t) &=&\lk(\DeltaA_0 w_\kappa(t)+   \Xi^1_\kappa(t) \cdot F_0(\xi(t),v_\kk^1(t),\cdots,v_\kk^\Jott (t))
\,\rk)\, dt +\Sigma_0(w_\kappa(t))\,dW_0(t),
\\
dv^\jott_\kk(t) &=&\lk(\DeltaA_\jott   v^\jott_\kk(t)+  \Xi^\jott_\kappa(t)\cdot F_\jott  (\xi(t),v_\kk^{\jott}(t),v_\kk^{\jott+1}(t),\cdots,v_\kk^\Jott (t))\rk)\, dt
\\
&&\quad {}+\Sigma_j(v_\kk^\jott(t))\,dW_\jott(t),\qquad \jott=1,\ldots ,\Jott ,
\\
w_\kk(0)&=&w_0,
\\
{v^\jott_\kk(0)}&=&v^\jott_0,\quad \jott=1,\ldots ,\Jott .
\earray\rk.
\end{align}

\begin{remark}[Partial cut-off]\label{rem:partialcut}
Theorem~\ref{meta} and its proof remain valid if the cut-off multiplies only a part of the nonlinearities, i.e.\ if in
(\ref{spdeskappa}) the term $\Xi^\jott_\kappa F_\jott$ is replaced by $F^{(1)}_\jott+\Xi^\jott_\kappa F^{(2)}_\jott$,
$\jott=0,\ldots,\Jott$ (with $\Xi^0_\kappa:=\Xi^1_\kappa$), where $F_\jott=F^{(1)}_\jott+F^{(2)}_\jott$ and both $F^{(1)}_\jott$ and
$F^{(2)}_\jott$ satisfy (C2). The hypotheses of Theorem~\ref{meta} are then required for this system. The proof uses the
structure of the cut-off only at three places, and each of them carries over.
(i) In Step~I, part (I.3), the drift has to converge along the strongly convergent sequences. For $F^{(1)}_\jott$ this is (C2), and
for $\Xi^\jott_\kappa F^{(2)}_\jott$ it follows from (C2), (C3) and $0\le\Xi^\jott_\kappa\le1$ by dominated convergence, as in
the proof of (\ref{driftconv}).
(ii) In Step~II, all cut-off factors are equal to $1$ on $[0,\tau^\ast_\kappa)$, so $F^{(1)}_\jott+\Xi^\jott_\kappa F^{(2)}_\jott=F_\jott$
there, and the solution of the modified system solves (\ref{spdes1}) up to $\tau^\ast_\kappa$.
(iii) The bound (\ref{metauniform}) is a hypothesis; if it is obtained from an a priori estimate, only
$0\le\Xi^\jott_\kappa\le1$ is used, see Remark~\ref{rem:meta}~(iv).
In Section~\ref{sec:application}, the cut-off $\Theta_\kappa(\bu,\cdot)$ multiplies only the transport terms, i.e.\
$F^{(2)}_0=-\Div(\bu\,\xi)$, $F^{(2)}_1=-\bu\cdot\nabla c$ and $F^{(2)}_2=0$.
\end{remark}

We now describe the requirements in more detail. Let us  assume that the equation \eqref{spdeskappa} is well-posed on a given subset.
Let us now define for any   $R>0$ and $\Phi:\X\to[0,\infty]$ the set
a subset
$\mathcal{X}^\kappa_{\MA}(R)$ of $\mathcal{M}_\MA^m(\mathbb{X})$ by
\DEQSZ\label{characteriseK}
\mathcal{X}^\kappa_{\MA}(R)&:=&\lk\{ \xi\in \mathcal{M}_\MA^m(\mathbb{X}):\EE\Phi_\kappa(\xi)\le {R} \rk\}.
\EEQSZ
Next, we show that for any $\kappa\in\NN$, there exists a  lower semicontinuous function $\Phi_\kappa:\X\to[0,\infty]$, a constant $R_\kappa>0$ such that 
for any stochastic probability space $\mathfrak A=(\Omega,\CF,\BF,\PP)$ with Wiener process $W$, $\CF_0$ measurable initial conditions, and process $\xi\in \mathcal{X}^\kappa_{\MA}(R_\kappa)$
there exists  a unique strong solution  $(w_\kappa(\xi), v^1_\kappa(\xi),...,v^J_\kappa(\xi))$ (in the stochastic sense) 
to the system \eqref{spdeskappa}
 satisfying $w_\kappa(\xi)\in \mathcal{X}_\MA^\kappa(R_\kappa)$. We will see later on that we need  $\mathcal{X}^{\kappa-1}_{\MA}(R_{\kappa-1})\subset \mathcal{X}^{\kappa}_{\MA}(R_\kappa)$.
 
 Knowing this, we now can find a solution to the system 
 \begin{align}
 	\label{spdeskappaexact} \tag{{\arabic{equation}$_\kappa$}}
 	&\lk\{\barray
 	dw_\kappa(t) &=&\lk(\DeltaA_0 w_\kappa(t)+   \Xi^1_\kappa(t) \cdot F_0(w_\kappa ,v_\kk^1(t),\cdots,v_\kk^\Jott (t))
 	\,\rk)\, dt +\Sigma_0(w_\kappa(t))\,dW_0(t),
 	\\
 	dv^\jott_\kk(t) &=&\lk(\DeltaA_\jott   v^\jott_\kk(t)+  \Xi^\jott_\kappa(t)\cdot F_\jott  (w_\kappa (t),v_\kk^{\jott}(t),v_\kk^{\jott+1}(t),\cdots,v_\kk^\Jott (t))\rk)\, dt
 	\\
 	&&\quad {}+\Sigma_j(v_\kk^\jott(t))\,dW_\jott(t),\qquad \jott=1,\ldots ,\Jott ,
 	\\
 	w_\kk(0)&=&w_0,
 	\\
 	v^\jott_\kk(0)&=&v^\jott_0,\quad \jott=1,\ldots ,\Jott .
 	\earray\rk.
 \end{align}
 To be more precise,  on the fixed filtered  probability space $\mathfrak{A}=(\Omega,\CF,(\CF_t)_{t\in[0,T]},\PP)$ given above, for a vector-valued Wiener process $\mathbb{W}=(W_0,W_1,\ldots,W_\Jott)$, we then define the operator
 \DEQS
 {\Vcal^\kappa_{\Afrak,\mathbb{W}}}:\mathcal{X}_\MA^\kappa(R)&\to&\mathcal{M}_{\MA}^m(\X)
 \\ \xi &\mapsto &{\Vcal^\kappa_{\Afrak,\mathbb{W}}}(\xi):=w_\kappa(\xi),
 \EEQS
where $w_\kappa(\xi)$ should be the first component of the solution $(w_\kappa(\xi), v^1_\kappa(\xi),...,v^J_\kappa(\xi))$   of  system \eqref{spdeskappa}.
  The probabilitic weak solution of \eqref{spdeskappaexact} will be the constructed by approximating the   system \eqref{spdeskappa} where we approximate $\xi$ by a piecewise constant processes, i.e. the shifted Haar projection $\Pro_\ell\xi$, see (\ref{hatdefined}).
  On each dyadic interval, $\Pro_\ell\xi$ depends only on $\xi$
 on the preceeding interval. To be more precise, for $\ell\in\NN$ we define
 $$
 \Vcal^{\kappa,\ell}_{\Afrak,\mathbb W}(\xi):=\Vcal^\kappa_{\Afrak,\mathbb W}(\Pro_\ell\xi),\qquad \xi\in\Xcal^\kappa_\Afrak(R_\kappa),
 $$
 and show that there exists a fixpoint $w_\kappa^\ell$ such that $\Vcal^{\kappa,\ell}_{\Afrak,\mathbb W}(w_\kappa^\ell)=w_\kappa^\ell$.
 By the properties imposed on $\Vcal^{\kappa,\ell}_{\Afrak,\mathbb W}$, we can find a 
 sequence $\{w^\ell_\kappa:\ell\in\NN\}\subset\Xcal^\kappa_\Afrak(R_\kappa)$ on the given basis
$\Afrak$. By the continuity (\ref{metacontinuity}) and the uniform integrability
(\ref{metacompact})--(\ref{metauniformintegrability}), $\EE\|w^\ell_\kappa-\Vcal^\kappa_{\Afrak,\mathbb W}(w^\ell_\kappa)\|^m_\X\to0$
(Claim~\ref{cl:asymfix}), i.e.\ $w^\ell_\kappa$ is an asymptotic fixed point of $\Vcal^\kappa_{\Afrak,\mathbb W}$. In the limit
$\ell\to\infty$ the input and the solution coincide, and the limit solves (\ref{spdeskappaexact}).

So  for any $\kappa\in\NN$ we get a 
 family of martingale solutions $\{ w_\kappa^\ell:\ell\in\NN\}$ solving being a fixpoint of $\Vcal^{\kappa,\ell}_{\Afrak,\mathbb{W}}$. 
  By the
 Skorokhod theorem we pass to the limit $\ell\to\infty$. The continuity (\ref{metacontinuity}) and the uniform integrability
 (\ref{metacompact})--(\ref{metauniformintegrability}) show that $w^\ell_\kappa$ is an asymptotic fixed point of
 $\Vcal^\kappa$, i.e.\ $\EE\|w^\ell_\kappa-\Vcal^\kappa(w^\ell_\kappa)\|^m_\X\to0$ (Claim~\ref{cl:asymfix}).
The Wiener process is kept
in the Skorokhod representation, the stochastic integrals pass to the limit by \cite[Lemma 2.1]{DGT2011}, and $\DeltaA_0$
is identified by monotonicity. This gives a
martingale solution of (\ref{spdeskappa}) with $\xi=w_\kappa$, i.e.\ of (\ref{spdeskappaexact}).

\emph{Step II (gluing).} 
For any $\kappa\in\NN$ we showed the existence of a fixpoint  $w_\kappa$  of the system \eqref{spdeskappa} on a new filtered  probability space 
$\mathfrak{A}_\kappa=(\Omega_\kappa,\CF_\kappa,$ $(\CF^\kappa_t)_{t\in[0,T]},\PP_\kappa)$ with respect to the new Wiener process  $\mathbb{W}_\kappa=(W^\kappa_0,W^\kappa_1,\ldots,W^\kappa_\Jott)$. Then, on the time interval
$[0,\tau^\ast_\kappa)$, where
\begin{align}\label{defintau}
	&\tau^\ast_\kappa:=\min_{1\le \jott \le \Jott }\tau^\jott_\kappa\quad\mbox{with}\quad
	\tau^\jott_\kappa=\tau^\jott_\kappa(v^\jott_\kappa)
	:=\inf\{t\in[0,T]:\ h_\jott(v^\jott_\kappa,t)\ge \kappa\}\wedge T,\,\, \jott=1,\ldots,\Jott,
\end{align}
The tuple $\bww _\kappa=(w_\kappa,v_\kappa^1(w_\kappa),\ldots,v^\Jott_\kappa(w_\kappa))$ is on the time interval $[0,\tau_\kappa^\ast]$ ]a probabilistic (non-unique) weak solution of the original system \eqref{spdes1}. Now, we globalise this solution by gluing.
In particular, we follow the solution with cut-off level $1$ up to its stopping time, continue with a
solution with cut-off level $2$ started from the state reached, and so on; the pieces are put on one
probability space by conditional kernels. Here, it is essential to characterise the set in such a way that  one can
transfer the definition to other stochastic bases: membership must depend only on the law of $\xi$, since the limit in the proof of Theorem~\ref{meta} is identified on a new stochastic basis.

Now we have to glue the local solutions together. Since we do not have pathwise uniqueness,
solutions with different $\kappa$ cannot be compared on one probability space. Instead, the global
solution is built by successive extension. We take the solution $\bww_1$ for $\kappa=1$ on
$[0,\tau^\ast_1)$. At time $\tau^\ast_1$ we continue with a solution $\bww_2$ of the system with cut-off
level $\kappa=2$, started from $(w_1(\tau^\ast_1),v^1_1(\tau^\ast_1),\ldots,v^\Jott_1(\tau^\ast_1))$, up to its
stopping time, and so on. The new piece must start from the \emph{same} random variable, not from
an independent copy with the same law $\mu_1$; otherwise the glued process jumps at $\tau^\ast_1$.
Since $\bww_2$ is only a martingale solution, it is defined on a new probability space, and the two
pieces are coupled through the regular conditional distribution of $(\bww_2,\mathbb W_2)$ given its
initial value (concatenation at a stopping time, cf.\ \cite[Section~6.1]{StroockVaradhan}). In this way we
obtain one filtered probability space $\bar{\mathfrak A}$, one Wiener process $\bar{\mathbb W}$, a process
$\bar\bww=(\bar w,\bar v^1,\ldots,\bar v^\Jott)$ and stopping times $\bar\tau_1\le\bar\tau_2\le\cdots$ such
that $\bar\bww$ solves (\ref{spdes1}) on $[0,\bar\tau_\kappa)$ for every $\kappa$.

\emph{Step III (globalising of the solution).} 	 Let us assume that the cut-off functions are chosen in such a way that for any cut-off parameter $\kappa\in\NN$,
there exists a bounded subset $\mathcal{X}_\MA^\kappa$ of $\Mcal_{\MA}^m(\X)$
such that for any $\xi\in\mathcal{X}_\MA^\kappa$
there exists  a unique strong solution  $(w_\kappa(\xi), v^1_\kappa(\xi),...,v^J_\kappa(\xi))$ (in the stochastic sense ) to the system \eqref{spdeskappa} satisfying $w_\kappa(\xi)\in \mathcal{X}_\MA^\kappa$.
The next step is to extend the local solution to a global solution. This will be done by  finding a Lyapunov functional,
$$
\Lambda:\YY _1\times\cdots\times \YY _\Jott\longrightarrow \RR^+_0
$$
such that there exists some $p>0$ and $C>0$ with
$$
\EE_{\mathfrak{A}} \lk|\Lambda(v^1(w_\kappa),v^2(w_\kappa),\ldots, v^\Jott (w_\kappa))\rk|^p\le C,\quad \forall\,\kappa\in\NN.
$$
Here, it is important that the inner part of the cut-off function is bounded by the functional. In particular, we need that there exists some $\delta>0$ such that we have
\DEQSZ\label{thiscase}
\EE\, h_\jott  (v^\jott_\kappa,T)^\delta\le C\EE\, \Lambda(w_\kappa,v^1(w_\kappa),v^2(w_\kappa),\ldots, v^\Jott (w_\kappa)),\quad \forall\,\kappa\in\NN, \, \text{  and  } \jott=1,\ldots \Jott .
\EEQSZ
Since  $h_\jott$ of the cut-off which is uniform in
$\kappa$, assumption (\ref{metauniform}), shows that the stopping times reach $T$ with probability tending to one.
Hence the glued process is a global martingale solution of (\ref{spdes1}).

Assume that the glued processes satisfy the bound (\ref{metauniform}) with a constant independent of $\kappa$.
If $\bar\tau_\kappa<T$, then the $\kappa$-th piece reaches the level $\kappa$, so
$\{\bar\tau_\kappa<T\}\subset\bigcup_{\jott}\{h_\jott(\bar v^\jott_\kappa,T)\ge\kappa\}$, and the Chebyshev inequality gives
$$
\bar\PP(\bar\tau_\kappa<T)\le\sum_{\jott=1}^\Jott C(T)\,\kappa^{-m_\jott}\longrightarrow0,\qquad\kappa\to\infty.
$$
The sets $\{\bar\tau_\kappa\ge T\}$ increase in $\kappa$, so their union $\bar\Omega_0$ satisfies $\bar\PP(\bar\Omega_0)=1$, and on
$\bar\Omega_0$ the glued process solves (\ref{spdes1}) on $[0,T]$. In this way, the existence of a global martingale solution
is proven.

\medskip

We are now ready to formulate our main theorem.

\begin{theorem}[Meta-Theorem]\label{meta}
Let $T>0$ and $m>1$, put $\mathbf U:=U\times U_1\times\cdots\times U_\Jott$, and let $\mathcal I$ be a set of probability
measures on $\mathbf U$ (the admissible initial laws) with $\text{Law}(\bww_0)\in\mathcal I$. For every $\kappa\in\NN$ let
$\Phi_\kappa:\X\to[0,\infty]$ be a lower semicontinuous function, and for $R>0$ and a stochastic basis $\Afrak$ put
\begin{equation}\label{boundedset}
\Xcal^\kappa_\Afrak(R):=\big\{\xi\in\CM^m_\Afrak(\X):\ \EE\,\Phi_\kappa(\xi)\le R\big\}.
\end{equation}

Let us assume that Assumption~\ref{ass:noise} and Assumption~\ref{ass:coeff} are in force. Assume that for every $\kappa\in\NN$, every stochastic basis $\Afrak$ with a Wiener process $\mathbb W=(W_0,\ldots,W_\Jott)$
as above, and every $\CF_0$-measurable initial datum $(w_0,v^1_0,\ldots,v^\Jott_0)$ with law in $\mathcal I$, the following
holds. The numbers $R_\kappa$, $m_0$, $m_1$, {$p$,} $C_0$, $C_1$ and the spaces $\X'$, $\X'_1,\ldots,\X'_\Jott$ may depend on $\kappa$ and on the initial law, but
not on $\Afrak$ and $\mathbb W$.
\begin{numref}
\item \label{metawellposed}
(Well-posedness.) For every $R>0$ and $\xi\in\Xcal^\kappa_\Afrak(R)$, the system (\ref{spdeskappa}) has a unique strong
solution\footnote{Here, uniqueness and strong solution means a solution on the given
	basis $\Afrak$, adapted to $\mathbb F$ and driven by the given $\mathbb W$; uniqueness is pathwise uniqueness in the class of
	solutions with paths in $\X\times\X_1\times\cdots\times\X_\Jott$; the equations hold in $K_0,\ldots,K_\Jott$, as in Definition~\ref{Def:mart-sol}.}
	 $\bww_\kappa(\xi)=(w_\kappa(\xi),v^1_\kappa(\xi),\ldots,v^\Jott_\kappa(\xi))$ with $w_\kappa(\xi)\in\CM^m_\Afrak(\X)$ and
$\bww_\kappa(\xi)\in\X\times\X_1\times\cdots\times\X_\Jott$ $\PP$-a.s., with the path spaces of Assumption~\ref{ass:coeff}. This defines the operator
$$
\Vcal^\kappa_{\Afrak,\mathbb W}:\Xcal^\kappa_\Afrak(R)\to\CM^m_\Afrak(\X),\qquad
\xi\mapsto\Vcal^\kappa_{\Afrak,\mathbb W}(\xi):=w_\kappa(\xi).
$$
\item \label{metainvariant}
(Invariance.) There is $R_\kappa>0$ such that $\Xcal^\kappa_\Afrak(R_\kappa)\neq\emptyset$ and $\Vcal^\kappa_{\Afrak,\mathbb W}$
maps $\Xcal^\kappa_\Afrak(R_\kappa)$ into itself.
\item \label{metacontinuity}
(Continuity.) 
The operator $\Vcal^\kappa_{\Afrak,\mathbb W}$ is uniformly continuous on $\Xcal^\kappa_\Afrak(R_\kappa)$ with respect to the
norm of $\CM^m_\Afrak(\X)$, i.e.\ there is a function $\omega_\kappa:[0,\infty)\to[0,\infty)$ with $\omega_\kappa(x)\to0$ as
$x\downarrow0$ such that
$$
\EE\big\|\Vcal^\kappa_{\Afrak,\mathbb W}(\xi_1)-\Vcal^\kappa_{\Afrak,\mathbb W}(\xi_2)\big\|^m_\X
\le\omega_\kappa\big(\EE\|\xi_1-\xi_2\|^m_\X\big),\qquad \xi_1,\xi_2\in\Xcal^\kappa_\Afrak(R_\kappa).
$$
\item \label{metacompact}
(Compactness.) There are $m_0>1$, Banach spaces $\X'$ and
$\X'_\jott$, $\jott=1,\ldots,\Jott$, with compact embeddings $\X'\hookrightarrow\X$ and $\X'_\jott\hookrightarrow\X_\jott$, and $C_0>0$ such that
$\EE\|w_\kappa(\xi)\|^{m_0}_{\X'}+\sum_{\jott=1}^\Jott\EE\|v^\jott_\kappa(\xi)\|^{m_0}_{\X'_\jott}\le C_0$ for all $\xi\in\Xcal^\kappa_\Afrak(R_\kappa)$,
and the laws of $w_\kappa(\xi)$ and $v^\jott_\kappa(\xi)$, $\xi\in\Xcal^\kappa_\Afrak(R_\kappa)$, are tight in $C([0,T];K_0)$ and in
$C([0,T];K_\jott)$, respectively.
\item \label{metauniformintegrability}
(Higher moments.) There are $m_1>m$, $p>2$, and $C_1>0$ such that, for all $\xi\in\Xcal^\kappa_\Afrak(R_\kappa)$,
$$
\EE\|w_\kappa(\xi)\|^{m_1}_\X+\sum_{i=0}^\Jott\EE\sup_{0\le t\le T}|v^i_\kappa(\xi)(t)|^p_{K_i}\le C_1,
$$
where $v^0_\kappa(\xi):=w_\kappa(\xi)$.
\item \label{metacontinuityw}
(Continuous paths.) For all $\xi\in\Xcal^\kappa_\Afrak(R_\kappa)$ we have $\PP$-a.s.\
$\bww_\kappa(\xi)\in C([0,T];U)\times C([0,T];U_1)\times\cdots\times C([0,T];U_\Jott)$. Moreover, $\PP$-a.s., $v^\jott_\kappa(\xi)\in\YY_\jott$ and $s\mapsto h_\jott(v^\jott_\kappa(\xi),s)$ is continuous on $[0,T]$, $\jott=1,\ldots,\Jott$. For $\YY_\jott=L^{m_\jott}(0,T;Y_\jott)$ with $m_\jott<\infty$ the continuity follows from $v^\jott_\kappa(\xi)\in\YY_\jott$; for $\YY_\jott=C([0,T];Y_\jott)$ it requires continuous paths in $Y_\jott$.
\item \label{metacausal}
(Causality.) For all $t\in[0,T]$ and $\xi_1,\xi_2\in\Xcal^\kappa_\Afrak(R_\kappa)$ with\footnote{ the term $\xi_1|_{[0,t]}=\xi_2|_{[0,t]}$ means equality
	$dt$-a.e.\ on $[0,t]$.} $\xi_1|_{[0,t]}=\xi_2|_{[0,t]}$
$\PP$-a.s., we have $\Vcal^\kappa_{\Afrak,\mathbb W}(\xi_1)|_{[0,t]}=\Vcal^\kappa_{\Afrak,\mathbb W}(\xi_2)|_{[0,t]}$ $\PP$-a.s.
\item \label{metaproj}
(Invariance under the projection.) For every $\ell\in\NN$, the shifted Haar projection $\Pro_\ell$ maps
$\Xcal^\kappa_\Afrak(R_\kappa)$ into itself. Here $(\Pro_\ell f)(s):=0$ for $s\in[0,T2^{-\ell})$ and
$(\Pro_\ell f)(s):=\frac{2^\ell}T\int_{T2^{-\ell}(k-1)}^{T2^{-\ell}k}f(r)\,dr$ for $s\in[T2^{-\ell}k,T2^{-\ell}(k+1))$,
$k=1,\ldots,2^\ell-1$; see (\ref{hatdefined}).
\item \label{metaident}
(The coefficients.) 
$\alpha'\theta<\min\{m_0,m_1\}$, with $\alpha$, $\theta$ and $\alpha'=\frac\alpha{\alpha-1}$ from
Assumption~\ref{ass:coeff}, and 
$$\EE\int_0^T|w_\kappa(\xi)(s)|^\alpha_{V_0}\,ds\le C_1
$$
 for $\xi\in\Xcal^\kappa_\Afrak(R_\kappa)$.
\item \label{initcond}
(Independence of the stochastic basis.) 
For every $\xi\in\Xcal^\kappa_\Afrak(R_\kappa)$ with $\Vcal^\kappa_{\Afrak,\mathbb W}(\xi)=\xi$, the law of
	$\bww_\kappa(\xi)(\tau^\ast_\kappa)$ belongs to $\mathcal I$, where $\tau^\ast_\kappa=\tau^\ast_\kappa(\bww_\kappa(\xi))$ is the stopping
	time in (\ref{defintau}). Since $\tau^\ast_\kappa$ is a measurable function of the path, this law depends only on the law of
	$\bww_\kappa(\xi)$, not on the stochastic basis.
	
\end{numref}

\medskip

Assume moreover that there are numbers $m_\jott>1$, $\jott=1,\ldots,\Jott$, and a constant $C(T)>0$ with the following
property. Every process $\bar\bww=(\bar w,\bar v^1,\ldots,\bar v^\Jott)$ on some stochastic basis which starts from $\bww_0$
in law and solves (\ref{spdeskappa}) with $\xi=\bar w$ and with the cut-off factors $\Xi^\jott_\kappa$ replaced by
progressively measurable processes with values in $[0,1]$ satisfies
\begin{equation}\label{metauniform}
\max_{\jott=1,\ldots,\Jott}\bar\EE\big[h_\jott(\bar v^\jott,T)^{m_\jott}\big]\le C(T).
\end{equation}
This class contains the glued processes constructed in the proof.

Then there exists a martingale solution of (\ref{spdes1}) on $[0,T]$ in the sense of Definition~\ref{Def:mart-sol}, with
initial law $\text{Law}(\bww_0)$.
\end{theorem}

\begin{remark}\label{rem:meta}
(i) The set $\Xcal^\kappa_\Afrak(R)$ is defined by the deterministic functional $\Phi_\kappa$. the criteria if $\xi\in \Xcal^\kappa_\Afrak(R)$ 
 depends only on
the law of $\xi$, and the assumptions can be used on the new stochastic basis constructed in Step~I of the proof.

(ii) Causality (\ref{metacausal}) follows from (\ref{metawellposed}) if the uniqueness holds on every interval $[0,t]$: the
cut-off factors at time $s$ depend only on the path up to time $s$, so the restrictions to $[0,t]$ of $\bww_\kappa(\xi_1)$ and
$\bww_\kappa(\xi_2)$ solve the same system on $[0,t]$.

(iii) 
(C0) consists of the conditions (H1) (hemicontinuity), (H2) (weak monotonicity, here without the noise) and
(H4) (growth) of \cite[Section 4.1]{weiroeckner}, for $\DeltaA_0$ alone; the coercivity (H3) is replaced by the bounds
(\ref{metacompact})--(\ref{metauniformintegrability}) and the bound in (\ref{metaident}). $F_0$ depends on the input $\xi$ and on
$v^1,\ldots,v^\Jott$, but not on $w$; in the monotonicity argument of Step~I it is a forcing term. For the porous medium
operator $\DeltaA_0w=r_n\Delta(|w|^{q-1}w)$ in the triple $L^{q+1}\subset H^{-1}_2\subset(L^{q+1})^\ast$, (C0) holds with
$\alpha=q+1$ and $\lambda_0=0$. For a linear $\DeltaA_i$, $i\ge1$, (C1) means that $\DeltaA_i$ is bounded from
$\X_i\cap C([0,T];K_i)$ into $L^1(0,T;D_i^\ast)$.

(iv) The bound (\ref{metauniform}) is usually obtained from an a priori estimate which uses only that the cut-off factors
take values in $[0,1]$. If there is a functional $\Lambda$ with $h_\jott(\bar v^\jott,T)^{m_\jott}\le\Lambda(\bar\bww)$,
$\jott=1,\ldots,\Jott$, and $\bar\EE\,\Lambda(\bar\bww)\le C(T)$ for all processes $\bar\bww$ admitted in (\ref{metauniform}),
then (\ref{metauniform}) holds. In Section~\ref{sec:application} this is Claim~\ref{cl1}.
\end{remark}

	\begin{remark}[The cut-off]\label{rem:cutoff}
		(i) Since $h_\jott(\eta,t)$ is the norm of the restriction of $\eta$ to $[0,t]$, the triangle inequality gives, for
		$\eta_1,\eta_2\in\YY_\jott$ and $t\in[0,T]$,
		$$
		|h_\jott(\eta_1,t)-h_\jott(\eta_2,t)|\le h_\jott(\eta_1-\eta_2,t)\le\|\eta_1-\eta_2\|_{\YY_\jott},\qquad
		|\Theta^\jott_\kappa(\eta_1,t)-\Theta^\jott_\kappa(\eta_2,t)|\le\frac{\|\phi'\|_\infty}{\kappa}\,h_\jott(\eta_1-\eta_2,t),
		$$
		and, since $0\le\Theta^i_\kappa\le1$, the difference of two products $\Xi^\jott_\kappa$ is bounded by the sum of the
		differences of the factors. In particular, (C3) holds if $\X_\jott\cap C([0,T];K_\jott)\hookrightarrow\YY_\jott$ continuously.
		
		(ii) (C3) is used only in Step~I, part (I.3) of the proof, pathwise on the new probability space. Convergence of the cut-off
		in the norm of $\CM^m_\Afrak(\X)$ or in probability is not needed.
		
		(iii) In the verification of (\ref{metacontinuity}) the cut-off enters only through a difference estimate. Let
		$\Theta_i=\Theta^\jott_\kappa(\eta_i,t)$ and let $G_i$ be the term multiplied by $\Theta_i$. If $\Theta_2\neq0$, write
		$\Theta_1G_1-\Theta_2G_2=\Theta_1(G_1-G_2)+(\Theta_1-\Theta_2)G_2$; if $\Theta_2=0$, then
		$\Theta_1G_1-\Theta_2G_2=(\Theta_1-\Theta_2)G_1$. In both cases only values with $h_\jott(\eta_i,t)<2\kappa$ enter, and
		the terms are controlled by (i) and a stability estimate for $\eta_1-\eta_2$ in $\YY_\jott(0,t)$. Thus $\YY_\jott$ has to
		be chosen such that $F_\jott$ is bounded on $\{h_\jott\le2\kappa\}$, the solutions are tight in a topology at least as
		strong as that of $\YY_\jott$, and a stability estimate in $\YY_\jott$ holds. In Section~\ref{sec:application},
		$\YY=C([0,T];\nsH)$ for $\bu$, the Lipschitz bound is (\ref{Thetadiff}) and the stability estimate is (\ref{Sdu}).
	\end{remark}

Now, we give a proof of Theorem \ref{meta}.
\begin{proof}
	As outlined before, the proof of Theorem \ref{meta} consists of several steps.	Part of the steps we have already introduced in the first paragraphs, as well as part of the ingredients, like the cut-off function $\phi$, the mappings $h$ and $\Theta$, the system \eqref{spdeskappa},  the operator $\mathcal{V}^\kappa_{\MA,W}$, and the corresponding bounded  sets   $\mathcal{X}_\kappa$ on which the operator $\mathcal{V}^\kappa_{\MA,W}$ acting on.
	\begin{step}

		\item
		In this step, we will show that for any $\kappa\in\NN$, there exists a fixed point and that the fixed point solves the system given in \eqref{spdeskappa} on a new probability space.
		In particular, for each $\kappa\in\NN$ one obtains a local martingale solution.
		\begin{claim}\label{schauder}
			For any $\kappa\in\NN$, for any  $\CF_0$ measurable initial conditions with {prescribed
				law} $\mu_0$, 
			there exists a probability space 
			$\tilde{\mathfrak{A}}_\kk =(\tilde \Omega_\kk ,\tilde{\CF}_\kk ,\tilde{\mathbb{F}}_\kk ,\tilde{\PP}_\kk ),
			$
			a vector-valued Wiener process
			$\tilde{\mathbb{W}}^\kk=(\tilde W_0^\kappa,\tilde W_1^\kappa,\ldots,\tilde W^\kappa_\Jott)
			$
			being defined on  $\tilde {\mathfrak{A}}_\kk $, and a tuple of random variable
			$$\tilde{\bww }_\kappa:=(\tilde w_{\kappa},\tilde v^1_\kappa,\ldots ,\tilde v^\Jott_\kappa)\in C(0,T;U)\times C(0,T;U_1)\times \cdots\times C(0,T;U_J), \ \tilde{\PP}_\kk-a.s.
			$$
			solving the system {(\ref{spdeskappa}) with $\xi=\tilde w_\kappa$}, i.e.\ the system (\ref{spdeskappaexact}),
			 $\tilde{\PP}_\kk$-a.s.
		\end{claim}

		\begin{proof}[Proof of Claim \ref{schauder}]

			\begin{stepproof}

\item  \label{proofsteppeins} In this step we will approximate the operator  $	\Vcal^\kappa_{\Afrak,\mathbb{W}}$ via a time discretisation. To this aim, in a similar way as in \cite{brz2013}, for $\kl\in \mathbb{N}$ we introduce a dyadic time grid $$\pi:=\lk\{\frac{iT}{2^{\kl}}; \ \ i=0,1,...,2^\kl-1\rk\},$$
and define  the shifted  Haar projection $\Pro_\kl $ of order $\ell$ by  for $f\in  \X$,
\begin{eqnarray}\label{hatdefined} (\Pro_\kl f)(s):=
\bcase 0 , &
\mbox{ if } s\in [0,T2^{-\kl }),
\\
\displaystyle
 \frac {2^\kl}{T} \int_{T2^{-\kl}(k-1)}^{T2^{-\kl}k} f(r)\: dr, &\mbox{ if
 } s\in [T{2}^{-\kl}k,T{2}^{-\kl}(k+1)), \\ &\qquad k=1,\ldots ,2^l-1.\ecase
\end{eqnarray}
Note, that if $f\in \CM_\MA^{m}(\X)$, then $\Pro_\kl f$ is a
progressively measurable, piecewise  constant,  $X$-valued
process.
\noindent On the first interval we set $\Pro_\ell f:=0$ instead of $f(0)$. 
With this choice, on every
interval $\Pro_\ell f$ is either $0$ or an average of $f$ over the preceding interval.

\noindent Hence, by Jensen's inequality, $\Phi(\Pro_\ell f)\le\Phi(f)$ for every functional of the form
$\Phi(f)=\int_0^T F(f(t))\,dt$ or $\Phi(f)=\sup_{t\le T}F(f(t))$ with $F$ convex, nonnegative and $F(0)=0$.
In particular $\|\Pro_\ell f\|_\X\le\|f\|_\X$, and assumption (\ref{metaproj}) holds for the sets used in
Section~\ref{sec:application}. Moreover, $(\Pro_\ell f)(s)=0$ for $s\in[0,T2^{-\ell})$, and for
$s\in[T2^{-\ell}k,T2^{-\ell}(k+1))$ the value $(\Pro_\ell f)(s)$ depends only on $f|_{[T2^{-\ell}(k-1),T2^{-\ell}k)}$.
This time delay makes the fixed-point problem in Lemma~\ref{lem:pathwisefix} solvable step by step.

\begin{remark}\label{projection}
We recall  in  \cite[Section B]{brz2013}, that for  all $\ell\in \mathbb{N}$,
	\begin{enumerate}[leftmargin=0cm,itemindent=.5cm,parsep=5pt,labelwidth=\itemindent,labelsep=0.0cm,align=left,topsep=0pt, label={{(\alph{*})  }}]
		\item $\Pro_
		\kl$ is a bounded linear operator from $\X$ into $\X$. In fact $$\|\Pro_
		\kl f\|_{\X}\le \|f\|_{\X},\quad \forall f\in \X.
		$$
		\item if $K$ is a compact subset of $\X$,
		then for all $\ep>0$ there exists a $\kl_\eps\in\NN$ such that
		$$\|\Pro_
		\kl  f - f\|_{\X}\le \ep,\quad f\in K,\quad \forall \kl\ge\kl_\eps.
		$$
		\item $\Pro_
		\kl$ {converges strongly to the identity} in  $\X$. $$	\Pro_
		\kl f \longrightarrow f \quad \text{as}\quad  \kl \longrightarrow  \infty\quad  \text{ in }\quad \X.$$
	\end{enumerate}
\end{remark}

\item 
For $\ell\in\NN$ we define the  operator 
$$
\Vcal^{\kappa,\ell}_{\Afrak,\mathbb W}(\xi):=\Vcal^\kappa_{\Afrak,\mathbb W}(\Pro_\ell\xi),\qquad \xi\in\Xcal^\kappa_\Afrak(R_\kappa),
$$
i.e.\ $\Vcal^{\kappa,\ell}_{\Afrak,\mathbb W}(\xi)$ is the first component of the solution of (\ref{spdeskappa}) with
$\xi$ replaced by $\Pro_\ell\xi$ in $F_0,\ldots,F_\Jott$. Since $\Pro_\ell$ maps $\Xcal^\kappa_\Afrak(R_\kappa)$ into itself
(assumption (\ref{metaproj})), $\Vcal^{\kappa,\ell}_{\Afrak,\mathbb W}$ is well defined and, by (\ref{metainvariant}), maps $\Xcal^\kappa_\Afrak(R_\kappa)$ into itself.

\noindent Recall the causality assumption (\ref{metacausal}): if $\xi_1=\xi_2$ on $[0,t]$ $\PP$-a.s., then
$\Vcal^\kappa_{\Afrak,\mathbb W}(\xi_1)=\Vcal^\kappa_{\Afrak,\mathbb W}(\xi_2)$ on $[0,t]$ $\PP$-a.s. It follows from
pathwise uniqueness for (\ref{spdeskappa}) on $[0,t]$.
\begin{lemma}\label{lem:pathwisefix}
Let assumptions (\ref{metawellposed}), (\ref{metainvariant}), (\ref{metacausal}) and (\ref{metaproj}) be in force. Then,
 for every $\ell\in\NN$ there is
$w^\ell_\kappa\in\Xcal^\kappa_\Afrak(R_\kappa)$, defined on the given basis $\Afrak$ with the given $\mathbb W$, such that
$w^\ell_\kappa=\Vcal^\kappa_{\Afrak,\mathbb W}(\Pro_\ell w^\ell_\kappa)$ $\PP$-a.s. Together with $v^{j,\ell}_\kappa:=v^j_\kappa(\Pro_\ell w^\ell_\kappa)$
it solves
\begin{align}\label{spdeskappaell}
dw^\ell_\kappa&=\big(\DeltaA_0w^\ell_\kappa+\Xi^1_\kappa\,F_0(\Pro_\ell w^\ell_\kappa,v^{1,\ell}_\kappa,\ldots,v^{\Jott,\ell}_\kappa)\big)dt+\Sigma_0(w^\ell_\kappa)\,dW_0,\\
dv^{\jott,\ell}_\kappa&=\big(\DeltaA_\jott v^{\jott,\ell}_\kappa+\Xi^\jott_\kappa\,F_\jott(\Pro_\ell w^\ell_\kappa,v^{\jott,\ell}_\kappa,\ldots,v^{\Jott,\ell}_\kappa)\big)dt
+\Sigma_\jott(v^{\jott,\ell}_\kappa)\,dW_\jott,\quad \jott=1,\ldots,\Jott,\notag
\end{align}
with $w^\ell_\kappa(0)=w_0$, $v^{\jott,\ell}_\kappa(0)=v^\jott_0$.
\end{lemma}
\begin{proof}
Let $t_k:=kT2^{-\ell}$, $k=0,\ldots,2^\ell$, take any $w^{(0)}\in\Xcal^\kappa_\Afrak(R_\kappa)$, and put
$w^{(k+1)}:=\Vcal^\kappa_{\Afrak,\mathbb W}(\Pro_\ell w^{(k)})$. By (\ref{metainvariant}) and (\ref{metaproj}), $w^{(k)}\in\Xcal^\kappa_\Afrak(R_\kappa)$ for
all $k$. By (\ref{hatdefined}), $(\Pro_\ell f)(s)=0$ for $s\in[0,t_1)$, and for $s\in[t_k,t_{k+1})$ it depends
only on $f|_{[t_{k-1},t_k)}$. Hence $f_1=f_2$ on $[0,t_{k-1}]$ implies $\Pro_\ell f_1=\Pro_\ell f_2$ on $[0,t_k)$.
We show $w^{(k+1)}=w^{(k)}$ on $[0,t_k]$ for $k=1,\ldots,2^\ell$. For $k=1$, $\Pro_\ell w^{(1)}=\Pro_\ell w^{(0)}=0$
on $[0,t_1)$, and (\ref{metacausal}) gives $w^{(2)}=w^{(1)}$ on $[0,t_1]$. If $w^{(k)}=w^{(k-1)}$ on $[0,t_{k-1}]$,
then $\Pro_\ell w^{(k)}=\Pro_\ell w^{(k-1)}$ on $[0,t_k)$, and (\ref{metacausal}) gives $w^{(k+1)}=w^{(k)}$ on $[0,t_k]$.
For $k=2^\ell$ we obtain $w^\ell_\kappa:=w^{(2^\ell)}=\Vcal^\kappa_{\Afrak,\mathbb W}(\Pro_\ell w^\ell_\kappa)$ on $[0,T]$.
\end{proof}

\begin{remark}\label{rem:newop}
 Since $\Pro_\ell w^\ell_\kappa\in\Xcal^\kappa_\Afrak(R_\kappa)$ and $w^\ell_\kappa=\Vcal^\kappa_{\Afrak,\mathbb W}(\Pro_\ell w^\ell_\kappa)$, assumptions (\ref{metacompact}) and (\ref{metauniformintegrability}) apply to $w^\ell_\kappa$ directly.
\end{remark}

\item \noindent\textbf{Uniform bounds and tightness.}
For $\ell\in\NN$ let $w^\ell_\kappa$ and $v^{\jott,\ell}_\kappa$, $\jott=1,\ldots,\Jott$, be given by Lemma~\ref{lem:pathwisefix}.
By Remark~\ref{rem:newop}, assumptions (\ref{metacompact}) and (\ref{metauniformintegrability}) give
\begin{equation}\label{unifboundsell}
\sup_{\ell\in\NN}\EE\|w^\ell_\kappa\|^{m_0}_{\X'}\le C_0,\qquad \sup_{\ell\in\NN}\EE\|{w^\ell_\kappa}\|^{m_1}_{\X}\le C_1 .
\end{equation}
By the Chebyshev inequality, $\PP(\|w^\ell_\kappa\|_{\X'}>r)\le C_0r^{-m_0}$, and the balls of $\X'$ are {relatively} compact in $\X$.
Hence the laws of $\{w^\ell_\kappa:\ell\in\NN\}$ are tight on $\X$. Together with the corresponding bounds for the other
components (hypothesis of Theorem~\ref{meta}), this is used in the Skorokhod step (I.1) of the identification of the limit.

\item \noindent\textbf{Asymptotic fixed points.}
The process $w^\ell_\kappa$ is a fixed point of $\Vcal^{\kappa,\ell}_{\Afrak,\mathbb W}=\Vcal^\kappa_{\Afrak,\mathbb W}\circ\Pro_\ell$, and not
of $\Vcal^\kappa_{\Afrak,\mathbb W}$.
We now use the continuity (\ref{metacontinuity}) and the uniform integrability
(\ref{metacompact})--(\ref{metauniformintegrability}) to show that this error term  vanishes as $\ell\to\infty$. For $\ell\in\NN$ put
$$
z^\ell_\kappa:=\Vcal^\kappa_{\Afrak,\mathbb W}(w^\ell_\kappa),\qquad y^{\jott,\ell}_\kappa:=v^\jott_\kappa(w^\ell_\kappa),\quad
\jott=1,\ldots,\Jott ,
$$
i.e.\ $(z^\ell_\kappa,y^{1,\ell}_\kappa,\ldots,y^{\Jott,\ell}_\kappa)=\bww_\kappa(w^\ell_\kappa)$ is the solution of (\ref{spdeskappa})
with $\xi=w^\ell_\kappa$ on $\Afrak$ with $\mathbb W$. It is well defined, since $w^\ell_\kappa\in\Xcal^\kappa_\Afrak(R_\kappa)$ by
Lemma~\ref{lem:pathwisefix}, and $z^\ell_\kappa\in\Xcal^\kappa_\Afrak(R_\kappa)$ by (\ref{metainvariant}). Thus
\begin{align}\label{spdeskappaexactell}
dz^\ell_\kappa&=\big(\DeltaA_0z^\ell_\kappa+\Xi^1_\kappa\,F_0(w^\ell_\kappa,y^{1,\ell}_\kappa,\ldots,y^{\Jott,\ell}_\kappa)\big)dt
+\Sigma_0(z^\ell_\kappa)\,dW_0,\\
dy^{\jott,\ell}_\kappa&=\big(\DeltaA_\jott y^{\jott,\ell}_\kappa+\Xi^\jott_\kappa\,F_\jott(w^\ell_\kappa,y^{\jott,\ell}_\kappa,\ldots,
y^{\Jott,\ell}_\kappa)\big)dt+\Sigma_\jott(y^{\jott,\ell}_\kappa)\,dW_\jott,\quad \jott=1,\ldots,\Jott,\notag
\end{align}
with $z^\ell_\kappa(0)=w_0$, $y^{\jott,\ell}_\kappa(0)=v^\jott_0$, where $\Xi^\jott_\kappa$ is computed from
$y^{\jott,\ell}_\kappa,\ldots,y^{\Jott,\ell}_\kappa$. This is the system (\ref{spdeskappaexact}), except that the first argument of
$F_0,\ldots,F_\Jott$ is $w^\ell_\kappa$ and not the solution $z^\ell_\kappa$ itself.

\begin{claim}\label{cl:asymfix}
Let the assumptions (\ref{metawellposed})--(\ref{metaproj}) hold. Then
$$
\lim_{\ell\to\infty}\EE\big\|w^\ell_\kappa-z^\ell_\kappa\big\|^m_\X=0 .
$$
\end{claim}
\begin{proof}
By Lemma~\ref{lem:pathwisefix}, $w^\ell_\kappa=\Vcal^\kappa_{\Afrak,\mathbb W}(\Pro_\ell w^\ell_\kappa)$, and
$\Pro_\ell w^\ell_\kappa,w^\ell_\kappa\in\Xcal^\kappa_\Afrak(R_\kappa)$ by (\ref{metaproj}). Hence (\ref{metacontinuity}) gives
\begin{equation}\label{asymfix1}
\EE\big\|w^\ell_\kappa-z^\ell_\kappa\big\|^m_\X
=\EE\big\|\Vcal^\kappa_{\Afrak,\mathbb W}(\Pro_\ell w^\ell_\kappa)-\Vcal^\kappa_{\Afrak,\mathbb W}(w^\ell_\kappa)\big\|^m_\X
\le\omega_\kappa(\eps_\ell),\qquad \eps_\ell:=\EE\big\|\Pro_\ell w^\ell_\kappa-w^\ell_\kappa\big\|^m_\X .
\end{equation}
It remains to show $\eps_\ell\to0$; here the uniform integrability is used. Let $r>0$ and
$K_r:=\{f\in\X:\|f\|_{\X'}\le r\}$. Since $\X'\hookrightarrow\X$ compactly, $K_r$ is relatively compact in $\X$, and
Remark~\ref{projection}-(b), applied to the closure of $K_r$, gives
$$
\delta_\ell(r):=\sup_{f\in K_r}\|\Pro_\ell f-f\|_\X\longrightarrow0,\qquad \ell\to\infty .
$$
On $\{\|w^\ell_\kappa\|_{\X'}>r\}$ we use $\|\Pro_\ell w^\ell_\kappa-w^\ell_\kappa\|_\X\le2\|w^\ell_\kappa\|_\X$
(Remark~\ref{projection}-(a)), the H\"older inequality with the exponents $m_1/m$ and $m_1/(m_1-m)$, the Chebyshev inequality
and (\ref{unifboundsell}):
\begin{align*}
\eps_\ell&\le\delta_\ell(r)^m+2^m\,\EE\Big[\mathds{1}_{\{\|w^\ell_\kappa\|_{\X'}>r\}}\|w^\ell_\kappa\|^m_\X\Big]
\le\delta_\ell(r)^m+2^m\,\PP\big(\|w^\ell_\kappa\|_{\X'}>r\big)^{1-\frac m{m_1}}
\big(\EE\|w^\ell_\kappa\|^{m_1}_\X\big)^{\frac m{m_1}}\\
&\le\delta_\ell(r)^m+2^m\big(C_0r^{-m_0}\big)^{1-\frac m{m_1}}C_1^{\frac m{m_1}} .
\end{align*}
Hence $\limsup_{\ell\to\infty}\eps_\ell\le2^m(C_0r^{-m_0})^{1-m/m_1}C_1^{m/m_1}$ for every $r>0$, and $r\to\infty$ gives
$\eps_\ell\to0$. Since $\omega_\kappa(x)\to0$ as $x\downarrow0$, the claim follows from (\ref{asymfix1}).
\end{proof}
The condition $m_1>m$ in (\ref{metauniformintegrability}) is essential here: the bound in $\X'$ controls
$\Pro_\ell w^\ell_\kappa-w^\ell_\kappa$ only on a set of large probability, and the higher moment controls the rest.
(\ref{metacontinuity}) is needed with a modulus, since both arguments $\Pro_\ell w^\ell_\kappa$ and $w^\ell_\kappa$ depend on $\ell$,
and $\{w^\ell_\kappa:\ell\in\NN\}$ is only tight; it need not converge on $\Afrak$.

\noindent\emph{Bounds for $\bww_\kappa(w^\ell_\kappa)$.} Since $w^\ell_\kappa\in\Xcal^\kappa_\Afrak(R_\kappa)$, the assumptions
(\ref{metacompact}), (\ref{metauniformintegrability}) and (\ref{metacontinuityw}) apply to
$\bww_\kappa(w^\ell_\kappa)=(z^\ell_\kappa,y^{1,\ell}_\kappa,\ldots,y^{\Jott,\ell}_\kappa)$. In particular, (\ref{unifboundsell}) holds with
$w^\ell_\kappa$ replaced by $z^\ell_\kappa$, and the laws of $z^\ell_\kappa$ and $y^{\jott,\ell}_\kappa$ are tight in
$\X\cap C([0,T];K_0)$ and in $\X_\jott\cap C([0,T];K_\jott)$, uniformly in $\ell$.

\item

\noindent\textbf{Identification of the limit.}
We pass to the limit in the equations (\ref{spdeskappaexactell}). The Wiener process is kept in the Skorokhod
representation, so that the approximations and the limit solve their equations on one probability space, and the
stochastic integrals pass to the limit by \cite[Lemma 2.1]{DGT2011}. The drift terms with $F_0,\ldots,F_\Jott$ and
$\DeltaA_1,\ldots,\DeltaA_\Jott$ converge strongly. $\DeltaA_0\hat v^0_j$ converges only weakly; its limit is identified
by the monotonicity argument of \cite[Section 4.2]{weiroeckner}. The input enters only through $F_0,\ldots,F_\Jott$, so in
this argument it is a given forcing term, and only the monotonicity of $\DeltaA_0$ is needed.

\smallskip
\noindent\emph{Notation.} 
\begin{equation}\label{driftdef}
	\begin{aligned}
		b_0(s;\xi,v)&:=\DeltaA_0v^0(s)+\Xi^1_\kappa(s)\,F_0\big(\xi(s),v^1(s),\ldots,v^\Jott(s)\big),\\
		b_i(s;\xi,v)&:=\DeltaA_iv^i(s)+\Xi^i_\kappa(s)\,F_i\big(\xi(s),v^i(s),\ldots,v^\Jott(s)\big),\qquad i=1,\ldots,\Jott,
	\end{aligned}
\end{equation}
where $\Xi^i_\kappa(s)$ is computed from $v|_{[0,s]}$. 

We distinguish the input and the solution: put $\xi_\ell:=w^\ell_\kappa$ and
$v_\ell=(v^0_\ell,\ldots,v^\Jott_\ell):=\bww_\kappa(w^\ell_\kappa)=(z^\ell_\kappa,y^{1,\ell}_\kappa,\ldots,y^{\Jott,\ell}_\kappa)$. Then
(\ref{spdeskappaexactell}) reads
$v^i_\ell(t)=v^i_0+\int_0^tb_i(s;\xi_\ell,v_\ell)\,ds+\int_0^t\Sigma_i(v^i_\ell(s))\,dW_i(s)$, $i=0,\ldots,\Jott$, and, by
Claim~\ref{cl:asymfix}, $\EE\|\xi_\ell-v^0_\ell\|^m_\X\to0$.

\smallskip

\noindent We use Assumption~\ref{ass:noise}, the bounds
(\ref{metacompact})--(\ref{metauniformintegrability}) and assumption (\ref{metaident}).

\smallskip

\noindent\emph{(I.1) Skorokhod representation with the Wiener process.} Put
$$
\mathcal Z_\ell:=(\bww_0,\xi_\ell,v_\ell,\mathbb W)=(\bww_0,w^\ell_\kappa,z^\ell_\kappa,y^{1,\ell}_\kappa,\ldots,y^{\Jott,\ell}_\kappa,\mathbb W).
$$
Here $\mathbb W$ is regarded as a process with paths in $C([0,T];H_1^{\Jott+1})$, where $H_1$ is the completion of $H$
with respect to $|h|^2_{H_1}:=\sum_{m\in\mathbb I}\lambda_m^2\langle h,\psi_m\rangle_H^2$ for some $\lambda_m>0$ with
$\sum_m\lambda_m^2<\infty$. Then $H\hookrightarrow H_1$ is Hilbert--Schmidt, and the Brownian motions
$\beta^k_m=\lambda_m^{-2}\langle W_k,\psi_m\rangle_{H_1}$ are continuous functions of $W_k$. The law of $\mathbb W$ does not
depend on $\ell$. By
(\ref{unifboundsell}) and the bounds for $\bww_\kappa(w^\ell_\kappa)$ above, the laws of $\{\mathcal Z_\ell:\ell\in\NN\}$ are
tight on
$$
\mathbf U\times\X\times\big(\X\cap C([0,T];K_0)\big)\times\prod_{i=1}^\Jott\big(\X_i\cap C([0,T];K_i)\big){}\times C([0,T];H_1^{\Jott+1}).
$$
By the Prohorov and Skorokhod theorems there are a subsequence $(\ell_j)$, a probability space
$(\hat\Omega,\hat\CF,\hat\PP)$ and random variables $\hat{\mathcal Z}_j=(\hat\bww^j_0,\hat\xi_j,\hat v_j,\hat{\mathbb W}_j)$,
$\hat{\mathcal Z}=(\hat\bww_0,\hat\xi,\hat v,\hat{\mathbb W})$, $\hat v_j=(\hat v^0_j,\ldots,\hat v^\Jott_j)$, $\hat v=(\hat v^0,\ldots,\hat v^\Jott)$,
with $\text{Law}(\hat{\mathcal Z}_j)=\text{Law}(\mathcal Z_{\ell_j})$ and $\hat{\mathcal Z}_j\to\hat{\mathcal Z}$ $\hat\PP$-a.s. By
(\ref{unifboundsell}) for $w^\ell_\kappa$ and for $z^\ell_\kappa$ and the equality of laws,
$\sup_j\hat\EE\|\hat\xi_j\|^{m_1}_\X\le C_1$ and $\sup_j\hat\EE\|\hat v^0_j\|^{m_1}_\X\le C_1$. Since $m_1>m$, the families
$\{\|\hat\xi_j-\hat\xi\|^r_\X:j\in\NN\}$ and $\{\|\hat v^0_j-\hat v^0\|^r_\X:j\in\NN\}$ are uniformly integrable for every
$r\in[1,m_1)$, and the Vitali convergence theorem gives
\begin{equation}\label{Lrconvnew}
\lim_{j\to\infty}\hat\EE\|\hat\xi_j-\hat\xi\|^r_\X=\lim_{j\to\infty}\hat\EE\|\hat v^0_j-\hat v^0\|^r_\X=0,\qquad r\in[1,m_1).
\end{equation}
\emph{In the limit, the input and the solution coincide.} Since $\hat\xi_j-\hat v^0_j$ has the law of
$w^{\ell_j}_\kappa-z^{\ell_j}_\kappa$, Claim~\ref{cl:asymfix} gives $\hat\EE\|\hat\xi_j-\hat v^0_j\|^m_\X\to0$. Together with
(\ref{Lrconvnew}) for $r=m$,
$$
\hat\EE\|\hat\xi-\hat v^0\|^m_\X\le3^{m-1}\Big(\hat\EE\|\hat\xi-\hat\xi_j\|^m_\X+\hat\EE\|\hat\xi_j-\hat v^0_j\|^m_\X
+\hat\EE\|\hat v^0_j-\hat v^0\|^m_\X\Big)\longrightarrow0,\qquad j\to\infty,
$$
hence
\begin{equation}\label{xiequalv0}
\hat\xi=\hat v^0=:\hat w\ \text{ in }\X\quad\hat\PP\text{-a.s., i.e.\ }\hat\xi(s)=\hat w(s)\text{ for a.e.\ }s\in[0,T].
\end{equation}
This is where the continuity (\ref{metacontinuity}) enters. Let $(\hat\CF_t)$ be the {right-continuous} augmented filtration generated by
$\hat\bww_0$, $\hat\xi|_{[0,t]}$, $\hat v|_{[0,t]}$ and $\hat{\mathbb W}|_{[0,t]}$, and $(\hat\CF^j_t)$ the one
generated by $\hat\bww^j_0$, $\hat\xi_j|_{[0,t]}$, $\hat v_j|_{[0,t]}$ and $\hat{\mathbb W}_j|_{[0,t]}$. The Wiener and martingale properties proved below for the filtrations generated by these processes pass to their right-continuous augmentations, since all processes have continuous paths.

\smallskip
\noindent\emph{(I.2) {The equations of the approximations on
$\hat\Omega$.}} Define on $\hat\Omega$
$$
\hat M^i_j(t):=\hat v^i_j(t)-\hat v^i_j(0)-\int_0^tb_i(s;\hat\xi_j,\hat v_j)\,ds,\qquad i=0,\ldots,\Jott ,
$$
and the same expression $M^i_\ell$ with $(\xi_\ell,v_\ell)$ on the original space. There, $M^i_\ell=\int_0^\cdot\Sigma_i(v^i_\ell)\,dW_i$,
and $W_0,\ldots,W_\Jott$ are independent. Hence, for $0\le s<t\le T$, $\varphi\in D_i$, $\psi\in D_k$, and every bounded
continuous function $\Psi$ of $\mathcal Z_\ell|_{[0,s]}$,
\begin{align}
&\EE\Big[\big\langle M^i_\ell(t)-M^i_\ell(s),\varphi\big\rangle\,\Psi\Big]=0,\label{eq:Mj1}\\
&\EE\Big[\Big(\langle M^i_\ell(t),\varphi\rangle\langle M^k_\ell(t),\psi\rangle-\langle M^i_\ell(s),\varphi\rangle\langle M^k_\ell(s),\psi\rangle
-\delta_{ik}\int_s^t\big\langle\Sigma_i(v^i_\ell(r))^\ast\varphi,\Sigma_i(v^i_\ell(r))^\ast\psi\big\rangle_H\,dr\Big)\Psi\Big]=0 .\label{eq:Mj2}
\end{align}
Moreover, with the Brownian motions $\beta^k_m$ of $W_k$,
\begin{equation}\label{eq:Mj3}
\EE\Big[\Big(\langle M^i_\ell(t),\varphi\rangle\beta^k_m(t)-\langle M^i_\ell(s),\varphi\rangle\beta^k_m(s)
-\delta_{ik}\int_s^t\big\langle\Sigma_i(v^i_\ell(r))\psi_m,\varphi\big\rangle\,dr\Big)\Psi\Big]=0,\qquad m\in\mathbb I,\
k=0,\ldots,\Jott .
\end{equation}
The expressions in (\ref{eq:Mj1})--(\ref{eq:Mj2}) and (\ref{eq:Mj3}) are measurable functions of
$\mathcal Z_\ell$ {(for $i=0$ on the Borel set where $\int_0^T|v^0(s)|^\alpha_{V_0}ds<\infty$, which carries the laws by
(\ref{metaident}))}. Since
$\text{Law}(\hat{\mathcal Z}_j)=\text{Law}(\mathcal Z_{\ell_j})$, they hold on $\hat\Omega$ with $M^i_\ell$, $v_\ell$, $\mathcal Z_\ell$ replaced
by $\hat M^i_j$, $\hat v_j$, $\hat{\mathcal Z}_j$. By the Burkholder--Davis--Gundy inequality on the original space, {Assumption~\ref{ass:noise}, (\ref{metauniformintegrability})} and
equality of laws, $\sup_j\hat\EE\sup_{t\le T}|\langle\hat M^i_j(t),\varphi\rangle|^p<\infty$ with $p>2$.
\noindent Since the increments of $\mathbb W$ after time $s$ are independent of $\mathcal Z_\ell|_{[0,s]}$, and this is a
property of the law of $\mathcal Z_\ell$, $\hat{\mathbb W}_j$ is an $(\hat\CF^j_t)$-Wiener process with the law of $\mathbb W$,
and by (\ref{eq:Mj1}) $\hat M^i_j$ is a continuous $(\hat\CF^j_t)$-martingale. Let $\hat W_{i,j}=\sum_m\psi_m\hat\beta^i_{m,j}$
be the $i$-th component of $\hat{\mathbb W}_j$, put $N^i_j:=\hat M^i_j-\int_0^\cdot\Sigma_i(\hat v^i_j(s))\,d\hat W_{i,j}(s)$,
and let $\varphi\in D_i$. By (\ref{eq:Mj2}), $\langle\hat M^i_j,\varphi\rangle$ has the quadratic variation
$\int_0^t|\Sigma_i(\hat v^i_j)^\ast\varphi|^2_H\,ds$. By (\ref{eq:Mj3}), its cross variation with $\hat\beta^i_{m,j}$ is
$\int_0^t\langle\Sigma_i(\hat v^i_j)\psi_m,\varphi\rangle\,ds$. Hence its cross variation with
$\langle\int_0^\cdot\Sigma_i(\hat v^i_j)\,d\hat W_{i,j},\varphi\rangle=\sum_m\int_0^\cdot\langle\Sigma_i(\hat v^i_j)\psi_m,\varphi\rangle\,
d\hat\beta^i_{m,j}$ is $\sum_m\int_0^t\langle\Sigma_i(\hat v^i_j)\psi_m,\varphi\rangle^2ds=\int_0^t|\Sigma_i(\hat v^i_j)^\ast\varphi|^2_H\,ds$,
and the stochastic integral has the same quadratic variation. So $\langle N^i_j,\varphi\rangle$ has quadratic variation $0$,
i.e.\ $\langle N^i_j,\varphi\rangle=0$. Taking $\varphi$ in a countable dense subset of $D_i$, $N^i_j=0$. Thus, on
$(\hat\Omega,\hat\CF,(\hat\CF^j_t),\hat\PP)$,
\begin{equation}\label{eq:hatj}
\hat v^i_j(t)=\hat v^i_j(0)+\int_0^tb_i(s;\hat\xi_j,\hat v_j)\,ds+\int_0^t\Sigma_i(\hat v^i_j(s))\,d\hat W_{i,j}(s),\qquad
t\in[0,T],\ i=0,\ldots,\Jott ,
\end{equation}
i.e.\ $(\hat v_j,\hat\xi_j)$ solves the system (\ref{spdeskappaexactell}) on $\hat\Omega$ with the Wiener process $\hat{\mathbb W}_j$.

\smallskip

\noindent\emph{(I.3) Passage to the limit in the equations.} \emph{The Wiener process.} Let $0\le s<t\le T$, let $\Psi$ be a
bounded continuous function of $\hat{\mathcal Z}_j|_{[0,s]}$ and $g$ a bounded continuous function on $H_1^{\Jott+1}$. By (I.2),
$\hat\EE[\Psi(\hat{\mathcal Z}_j|_{[0,s]})\,g(\hat{\mathbb W}_j(t)-\hat{\mathbb W}_j(s))]
=\hat\EE[\Psi(\hat{\mathcal Z}_j|_{[0,s]})]\,\hat\EE[g(\hat{\mathbb W}_j(t)-\hat{\mathbb W}_j(s))]$. Letting $j\to\infty$, the same
holds for $\hat{\mathcal Z}$ and $\hat{\mathbb W}$. Hence $\hat{\mathbb W}=(\hat W_0,\ldots,\hat W_\Jott)$ is an
$(\hat\CF_t)$-Wiener process with the law of $\mathbb W$, i.e.\ $\hat W_0,\ldots,\hat W_\Jott$ are independent cylindrical
Wiener processes on $H$ with respect to $(\hat\CF_t)$.

\emph{The stochastic integrals.} By (I.1), $\hat v^i_j\to\hat v^i$ in $C([0,T];K_i)$ $\hat\PP$-a.s. By
Assumption~\ref{ass:noise} and dominated convergence, $\Sigma_i(\hat v^i_j)\to\Sigma_i(\hat v^i)$ in $L^2(0,T;L_{HS}(H;K_i))$
$\hat\PP$-a.s. Since moreover $\hat{\mathbb W}_j\to\hat{\mathbb W}$ in $C([0,T];H_1^{\Jott+1})$ $\hat\PP$-a.s.,
\cite[Lemma 2.1]{DGT2011} gives
\begin{equation}\label{stochconv}
\int_0^\cdot\Sigma_i(\hat v^i_j(s))\,d\hat W_{i,j}(s)\longrightarrow\int_0^\cdot\Sigma_i(\hat v^i(s))\,d\hat W_i(s)\quad
\text{in }L^2(0,T;K_i)\text{ in probability},
\end{equation}
$i=0,\ldots,\Jott$. By the Burkholder--Davis--Gundy inequality, Assumption~\ref{ass:noise} and (\ref{metauniformintegrability}), these
integrals are bounded in $L^p(\hat\Omega;C([0,T];K_i))$ with $p>2$; hence (\ref{stochconv}) also holds in
$L^2(\hat\Omega\times[0,T];K_i)$.

\emph{The components $i=1,\ldots,\Jott$.} By (\ref{metaident}),
\begin{equation}\label{driftconv}
\int_0^t\langle b_i(s;\hat\xi_j,\hat v_j),\varphi\rangle\,ds\longrightarrow\int_0^t\langle b_i(s;\hat\xi,\hat v),\varphi\rangle\,ds
\quad\text{in }L^1(\hat\Omega),\qquad\varphi\in D_i,\ t\in[0,T],\ i=1,\ldots,\Jott .
\end{equation}
Indeed, fix $i\ge1$, $t$ and $\varphi\in D_i$, and let $\hat\Xi_j(s)$ and $\hat\Xi(s)$ be the factor $\Xi^i_\kappa(s)$ computed
from $\hat v_j$ and from $\hat v$. By (I.1), $\hat\PP$-a.s., $\hat\xi_j\to\hat\xi$ in $\X$ and $\hat v^\jott_j\to\hat v^\jott$ in
$\X_\jott\cap C([0,T];K_\jott)$. Hence, $\hat\PP$-a.s., (C1) and (C2) give
$a_j:=\langle\DeltaA_i\hat v^i_j(\cdot),\varphi\rangle\to a:=\langle\DeltaA_i\hat v^i(\cdot),\varphi\rangle$ and
$g_j:=\langle F_i(\hat\xi_j(\cdot),\ldots),\varphi\rangle\to g:=\langle F_i(\hat\xi(\cdot),\ldots),\varphi\rangle$ in $L^1(0,T)$. By
(C3) and the continuity of $\phi_\kappa$, $\hat\Xi_j(s)\to\hat\Xi(s)$ for every $s\in[0,T]$. Since $0\le\hat\Xi_j\le1$,
$$
\int_0^T\big|\hat\Xi_j(s)g_j(s)-\hat\Xi(s)g(s)\big|\,ds\le\|g_j-g\|_{L^1(0,T)}+\int_0^T|\hat\Xi_j(s)-\hat\Xi(s)|\,|g(s)|\,ds\longrightarrow0
$$
by dominated convergence. Thus $\langle b_i(\cdot;\hat\xi_j,\hat v_j),\varphi\rangle=a_j+\hat\Xi_jg_j\to a+\hat\Xi g
=\langle b_i(\cdot;\hat\xi,\hat v),\varphi\rangle$ in $L^1(0,T)$ $\hat\PP$-a.s. By the growth bounds in (C1) and (C2),
$$
\Big|\int_0^t\langle b_i(s;\hat\xi_j,\hat v_j),\varphi\rangle\,ds\Big|\le2C_\varphi\Big(1+\|\hat\xi_j\|^\theta_\X
+\sum_{\jott=1}^\Jott\|\hat v^\jott_j\|^\theta_{\X_\jott}\Big).
$$
By (I.1), $\hat\EE\|\hat\xi_j\|^{m_1}_\X$ is bounded in $j$. By (\ref{metacompact}), the continuous embedding
$\X'_\jott\hookrightarrow\X_\jott$ and the equality of laws, $\hat\EE\|\hat v^\jott_j\|^{m_0}_{\X_\jott}$ is bounded in $j$. Since
$\theta<\alpha'\theta<\min\{m_0,m_1\}$ by (\ref{metaident}), the right-hand side is bounded in $L^r(\hat\Omega)$ with $r:=\min\{m_0,m_1\}/\theta>1$, so the
integrals are uniformly integrable, and the Vitali theorem gives (\ref{driftconv}). With (\ref{stochconv}), the limit in
(\ref{eq:hatj}), tested with $\varphi\in D_i$, gives for a.e.\ $t$, and by continuity for all $t$,
\begin{equation}\label{eq:limiti}
\hat v^i(t)=\hat v^i(0)+\int_0^tb_i(s;\hat\xi,\hat v)\,ds+\int_0^t\Sigma_i(\hat v^i(s))\,d\hat W_i(s),\qquad t\in[0,T],\
i=1,\ldots,\Jott .
\end{equation}

\emph{The component $i=0$.} Let $\hat\Xi_j$ and $\hat\Xi$ be the factor $\Xi^1_\kappa$ computed from $\hat v_j$ and from $\hat v$,
and put $\hat F_j(s):=\hat\Xi_j(s)F_0(\hat\xi_j(s),\hat v^1_j(s),\ldots,\hat v^\Jott_j(s))$ and
$\hat F(s):=\hat\Xi(s)F_0(\hat\xi(s),\hat v^1(s),\ldots,\hat v^\Jott(s))$. As in the proof of (\ref{driftconv}), (C2), (C3)
and dominated convergence give $\hat F_j\to\hat F$ in $L^{\alpha'}(0,T;V_0^\ast)$ $\hat\PP$-a.s. By the growth bound in (C2),
the moment bounds above and $\alpha'\theta<\min\{m_0,m_1\}$, the family $\{\|\hat F_j\|^{\alpha'}_{L^{\alpha'}(0,T;V_0^\ast)}\}$ is
uniformly integrable, and
\begin{equation}\label{Fconv}
\hat F_j\longrightarrow\hat F\quad\text{in }L^{\alpha'}(\hat\Omega\times[0,T];V_0^\ast).
\end{equation}
Since $\text{Law}(\hat v^0_j)=\text{Law}(z^{\ell_j}_\kappa)$ and $z^{\ell_j}_\kappa=w_\kappa(w^{\ell_j}_\kappa)$ with
$w^{\ell_j}_\kappa\in\Xcal^\kappa_\Afrak(R_\kappa)$, (\ref{metaident}) gives $\hat\EE\int_0^T|\hat v^0_j(s)|^\alpha_{V_0}ds\le C_1$, and by the
growth bound in (C0), $\DeltaA_0\hat v^0_j$ is bounded in $L^{\alpha'}(\hat\Omega\times[0,T];V_0^\ast)$. These spaces are
reflexive. Hence, for a further subsequence,
\begin{equation}\label{weakconv}
\hat v^0_j\rightharpoonup\hat v^0\ \text{ in }L^\alpha(\hat\Omega\times[0,T];V_0),\qquad
\DeltaA_0\hat v^0_j\rightharpoonup\chi\ \text{ in }L^{\alpha'}(\hat\Omega\times[0,T];V_0^\ast).
\end{equation}
The weak limit of $\hat v^0_j$ is $\hat v^0$, since $\hat v^0_j\to\hat v^0$ in $L^2(\hat\Omega\times[0,T];K_0)$ by (I.1),
(\ref{metauniformintegrability}) and the Vitali theorem. Let $\varphi\in V_0$. By (\ref{eq:hatj}),
$$
\int_0^t\langle\DeltaA_0\hat v^0_j(s),\varphi\rangle\,ds=\langle\hat v^0_j(t)-\hat v^0_j(0),\varphi\rangle-\int_0^t\langle\hat F_j(s),\varphi\rangle\,ds
-\Big\langle\int_0^t\Sigma_0(\hat v^0_j(s))\,d\hat W_{0,j}(s),\varphi\Big\rangle .
$$
The right-hand side converges in $L^1(\hat\Omega\times[0,T])$, and the left-hand side converges weakly in
$L^{\alpha'}(\hat\Omega\times[0,T])$ to $\int_0^t\langle\chi(s),\varphi\rangle\,ds$. Hence, $\hat\PP\otimes dt$-a.e.\ and, by
continuity, for all $t$,
\begin{equation}\label{eq:limit0}
\hat v^0(t)=\hat v^0(0)+\int_0^t\big(\chi(s)+\hat F(s)\big)\,ds+\int_0^t\Sigma_0(\hat v^0(s))\,d\hat W_0(s)\quad\text{in }V_0^\ast,
\qquad t\in[0,T].
\end{equation}
In particular, $t\mapsto\int_0^t\chi(s)\,ds$ is $(\hat\CF_t)$-adapted and continuous in $V_0^\ast$, so $\chi$ has an
$(\hat\CF_t)$-progressively measurable version.

\smallskip
\noindent\emph{(I.4) Identification of $\chi$ by monotonicity.} We apply the It\^o formula for $|\cdot|^2_{K_0}$ in the Gelfand
triple $V_0\hookrightarrow K_0\hookrightarrow V_0^\ast$ \cite[Theorem 4.2.5]{weiroeckner} to (\ref{eq:hatj}) with $i=0$ on
$(\hat\Omega,\hat\CF,(\hat\CF^j_t),\hat\PP)$ and to (\ref{eq:limit0}) on $(\hat\Omega,\hat\CF,(\hat\CF_t),\hat\PP)$. Its
assumptions hold by (\ref{weakconv}), (\ref{Fconv}), Assumption~\ref{ass:noise} and (\ref{metauniformintegrability}). The
stochastic integrals are martingales by the Burkholder--Davis--Gundy inequality, and taking expectations gives
\begin{align*}
\hat\EE|\hat v^0_j(T)|^2_{K_0}&=\hat\EE|\hat v^0_j(0)|^2_{K_0}+\hat\EE\int_0^T\Big(2\langle\DeltaA_0\hat v^0_j+\hat F_j,\hat v^0_j\rangle
+|\Sigma_0(\hat v^0_j)|^2_{L_{HS}(H;K_0)}\Big)ds,\\
\hat\EE|\hat v^0(T)|^2_{K_0}&=\hat\EE|\hat v^0(0)|^2_{K_0}+\hat\EE\int_0^T\Big(2\langle\chi+\hat F,\hat v^0\rangle
+|\Sigma_0(\hat v^0)|^2_{L_{HS}(H;K_0)}\Big)ds.
\end{align*}
By (I.1), (\ref{metauniformintegrability}), Assumption~\ref{ass:noise} and the Vitali theorem, $\hat\EE|\hat v^0_j(T)|^2_{K_0}$,
$\hat\EE|\hat v^0_j(0)|^2_{K_0}$ and $\hat\EE\int_0^T|\Sigma_0(\hat v^0_j)|^2_{L_{HS}(H;K_0)}ds$ converge to the corresponding
expressions for $\hat v^0$. By (\ref{Fconv}) and (\ref{weakconv}) (strong times weak convergence),
$\hat\EE\int_0^T\langle\hat F_j,\hat v^0_j\rangle\,ds\to\hat\EE\int_0^T\langle\hat F,\hat v^0\rangle\,ds$. Comparing the two
identities,
\begin{equation}\label{limsupA}
\lim_{j\to\infty}\hat\EE\int_0^T\langle\DeltaA_0\hat v^0_j,\hat v^0_j\rangle\,ds=\hat\EE\int_0^T\langle\chi,\hat v^0\rangle\,ds .
\end{equation}
Let $y\in L^\infty(\hat\Omega\times[0,T];V_0)$, $\eps\in(0,1]$ and $u:=\hat v^0-\eps y$. By the monotonicity in (C0),
$$
\hat\EE\int_0^T\langle\DeltaA_0\hat v^0_j-\DeltaA_0u,\hat v^0_j-u\rangle\,ds\le\lambda_0\,\hat\EE\int_0^T|\hat v^0_j-u|^2_{K_0}\,ds .
$$
By (\ref{limsupA}), (\ref{weakconv}) and $\hat v^0_j\to\hat v^0$ in $L^2(\hat\Omega\times[0,T];K_0)$, the limit $j\to\infty$
gives $\hat\EE\int_0^T\langle\chi-\DeltaA_0u,\hat v^0-u\rangle\,ds\le\lambda_0\eps^2\,\hat\EE\int_0^T|y|^2_{K_0}\,ds$. Since
$\hat v^0-u=\eps y$,
$$
\hat\EE\int_0^T\big\langle\chi-\DeltaA_0(\hat v^0-\eps y),y\big\rangle\,ds\le\lambda_0\,\eps\,\hat\EE\int_0^T|y|^2_{K_0}\,ds .
$$
As $\eps\downarrow0$, the hemicontinuity and the growth bound in (C0) (dominated convergence) give
$\hat\EE\int_0^T\langle\chi-\DeltaA_0\hat v^0,y\rangle\,ds\le0$. Replacing $y$ by $-y$, equality holds for all
$y\in L^\infty(\hat\Omega\times[0,T];V_0)$, hence $\chi=\DeltaA_0\hat v^0$ $\hat\PP\otimes dt$-a.e. This is the monotonicity
argument of \cite[Section 4.2]{weiroeckner}. Only the monotonicity of $\DeltaA_0$ is needed: the input $\hat\xi$ enters
(\ref{eq:limit0}) only through the forcing term $\hat F$, which converges strongly by (\ref{Fconv}); its integrability comes
from the bounds on $\Xcal^\kappa_\Afrak(R_\kappa)$ and the cut-off. No weight $e^{-ct}$ is needed, because the convergence
$\hat v^0_j\to\hat v^0$ in $L^2(\hat\Omega\times[0,T];K_0)$ is strong.

\smallskip
\noindent\emph{(I.5) Conclusion.} Put
$\tilde\Afrak:=(\hat\Omega,\hat\CF,(\hat\CF_t),\hat\PP)$ and $\tilde{\mathbb W}:=\hat{\mathbb W}$. By (\ref{eq:limiti}),
(\ref{eq:limit0}) and (I.4),
$$
\hat v^i(t)=\hat v^i(0)+\int_0^tb_i(s;\hat\xi,\hat v)\,ds+\int_0^t\Sigma_i(\hat v^i(s))\,d\tilde W_i(s),\qquad t\in[0,T],\ i=0,\ldots,\Jott.
$$
By (\ref{xiequalv0}), $\hat\xi(s)=\hat w(s)$ for a.e.\ $s\in[0,T]$. Since the input enters $b_i(s;\xi,v)$ only through
$\xi(s)$, we have $\int_0^tb_i(s;\hat\xi,\hat v)\,ds=\int_0^tb_i(s;\hat w,\hat v)\,ds$ for all $t$. Hence, with $\hat v^0=\hat w$,
$(\hat w,\hat v^1,\ldots,\hat v^\Jott)$ solves (\ref{spdeskappa}) with $\xi=\hat w$, i.e.\ the system
(\ref{spdeskappaexact}) with $w_\kappa=\hat w$, on $\tilde\Afrak$ with the Wiener process
$\tilde{\mathbb W}$, and $\hat\bww_0$ has the law of $\bww_0$.  Moreover $\hat w$ is
$(\tilde\CF_t)$-progressive and $\hat\EE\Phi_\kappa(\hat w)\le\liminf_j\hat\EE\Phi_\kappa(\hat v^0_j)\le R_\kappa$ by lower semicontinuity
and Fatou {(recall $\text{Law}(\hat v^0_j)=\text{Law}(z^{\ell_j}_\kappa)$ and $z^{\ell_j}_\kappa\in\Xcal^\kappa_\Afrak(R_\kappa)$)}, so $\hat w\in\Xcal^\kappa_{\tilde\Afrak}(R_\kappa)$. By the uniqueness in (\ref{metawellposed}),
$\hat w=\Vcal^\kappa_{\tilde\Afrak,\tilde{\mathbb W}}(\hat w)$ and $\hat v^i=v^i_\kappa(\hat w)$. Hence, by (\ref{metacontinuityw}) on $\tilde\Afrak$, $(\hat w,\hat v^1,\ldots,\hat v^\Jott)\in C([0,T];U)\times C([0,T];U_1)\times\cdots\times C([0,T];U_\Jott)$ $\hat\PP$-a.s.; the Skorokhod part (I.1) gives continuity only in $K_0\times K_1\times\cdots\times K_\Jott$. Since $\hat v^i(t)$, $\hat v^i(0)$ and the stochastic integrals are $K_i$-valued, the equations, obtained in $V_0^\ast$ and in $D_i^\ast$, hold in $K_i$, as required in Definition~\ref{Def:mart-sol}. This proves Claim~\ref{schauder} with
$\tilde{\mathfrak A}_\kappa:=\tilde\Afrak$ and $\tilde{\mathbb W}^\kappa:=\tilde{\mathbb W}$. 
The continuity (\ref{metacontinuity}) is used, through Claim~\ref{cl:asymfix}, in (\ref{xiequalv0}): it
identifies the input $\hat\xi$ with the solution $\hat w$. The passage to the limit in the drift and in the noise uses
(\ref{metaident}), Assumption~\ref{ass:noise} and \cite[Lemma 2.1]{DGT2011}; the operator $\DeltaA_0$ is identified by
monotonicity in (I.4).

			\end{stepproof}
		\end{proof}

		\item \label{proofsteppacht}

\noindent\textbf{Gluing of the local solutions.}
In this step we construct a martingale solution of (\ref{spdes1}) on $[0,T]$, i.e.\ a filtered probability space
$\overline{\mathfrak A}$, a Wiener process $\bar{\mathbb W}=(\bar W_0,\ldots,\bar W_\Jott)$ over $\overline{\mathfrak A}$ and a process
$\bar\bww=(\bar w,\bar v^1,\ldots,\bar v^\Jott)$ solving (\ref{spdes1}) over $\overline{\mathfrak A}$. Pathwise uniqueness for
(\ref{spdeskappa}) with $\xi=w_\kappa$ is not available, so the local solutions for different $\kappa$ cannot be compared
on one probability space. We therefore construct the solution by successive extension: we follow the solution
with cut-off level $1$ up to its stopping time, continue with a solution with cut-off level $2$ started from the
state reached at that time, and so on. For each $\kappa$ this gives a process $\bar\bww_\kappa$ on $[0,T]$ and a stopping
time $\bar\tau_\kappa$ such that $\bar\tau_\kappa\le\bar\tau_{\kappa+1}$ and $\bar\bww_{\kappa+1}=\bar\bww_\kappa$ on $[0,\bar\tau_\kappa]$.

\smallskip
\noindent\emph{Local solutions and their stopping times.} Put $\mathbf U:=U\times U_1\times\cdots\times U_\Jott$, and let
$H_1\supset H$ be a Hilbert space such that $H\hookrightarrow H_1$ is Hilbert--Schmidt. Let $\kappa\in\NN$ and let {$\mu\in\mathcal I$ be an admissible initial law, see Theorem~\ref{meta}}. By Step~I (Claim~\ref{schauder}) there is a martingale solution
$(\mathfrak A_\kappa,\mathbb W^\kappa,\bww_\kappa)$, $\bww_\kappa=(w_\kappa,v^1_\kappa,\ldots,v^\Jott_\kappa)$, of (\ref{spdeskappa}) with
$\xi=w_\kappa$ and $\text{Law}(\bww_\kappa(0))=\mu$. Only its law matters, so we may take it on the canonical space:
\begin{itemize}
\item $\Omega_\kappa:=\mathbf U\times C([0,T];H_1^{\Jott+1})\times C([0,T];\mathbf U)$ with its Borel $\sigma$-algebra $\CF_\kappa$;
\item the coordinate process is $(\bww_\kappa(0),\mathbb W^\kappa,\bww_\kappa)$, and $\PP_\kappa$ is its law;
\item $\CF^\kappa_t$ is the {right-continuous} augmented filtration generated by the coordinates up to time $t$.
\end{itemize}
Let
$$
\tau^\ast_\kappa(\bww_\kappa):=\inf\Big\{s\ge0:\ \max_{1\le\jott\le\Jott}h_\jott(v^\jott_\kappa,s)\ge\kappa\Big\}\wedge T .
$$
Since $s\mapsto h_\jott(v^\jott_\kappa,s)$ is continuous (by (\ref{metacontinuityw}), since $w_\kappa\in\Xcal^\kappa_{\Afrak_\kappa}(R_\kappa)$ and $\bww_\kappa=\bww_\kappa(w_\kappa)$ by the end of the proof of Claim~\ref{schauder}), nondecreasing and adapted, $\tau^\ast_\kappa(\bww_\kappa)$ is an
$(\CF^\kappa_t)$-stopping time. On $[0,\tau^\ast_\kappa(\bww_\kappa))$ all cut-off factors $\Xi^\jott_\kappa$ are equal to $1$, so
$\bww_\kappa$ solves (\ref{spdes1}) there; on $[0,T]$ it solves (\ref{spdes1}) with $F_0,\ldots,F_\Jott$ multiplied by the
cut-off factors, which take values in $[0,1]$. Since $\Omega_\kappa$ is Polish, there is a regular conditional distribution
$$
Q_\kappa(x,\cdot):=\PP_\kappa\big(\,\cdot\,\big|\,\bww_\kappa(0)=x\big),\qquad x\in\mathbf U ,
$$
i.e.\ a probability kernel from $\mathbf U$ to $\Omega_\kappa$, with $\int_{\mathbf U}Q_\kappa(x,\cdot)\,\mu(dx)=\PP_\kappa$ and
$Q_\kappa(x,\{\bww_\kappa(0)=x\})=1$ for $\mu$-a.e.\ $x$. The initial value $\bww_\kappa(0)$ is $\CF^\kappa_0$-measurable. Hence, for
$\mu$-a.e.\ $x$, the process $\mathbb W^\kappa$ is again an $(\CF^\kappa_t)$-Wiener process under $Q_\kappa(x,\cdot)$, and $\bww_\kappa$
solves (\ref{spdeskappa}) with $\xi=w_\kappa$ and initial value $x$ under $Q_\kappa(x,\cdot)$; see
\cite[Section~6.1]{StroockVaradhan}.

\smallskip
\noindent\emph{The first piece ($\kappa=1$).} Take $\mu:=\text{Law}(\bww_0)$ and let $(\Omega_1,\CF_1,(\CF^1_t),\PP_1)$, $\mathbb W^1$,
$\bww_1$ be as above. Put
$$
\overline{\mathfrak A}_1:=(\overline\Omega_1,\overline\CF_1,(\overline\CF^1_t),\overline\PP_1):=(\Omega_1,\CF_1,(\CF^1_t),\PP_1),\qquad
\bar{\mathbb W}^1:=\mathbb W^1,\qquad \bar\bww_1:=\bww_1\ \text{ on }[0,T],\qquad \bar\tau_1:=\tau^\ast_1(\bww_1).
$$
Thus $\bar\bww_1$ solves (\ref{spdes1}) on $[0,\bar\tau_1)$, and on $[0,T]$ it solves (\ref{spdes1}) with $F_0,\ldots,F_\Jott$
multiplied by the cut-off factors $\Xi^1_1,\ldots,\Xi^\Jott_1\in[0,1]$. No continuation by an auxiliary linear equation is
needed: the first piece is continued by $\bww_1$ itself.

\smallskip
\noindent\emph{The second piece ($\kappa=2$).} Let $\mathfrak x_1:=\bar\bww_1(\bar\tau_1)=\bww_1(\tau^\ast_1(\bww_1))$, since $\bar\tau_1=\tau^\ast_1(\bww_1)$. It is an $\overline\CF^1_{\bar\tau_1}$-measurable
$\mathbf U$-valued random variable; let $\mu_1$ be its law. By {assumption (\ref{initcond}) of} Theorem~\ref{meta}, $\mu_1\in\mathcal I$. The construction above with $\kappa=2$ and $\mu=\mu_1$ gives the canonical space
$(\Omega_2,\CF_2,(\CF^2_t),\PP_2)$, the Wiener process $\mathbb W^2$, the solution $\bww_2$ with cut-off level $2$, and the kernel
$Q:=Q_2$. We now put the two pieces on one probability space, in such a way that the second piece starts from
$\mathfrak x_1$ itself (and not from an independent copy with the same law).

\smallskip
\noindent\emph{Probability space.} Let
$\overline\Omega_2:=\overline\Omega_1\times\Omega_2$, $\overline\CF_2:=\overline\CF_1\otimes\CF_2$, and
\begin{equation}\label{kernelP2}
\overline\PP_2(d\omega_1,d\omega_2):=\overline\PP_1(d\omega_1)\,Q\big(\mathfrak x_1(\omega_1),d\omega_2\big).
\end{equation}
The marginals of $\overline\PP_2$ are $\overline\PP_1$ and $\int Q(x,\cdot)\,\mu_1(dx)=\PP_2$. Moreover,
\begin{equation}\label{matching}
\bww_2(0)(\omega_2)=\mathfrak x_1(\omega_1)=\bww_1(\bar\tau_1)(\omega_1)\qquad\text{for }\overline\PP_2\text{-a.e.\ }(\omega_1,\omega_2),
\end{equation}
i.e.\ the second piece starts from the state of the first piece at time $\bar\tau_1$. For $\mu_1$-a.e.\ $x$, the law of
$\mathbb W^2$ under $Q(x,\cdot)$ is the Wiener measure. Hence $\mathbb W^2$ is independent of $\overline\CF_1$ under $\overline\PP_2$.
With the product measure $\overline\PP_1\otimes\PP_2$ instead of (\ref{kernelP2}), $\bww_2(0)$ would be an independent copy of
$\bww_1(\bar\tau_1)$, and (\ref{matching}) would fail.

\smallskip
\noindent\emph{Filtration.} For $t\in[0,T]$ let $\overline\CF^2_t$ be the $\overline\PP_2$-augmentation of
$$
\big(\overline\CF^1_t\otimes\{\emptyset,\Omega_2\}\big)\vee\sigma\Big(\mathds 1_{\{\bar\tau_1\le t\}}\,\bww_2\big(r\wedge(t-\bar\tau_1)\big),\;
\mathds 1_{\{\bar\tau_1\le t\}}\,\mathbb W^2\big(r\wedge(t-\bar\tau_1)\big)\;:\;r\in[0,T]\Big),
$$
made right-continuous, and put $\overline{\mathbb F}_2:=(\overline\CF^2_t)_{t\in[0,T]}$.
\begin{itemize}
\item $\bar\tau_1$ is an $(\overline\CF^1_t)$-stopping time, so $\{\bar\tau_1\le t\}$ and $(t-\bar\tau_1)\mathds 1_{\{\bar\tau_1\le t\}}$ are
$\overline\CF^1_t$-measurable. Hence $\overline\CF^2_t$ is a $\sigma$-algebra on $\overline\Omega_2$.
\item $\overline\CF^2_s\subset\overline\CF^2_t$ for $s\le t$.
\item $\overline\CF^2_t$ contains the information of the first piece up to time $t$ and, on $\{\bar\tau_1\le t\}$, the
information of the second piece up to the shifted time $t-\bar\tau_1$.
\end{itemize}
This is the precise form of ``$\overline\CF^2_t=\overline\CF^1_t$ for $t<\bar\tau_1$ and $\overline\CF^2_t=\CF^2_{t-\bar\tau_1}$ for
$t\ge\bar\tau_1$''. Let $\overline{\mathfrak A}_2:=(\overline\Omega_2,\overline\CF_2,\overline{\mathbb F}_2,\overline\PP_2)$.

\smallskip
\noindent\emph{Wiener process.} For $\jott=0,\ldots,\Jott$ and $t\in[0,T]$ put
$$
\bar W^2_\jott(t):=\bar W^1_\jott(t\wedge\bar\tau_1)+W^2_\jott\big((t-\bar\tau_1)^+\big),
$$
i.e.\ $\bar W^2_\jott=\bar W^1_\jott$ on $[0,\bar\tau_1]$, and $\bar W^2_\jott(t)=\bar W^1_\jott(\bar\tau_1)+W^2_\jott(t-\bar\tau_1)$ for
$t\ge\bar\tau_1$.

\begin{lemma}\label{lem:concatW}
$\bar{\mathbb W}^2:=(\bar W^2_0,\ldots,\bar W^2_\Jott)$ is a vector of independent cylindrical Wiener processes on $H$ over
$\overline{\mathfrak A}_2$.
\end{lemma}
\begin{proof}
Fix $\jott$ and $h\in H$ with $|h|_H=1$, and let $\beta^1:=(\bar W^1_\jott,h)_H$, $\beta^2:=(W^2_\jott,h)_H$ and
$$
\beta(t):=(\bar W^2_\jott(t),h)_H=\beta^1(t\wedge\bar\tau_1)+\beta^2\big((t-\bar\tau_1)^+\big).
$$
The process $\beta$ is continuous and $\overline{\mathbb F}_2$-adapted. Let $s<t$ and $A\in\overline\CF^2_s$. We show
$\overline\EE_2[(\beta(t)-\beta(s))\mathds 1_A]=0$, treating the two parts of the increment separately.

(i) The increment $\beta^1(t\wedge\bar\tau_1)-\beta^1(s\wedge\bar\tau_1)$ vanishes on $\{\bar\tau_1\le s\}$ and depends only on
$\omega_1$. On $\{\bar\tau_1>s\}$ the second generator of $\overline\CF^2_s$ vanishes. Hence, up to a null set,
$A\cap\{\bar\tau_1>s\}=(A_1\times\Omega_2)\cap\{\bar\tau_1>s\}$ for some $A_1\in\overline\CF^1_s$. Since $\beta^1(\cdot\wedge\bar\tau_1)$ is an
$(\overline\CF^1_t)$-martingale, this part of the increment has zero expectation against $\mathds 1_A$.

(ii) Fix $\omega_1$; then $\bar\tau_1(\omega_1)$ is fixed. The section $A_{\omega_1}:=\{\omega_2:(\omega_1,\omega_2)\in A\}$ belongs to
$\CF^2_{(s-\bar\tau_1(\omega_1))^+}$. Under $Q(\mathfrak x_1(\omega_1),\cdot)$, the increment
$\beta^2((t-\bar\tau_1)^+)-\beta^2((s-\bar\tau_1)^+)$ is centred and independent of $\CF^2_{(s-\bar\tau_1)^+}$. Integrating first in
$\omega_2$ with respect to $Q(\mathfrak x_1(\omega_1),d\omega_2)$ and then in $\omega_1$ gives zero.

Hence $\beta$ is an $\overline{\mathbb F}_2$-martingale. The same argument applies to $\beta(t)^2-t$ and, for orthogonal
$h,h'$ or for $\jott\neq\jott'$, to the products of the corresponding processes. So the quadratic variation of $\beta$ is $t$,
and the cross variations vanish; the two pieces live on disjoint time intervals. By L\'evy's characterization theorem
\cite[Theorem 3.3.16]{KaratzasShreve}, the processes $(\bar W^2_\jott,\psi_i)_H$, $i\in\mathbb I$, $\jott=0,\ldots,\Jott$, are independent
standard $\overline{\mathbb F}_2$-Brownian motions. This proves the claim.
\end{proof}

\noindent\emph{The glued process for $\kappa=2$.} Let
$$
\bar\tau_2:=\big(\bar\tau_1+\tau^\ast_2(\bww_2)\big)\wedge T,\qquad
\bar\bww_2(t):=\begin{cases}\bar\bww_1(t),& 0\le t\le\bar\tau_1,\\ \bww_2(t-\bar\tau_1),&\bar\tau_1<t\le T.\end{cases}
$$
By (\ref{matching}), $\bar\bww_2$ has continuous paths, and it is $\overline{\mathbb F}_2$-adapted; $\bar\tau_2$ is an
$\overline{\mathbb F}_2$-stopping time, and $\bar\bww_2=\bar\bww_1$ on $[0,\bar\tau_1]$. For $t\ge\bar\tau_1$,
$$
\int_{\bar\tau_1}^{t}\Sigma_\jott\big(\bar v^\jott_2(s)\big)\,d\bar W^2_\jott(s)=\int_0^{t-\bar\tau_1}\Sigma_\jott\big(v^\jott_2(r)\big)\,dW^2_\jott(r)
\qquad\overline\PP_2\text{-a.s.},\quad \jott=0,\ldots,\Jott\ (v^0:=w);
$$
this holds first for elementary integrands, and then in general by approximation. Hence $\bar\bww_2$ solves (\ref{spdes1}) on
$[0,\bar\tau_2)$ over $\overline{\mathfrak A}_2$ with the Wiener process $\bar{\mathbb W}^2$. On $[0,T]$, $\bar\bww_2$ starts from $\bww_0$ (in
law) and solves (\ref{spdes1}) with $F_0,\ldots,F_\Jott$ multiplied by progressively measurable processes with values in
$[0,1]$: the cut-off factors of the first piece on $[0,\bar\tau_1)$ and those of the second piece, shifted by $\bar\tau_1$,
on $[\bar\tau_1,T]$.

\smallskip
\noindent\emph{Induction step from $\kappa$ to $\kappa+1$.} Suppose that $\overline{\mathfrak A}_\kappa=(\overline\Omega_\kappa,\overline\CF_\kappa,
(\overline\CF^\kappa_t),\overline\PP_\kappa)$ with $\overline\Omega_\kappa=\Omega_1\times\cdots\times\Omega_\kappa$, the Wiener process
$\bar{\mathbb W}^\kappa$, the process $\bar\bww_\kappa$ and the $(\overline\CF^\kappa_t)$-stopping time $\bar\tau_\kappa$ have been constructed.
Let $\mathfrak x_\kappa:=\bww_\kappa\big(\tau^\ast_\kappa(\bww_\kappa)\big)$ be the state of the $\kappa$-th piece at its own stopping time, and $\mu_\kappa:=\text{Law}(\mathfrak x_\kappa)$; under $\overline\PP_\kappa$ the $\kappa$-th piece has the law $\PP_\kappa$, so by assumption (\ref{initcond}),
$\mu_\kappa\in\mathcal I$. On $\{\bar\tau_\kappa<T\}$ we have $\bar\tau_\kappa=\bar\tau_{\kappa-1}+\tau^\ast_\kappa(\bww_\kappa)$, hence $\mathfrak x_\kappa=\bar\bww_\kappa(\bar\tau_\kappa)$. On $\{\bar\tau_\kappa=T\}$, $\bar\bww_\kappa(T)$ is the $\kappa$-th piece at the time $T-\bar\tau_{\kappa-1}$, which may be smaller than $\tau^\ast_\kappa(\bww_\kappa)$; (\ref{initcond}) says nothing about its law. There the continuation is not used, see (\ref{stopp_time}). Let $(\Omega_{\kappa+1},\CF_{\kappa+1},(\CF^{\kappa+1}_t),\PP_{\kappa+1})$, $\mathbb W^{\kappa+1}$,
$\bww_{\kappa+1}$ and $Q_{\kappa+1}$ be the local solution with cut-off level $\kappa+1$ and initial law $\mu_\kappa$ constructed
above. Define
$$
\overline\Omega_{\kappa+1}:=\overline\Omega_\kappa\times\Omega_{\kappa+1},\qquad
\overline\PP_{\kappa+1}(d\bar\omega,d\omega'):=\overline\PP_\kappa(d\bar\omega)\,Q_{\kappa+1}\big(\mathfrak x_\kappa(\bar\omega),d\omega'\big),
$$
the filtration $(\overline\CF^{\kappa+1}_t)$ and the Wiener process $\bar{\mathbb W}^{\kappa+1}$ as for $\kappa=2$, with
$(\overline\CF^\kappa_t,\bar\tau_\kappa,\mathbb W^{\kappa+1},\bww_{\kappa+1})$ in place of $(\overline\CF^1_t,\bar\tau_1,\mathbb W^2,\bww_2)$, and
\begin{equation}\label{stopp_time}
\bar\tau_{\kappa+1}:=\big(\bar\tau_\kappa+\tau^\ast_{\kappa+1}(\bww_{\kappa+1})\big)\wedge T,\qquad
\bar\bww_{\kappa+1}(t):=\begin{cases}\bar\bww_\kappa(t),& 0\le t\le\bar\tau_\kappa,\\ \bww_{\kappa+1}(t-\bar\tau_\kappa),&\bar\tau_\kappa<t\le T.\end{cases}
\end{equation}
Exactly as for $\kappa=2$ (with Lemma~\ref{lem:concatW} and the shift identity for the stochastic integrals) we obtain:
\begin{itemize}
\item[(G1)] $\bar{\mathbb W}^{\kappa+1}$ is a vector of independent cylindrical Wiener processes on $H$ over $\overline{\mathfrak A}_{\kappa+1}$,
and $\bar{\mathbb W}^{\kappa+1}=\bar{\mathbb W}^\kappa$ on $[0,\bar\tau_\kappa]$;
\item[(G2)] $\bar\bww_{\kappa+1}$ has continuous paths, is $(\overline\CF^{\kappa+1}_t)$-adapted, $\bar\tau_{\kappa+1}$ is an
$(\overline\CF^{\kappa+1}_t)$-stopping time, and $\bar\bww_{\kappa+1}=\bar\bww_\kappa$ on $[0,\bar\tau_\kappa]$;
\item[(G3)] $\bar\bww_{\kappa+1}$ solves (\ref{spdes1}) on $[0,\bar\tau_{\kappa+1})$ with the Wiener process $\bar{\mathbb W}^{\kappa+1}$;
\item[(G4)] on $[0,T]$, $\bar\bww_{\kappa+1}$ starts from $\bww_0$ (in law) and solves (\ref{spdes1}) with $F_0,\ldots,F_\Jott$ multiplied by
progressively measurable processes with values in $[0,1]$: the cut-off factors of the $i$-th piece on
$[\bar\tau_{i-1},\bar\tau_i)$, $i\le\kappa$ ($\bar\tau_0:=0$), and those of the $(\kappa+1)$-th piece on $[\bar\tau_\kappa,T]$.
\end{itemize}
The measure $\overline\PP_{\kappa+1}$ has marginal $\overline\PP_\kappa$ on $\overline\Omega_\kappa$. In (G2), the continuity at $\bar\tau_\kappa$ on $\{\bar\tau_\kappa<T\}$ follows from $\bww_{\kappa+1}(0)=\mathfrak x_\kappa=\bar\bww_\kappa(\bar\tau_\kappa)$, the analogue of (\ref{matching}).

\smallskip
\item \textbf{Globalisation of the solution:}
\noindent\emph{From local to global solutions.} By (G3), on the set $\{\bar\tau_\kappa\ge T\}$ the process $\bar\bww_\kappa$
solves (\ref{spdes1}) on the whole interval $[0,T]$. In the remaining part of this step we show that
$$
\overline\PP\big(\bar\tau_\kappa<T\big)\longrightarrow0,\qquad\text{i.e.}\qquad\overline\PP\big(\bar\tau_\kappa\ge T\big)\longrightarrow1,\qquad\kappa\to\infty .
$$
Since the sets $\{\bar\tau_\kappa\ge T\}$ increase in $\kappa$, their union has probability one, and therefore the glued
process is, $\overline\PP$-a.s., a solution of (\ref{spdes1}) on $[0,T]$. To make this precise, we first place all pieces on
one probability space.

\smallskip
\noindent\emph{The infinite product.} Let $\overline\Omega:=\prod_{\kappa\in\NN}\Omega_\kappa$ and
$\overline\CF:=\bigotimes_{\kappa\in\NN}\CF_\kappa$, and let
$\pi_\kappa:\overline\Omega\to\overline\Omega_\kappa=\Omega_1\times\cdots\times\Omega_\kappa$ be the projection. By construction,
$\overline\PP_{\kappa+1}$ is obtained from $\overline\PP_\kappa$ and the probability kernel $Q_{\kappa+1}$, as in (\ref{kernelP2}). By the
Ionescu-Tulcea theorem \cite{IonescuTulcea1949}, see also \cite[Theorem~14.35]{Klenke2020}, there is a unique probability measure $\overline\PP$ on
$(\overline\Omega,\overline\CF)$ with $\overline\PP\circ\pi_\kappa^{-1}=\overline\PP_\kappa$ for all $\kappa\in\NN$. We identify
$\overline\CF^\kappa_t$ with $\pi_\kappa^{-1}(\overline\CF^\kappa_t)$, and processes on $\overline\Omega_\kappa$ with their compositions
with $\pi_\kappa$. Since $\overline\CF^\kappa_t\subset\overline\CF^{\kappa+1}_t$, we let $\overline\CF_t$ be the augmented,
right-continuous version of $\sigma\big(\bigcup_{\kappa\in\NN}\overline\CF^\kappa_t\big)$. The intersection
$\bigcap_\kappa\overline\CF^\kappa_t$ would only contain the information of the first piece. Put
$\overline{\mathfrak A}:=(\overline\Omega,\overline\CF,(\overline\CF_t)_{t\in[0,T]},\overline\PP)$.

\smallskip
\noindent\emph{The pieces cover $[0,T]$.} The stopping times $\bar\tau_\kappa$ of (\ref{stopp_time}) are nondecreasing in $\kappa$.
Let
$$
\Gamma_\kappa:=\{\bar\tau_\kappa\ge T\},\qquad \Gamma:=\bigcup_{\kappa\in\NN}\Gamma_\kappa .
$$
We write $\Gamma_\kappa$ instead of $\bar\Omega_\kappa$, which already denotes the $\kappa$-th product space.
The glued process $\bar\bww_\kappa=(\bar w_\kappa,\bar v^1_\kappa,\ldots,\bar v^\Jott_\kappa)$ starts from $\bww_0$. On each interval
$[\bar\tau_{i-1},\bar\tau_i)$, $i\le\kappa$, it solves (\ref{spdes1}) with the cut-off of the $i$-th piece, and on $[\bar\tau_{\kappa-1},T]$
with the cut-off of the $\kappa$-th piece. So it is of the type admitted in (\ref{metauniform}), which therefore gives
\begin{equation}\label{unifglued}
\max_{\jott=1,\ldots,\Jott}\overline\EE\big[h_\jott(\bar v^\jott_\kappa,T)^{m_\jott}\big]\le C(T),\qquad\kappa\in\NN .
\end{equation}
If $\bar\tau_\kappa<T$, then the $\kappa$-th piece reaches the level $\kappa$ before time $T-\bar\tau_{\kappa-1}$, i.e.\
$h_\jott(v^\jott_\kappa,\tau^\ast_\kappa(\bww_\kappa))\ge\kappa$ for some $\jott$ (by the continuity of $s\mapsto h_\jott(v^\jott_\kappa,s)$ in (\ref{metacontinuityw})). We have $v^\jott_\kappa(s)=\bar v^\jott_\kappa(\bar\tau_{\kappa-1}+s)$ for
$0\le s\le T-\bar\tau_{\kappa-1}$. The norm of $\YY_\jott$ is invariant under time shifts and does not increase under
restriction to a subinterval; this holds for $\YY_\jott=L^{m_\jott}(0,T;Y_\jott)$ and $\YY_\jott=C([0,T];Y_\jott)$. Hence
$$
\{\bar\tau_\kappa<T\}\subset\bigcup_{\jott=1}^\Jott\big\{h_\jott(\bar v^\jott_\kappa,T)\ge\kappa\big\},
$$
and the Markov inequality and (\ref{unifglued}) give
$$
\overline\PP(\overline\Omega\setminus\Gamma_\kappa)=\overline\PP(\bar\tau_\kappa<T)\le\sum_{\jott=1}^\Jott\frac{C(T)}{\kappa^{m_\jott}}\longrightarrow0,
\qquad\kappa\to\infty .
$$
The sets $\Gamma_\kappa$ increase in $\kappa$, so $\overline\PP(\Gamma)=\lim_{\kappa\to\infty}\overline\PP(\Gamma_\kappa)=1$.

\smallskip
\noindent\emph{The global process.} We have $\bar\bww_{\kappa+1}=\bar\bww_\kappa$ on $[0,\bar\tau_\kappa]$; this holds by (\ref{stopp_time}). We also have $\bar{\mathbb W}^{\kappa+1}=\bar{\mathbb W}^\kappa$ on $[0,\bar\tau_\kappa]$. For
$t\in[0,T]$ we put
$$
\bar\bww(t):=\bar\bww_\kappa(t),\quad \bar{\mathbb W}(t):=\bar{\mathbb W}^\kappa(t)\qquad\text{on }\{t\le\bar\tau_\kappa\},\ \kappa\in\NN,
$$
and $\bar\bww(t):=0$, $\bar{\mathbb W}(t):=\bar{\mathbb W}^1(t)$ on the null set $\overline\Omega\setminus\Gamma$. This does not
depend on $\kappa$. Since $\{t\le\bar\tau_\kappa\}\in\overline\CF^\kappa_t$, the processes $\bar\bww$ and $\bar{\mathbb W}$ are
$\overline{\mathbb F}$-adapted, and on $\Gamma$ they have continuous paths.

\smallskip
\noindent\emph{$\bar{\mathbb W}$ is an $\overline{\mathbb F}$-Wiener process.} Fix $\jott$ and $h\in H$ with $|h|_H=1$, and put
$\beta:=(\bar W_\jott,h)_H$ and $\beta^\kappa:=(\bar W^\kappa_\jott,h)_H$. By Lemma~\ref{lem:concatW} and its iteration, $\beta^\kappa$ is a
Brownian motion with respect to $(\overline\CF^\kappa_t)$. Moreover, $\beta^\kappa(t)=\beta(t)$ on $\Gamma_\kappa$, so
$\beta^\kappa(t)\to\beta(t)$ $\overline\PP$-a.s. Let $s<t$ and $A\in\overline\CF^n_s$ for some $n\in\NN$. For $\kappa\ge n$, we have
$A\in\overline\CF^\kappa_s$, and hence $\overline\EE[(\beta^\kappa(t)-\beta^\kappa(s))\mathds 1_A]=0$. Since
$\overline\EE|\beta^\kappa(t)-\beta^\kappa(s)|^2=t-s$, the family is uniformly integrable, and letting $\kappa\to\infty$ gives
$\overline\EE[(\beta(t)-\beta(s))\mathds 1_A]=0$. The union $\bigcup_n\overline\CF^n_s$ is a $\pi$-system generating $\overline\CF_s$
up to null sets. Since $\beta$ is continuous, $\beta$ is an $\overline{\mathbb F}$-martingale. In the same way, $\beta(t)^2-t$ is
an $\overline{\mathbb F}$-martingale (for uniform integrability use $\overline\EE|\beta^\kappa(t)|^4=3t^2$), and so are the products
for orthogonal $h,h'$ or for $\jott\neq\jott'$. By L\'evy's characterization theorem \cite[Theorem 3.3.16]{KaratzasShreve},
$\bar{\mathbb W}=(\bar W_0,\ldots,\bar W_\Jott)$ is a vector of independent cylindrical Wiener processes on $H$ over
$\overline{\mathfrak A}$.

\smallskip
\noindent\emph{$\bar\bww$ is a global solution.} Fix $\kappa$. On $[0,\bar\tau_\kappa\wedge T)$ the process $\bar\bww_\kappa$ solves
(\ref{spdes1}) with the Wiener process $\bar{\mathbb W}^\kappa$, since each piece solves (\ref{spdes1}) up to its stopping time.
On $[0,\bar\tau_\kappa\wedge T]$ we have $\bar\bww=\bar\bww_\kappa$ and $\bar{\mathbb W}=\bar{\mathbb W}^\kappa$. The stopped processes
$\bar{\mathbb W}(\cdot\wedge\bar\tau_\kappa)=\bar{\mathbb W}^\kappa(\cdot\wedge\bar\tau_\kappa)$ are martingales with respect to both
$(\overline\CF^\kappa_t)$ and $\overline{\mathbb F}$. Since the stochastic integral is local, we get $\overline\PP$-a.s.
$$
\int_0^{t\wedge\bar\tau_\kappa}\Sigma_\jott\big(\bar v^\jott(s)\big)\,d\bar W_\jott(s)=\int_0^{t\wedge\bar\tau_\kappa}\Sigma_\jott\big(\bar v^\jott_\kappa(s)\big)\,d\bar W^\kappa_\jott(s),
\qquad t\in[0,T],
$$
and the same holds for $\jott=0$. Hence the integral equations of Definition~\ref{Def:mart-sol} hold for $\bar\bww$ with
$\bar{\mathbb W}$ for all $t<\bar\tau_\kappa\wedge T$, $\overline\PP$-a.s. On $\Gamma_\kappa$ this covers $[0,T)$, and $t=T$ follows
from the continuity of both sides. Since $\Gamma_\kappa\uparrow\Gamma$ and $\overline\PP(\Gamma)=1$, the tuple
$\big(\overline{\mathfrak A},\bar{\mathbb W},\bar\bww\big)$ is a martingale solution of (\ref{spdes1}) on $[0,T]$ in the sense of
Definition~\ref{Def:mart-sol}. By (\ref{metacontinuityw}), its paths are continuous. This finishes the proof.

\begin{remark}\label{rem:whyglue}
Why gluing is needed: the pieces $\bww_\kappa$ are martingale solutions on different stochastic bases. Assumption
(\ref{metauniform}) only controls the arguments of the cut-off functions, uniformly in $\kappa$; it gives no
$\kappa$-uniform compactness. Gluing uses nothing more than this bound, through the estimate of $\overline\PP(\bar\tau_\kappa<T)$
above. The alternative would be to prove tightness of the laws of the stopped cut-off solutions and to identify the
limit. That would require $\kappa$-uniform versions of (\ref{metacompact}) and a stability property of the equations
under convergence in law, which are not among the assumptions of Theorem~\ref{meta}.
\end{remark}

	\end{step}
	
\end{proof}

\medskip

\section{Application to a stochastic chemotaxis-fluid model}\label{sec:application}

Chemotaxis is defined as the oriented movement of cells (or an organism) in response to a chemical gradient. Many different kinds of motile cells undergo chemotaxis. For example, bacteria and many
amoeboid cells move toward a food source; immune cells like macrophages and neutrophils move toward invading cells in our bodies. Other cells, connected with the immune response and wound healing, are attracted to areas of inflammation by chemical signals.
In particular, there are many processes governed by chemotaxis. Besides Pfeffer, who described already chemotaxis $1884$ by sperm of ferns and mosses \cite{Pfeffer1884},
 Raper studied in $1935$ chemotaxis at the cellular slime mould {\sl Dictyostelium discoideum}, a simple model of morphogenesis \cite{Raper1935}.
During its life cycle, a population of `Dicto' grows by cell division as long as there is sufficient nourishment. If food is rare and the environmental living conditions are getting poor, then the individual amoeba starts to  emit cAMP
(Cyclic Adenosine Monophosphate).
Amoeba having sufficient nutriment do not strongly react to cAMP. While suffering from hunger, the reaction sensitivity
to cAMP increases, and with a small delay, the cAMP concentration
is amplified by famished cells that, additionally, by being hungry, sense cAMP. By this stimulus, the amoeba enters the aggregation phase, which is characterised by all hungry amoeba coming
together, forming a slug, and ends with building a fruiting body, which finally releases the spores. Theoretical and mathematical modelling of chemotaxis
dates to the works of \textsc{Patlak} \cite{Patlak1953} in the $1950$s and \textsc{Keller and Segel}  \cite{KSS} in the 1970s.
The Keller-Segel model describes the aggregation phase
which models the directed movement of microorganisms and cells stimulated by
a chemical produced by themselves.
For more details, we refer to the
 surveys about chemotaxis by the compelling articles of \textsc{Horstmann} \cite{Horstmann2003}, the  summary of  \textsc{Hillen et al.} \cite{Hillen2009} and  the more recently
   summaries of \textsc{ Bellomo et  al.} \cite{bellomo2015,bellomo201501} and of \textsc{Perthame} \cite{Perthame2004}.

Aerobic bacteria often live in thin fluid layers near solid-air-water contact lines. In the absence of fluid, cells can move in a straight line or exhibit random motion. In contrast, in the presence of fluid, cells may experience hydrodynamic interactions with the surrounding fluid and undergo fluid-mediated transport, such as diffusion, advection, or convective mixing. Fluid flow can also affect the formation and stability of chemical gradients, which are essential for chemotaxis. Here, the chemoattractant diffuses linearly, is transported in the fluid flow direction, and is consumed by the cells. On the other hand, the cell's motion is
influenced by linear or nonlinear diffusion,  transported along the fluid velocity, and partially directed in the direction of the chemoattractant
gradient.
The chemotaxis system describes a biological process in which cells (e.g., bacteria) accumulate by moving toward higher concentrations of chemotactic substances. Experiments have shown that the mechanism is a chemotactic movement of bacteria, often towards a higher concentration of oxygen, which they consume, a gravitational effect on the motion of the fluid by the heavier bacteria, and a convective transport of both cells and oxygen through the water. We refer to  \textsc{Liu and Lorz}  \cite{Lorz} and \textsc{Tuval} \cite{Tuval2005} for more details about its physical background.

We consider an initial–boundary value problem for the incompressible stochastic chemotaxis–Stokes equations with the porous-media-type diffusion in  a   complete probability space $\mathfrak{A}=(\Omega, \CF,\BF,\PP)$
with filtration $\BF=(\CF_t)_{t\ge 0}$  satisfying the usual conditions i.e.,
$\mathbb{P}$ is complete on $(\Omega, \CF)$,
for each $t\geq 0$, $\CF_t$ contains all $(\CF,\mathbb{P})$-null sets, and
the filtration $(\CF_t)_{t\ge 0}$ is right-continuous.
Here, the underlying spatial domain $\CO$ is a bounded domain of $\mathbb{R}^2$ with smooth boundary. Let  $W_1$, $W_2$ and $W_3$  be  three  independent cylindrical Wiener processes on $L^2(\CO)$ defined over the probability space
$\mathfrak{A}$  by
\DEQSZ
W_j(t) &=&\sum_{k=1}^\infty \varphi_k\,\mathbf{w}^j_k(t),\quad t\in [0,T],\,\quad j=1,2,3,
\EEQSZ
where $\{\mathbf{w}^j_k:k\in\NN\}$, $j=1,2,3$, are three mutually independent families of i.i.d.\ standard scalar Wiener processes and  $\{ \varphi_k:k\in\NN\}$ the eigenfunction of the Laplace operator $(-\Delta)$ on $L^2(\CO)$ corresponding to the associated eigenvalues  $\{\lambda_k:k\in\NN\}$. We investigate the  following coupled system ($\delta_n,\delta_c>0$)
\begin{align}\label{1.1*}
\lk\{
\barray d {n} +\delta _n \bu\cdot \nabla ndt& =& \lk[ r_n\Delta   |n|^{q-1}n- \chi \Div( n\nabla c) \rk]\, dt+g_{\gamma_1}(n) \circ dW_1,\vspace{0.2cm}
\\
d{c}+\delta _c \bu\cdot\nabla cdt & =&\lk[r_c \Delta c  -\zeta c+ \beta n \rk]\, dt+g_{\gamma_2}(c) \circ dW_2,\vspace{0.2cm}
\\
 d \bu +\nabla Pdt&=&\left[r_\bu\Delta \bu
 +n\star\Phi \right]dt+ \sigma_{\gamma_3}\, dW_3\vspace{0.2cm}
 \\
  \nabla \cdot \bu& =&0
\earray \rk.
\end{align}
with boundary and initial conditions
\begin{equation}\label{init}
\begin{split}
	&{n=0,}\qquad\nabla c\cdot \nu=0,\qquad \qquad\bu=0, \quad\text{ on } \partial\bo,\\
	&c(0)=c_0,\qquad n(0)=n_0,\qquad \bu(0)=\bu_0,
\end{split}
\end{equation}
 on a filtered probability space  $\mathfrak{A}$ where the noise intensities $\gamma_j$, $j=1,2,3,$ are given positive real numbers and  the linear diffusion coefficients $g_{\gamma_1}$, $g_{\gamma_2}$ and $\sigma_{\gamma_3}$ are defined respectively   by
\DEQSZ\label{g_sr}
g_{\gamma _j}(\psi) h &:=&  \psi \cdot (-\Delta)^{-\gamma_j/2} h=\sum_{k=1}^\infty \lambda _k^{-\gamma_j/2} \psi\cdot \varphi_k \la h,\varphi_k\ra, \quad \psi\in L^2(\CO),\ h\in H ,\ \ j=1,2,
\EEQSZ
and
\DEQSZ\label{sigma_sr}
\sigma _{\gamma_3}  h &:=&  \sum_{k=1}^\infty \mu_k^{-\gamma_3/2}\,\la h,\varphi_k\ra\,\mathbf e_k,\quad  h\in H,
\EEQSZ
where $(\mu_k,\mathbf e_k)_{k\in\NN}$ are the eigenpairs of the Stokes operator $A$ introduced below, with $(\mathbf e_k)$
orthonormal in $\mathbf L^2(\bo)$. That is, $\sigma_{\gamma_3}=A^{-\gamma_3/2}J$, where $J\varphi_k:=\mathbf e_k$ defines an
isometry from $L^2(\CO)$ onto $\mathbf L^2(\bo)$; in particular, $\sigma_{\gamma_3}$ takes values in divergence-free vector fields.

In system (\ref{1.1*}),  the unknowns are the cell density $n$,  the concentration of the chemical signal $c$, the pressure $P$ and velocity field  $\bu$ over the domain $\mathcal{O}$.
The positive terms $r_n$ and $r_c$  are the diffusivity of the cells and chemo-attractant, respectively, the positive value   $\chi$ is the chemotactic sensitivity,
 and the positive constant $\zeta$ is the so-called damping constant.  The given function $\Phi=(\phi_1,\phi_2)$  represents the kernel of gravitational potential. Let us note that the gravitational potential influences the cell's behaviour by eliminating natural convection, and this phenomenon impacts the dynamics of the fluid by inducing the external force $n\ast\Phi$ on the velocity field equation.  We  then  recall that the system (\ref{1.1*}) is equivalent to the following It\^o system
  \begin{align}\label{1.1}
  	\lk\{
  	\barray d {n} +\delta _n \bu\cdot \nabla ndt& =& \lk[ r_n\Delta   |n|^{q-1}n+\theta n- \chi \Div( n\nabla c) \rk]\, dt+g_{\gamma_1}(n)  dW_1,\vspace{0.2cm}
  	\\
  	d{c}+\delta _c \bu\cdot\nabla cdt & =&\lk[r_c \Delta c  -\alpha c+ \beta n \rk]\, dt+g_{\gamma_2}(c) dW_2,\vspace{0.2cm}
  	\\
  	d \bu +\nabla Pdt&=&\left[r_\bu\Delta \bu
  	+n\star\Phi \right]dt+ \sigma_{\gamma_3}\, dW_3 ,\vspace{0.2cm}%
  	\\
  	\nabla \cdot \bu& =&0,
  	\earray \rk.
  \end{align}
where $\alpha,r_c>0$,  $$\theta(x)=\frac12\sum_{k}\lambda_k^{-\gamma_1}\varphi_k(x)^2,
\quad \alpha(x)=\zeta-\frac12\sum_{k}\lambda_k^{-\gamma_2}\varphi_k(x)^2.$$

We now introduce the following spaces
\begin{align*}
	&\mathcal{V} =\left\{ \bu\in \mathcal{C}_{c}^{\infty }(\bo,\mathbb{R}^2)\,\,\text{such that}%
	\,\,\nabla\cdot \bu=0\right\}, \\
		&\mathbf{L}^2(\bo)=\,\,\text{closure of $\mathcal{V}$ in }\,\,(L^{2}(\bo))^2 ,\\
&	\mathbf{H} =\,\,\text{closure of $\mathcal{V}$ in }\,\,(H_0^{1}(\bo))^2 ,\qquad\text{ and }\qquad\mathbf{V}=	\mathbf{H}\cap (H^{2}(\bo))^2.
\end{align*}
We endow $\mathbf{L}^2(\bo)$ with the scalar product and norm of $(L^2(\bo))^2$. As
usual we equip the space $\mathbf{H}$ with the scalar product
$(\nabla \bu, \nabla \bv)_{\mathbf{L}^2}$ which, owing to the Poincar\'e inequality, is equivalent to the
$(H^1(\bo))^2$-scalar product. Let $\Pi: (L^2(\bo))^2 \rightarrow \mathbf{L}^2(\bo)$ be the Helmholtz-Leray projection. We denote by $$A=-\Pi\Delta$$ the
Stokes operator with domain $D(A)=\mathbf{V}$.
In our setting, we define the notion of a martingale solution to our problem as follows\footnote{	We set $H^{-1}_2(\CO):=(H^1_0(\CO))'$ with the norm
	$|f|_{H^{-1}_2}:=|\nabla(-\Delta)^{-1}f|_{L^2}=\langle f,(-\Delta)^{-1}f\rangle^{1/2}$.
}.
\begin{definition}\label{Def:mart-sol:cns}
Fix $T>0$. A \textnormal{martingale solution} to the problem
\eqref{1.1}   with boundary and initial conditions \eqref{init} is given by the tuple
$$
\left(\Omega ,{{\mathcal{F}}},\mathbb{P},{\mathbb{F}},
(W_1,W_2,W_3),n,c,\bu\right),
$$
such that
\begin{itemize}
	\item  $\mathfrak{A}:=(\Omega ,{{\mathcal{F}}},{\mathbb{F}},\mathbb{P})$ is a complete filtered
	probability space with a filtration ${\mathbb{F}}=\{{{\mathcal{F}}}_t:t\in [0,T]\}$ satisfying the usual conditions,
	\item $W_1$, $W_2$ and $W_3$ are three mutually independent cylindrical Wiener processes on $H$ over the probability space
	$\mathfrak{A}$,
	\item the initial condition $(n_0,c_0,\bu_0)$ is an $\CF_0$-measurable random variable with a prescribed law $\mu_0$ on
	$H^{-1}_2(\CO)\times H^1_2(\CO)\times\nsH$, and
	\item $n:[0,T] \times \Omega \rightarrow H^{-1}_2(\CO)$, and $c:[0,T] \times \Omega \rightarrow
	L^2(\CO)$, and $\bu:[0,T] \times \Omega \rightarrow\nsH$
	are $\BF$--progressively  measurable processes  such that $\PP$-a.s.\ $$(n,c,\bu)\in C([0,T];H^{-1}_{2}(\CO))\times  C([0,T];L^2(\CO))\times  C([0,T];\mathbf{L}^2(\CO)),$$
	and satisfy for all $t\in [0,T]$ the following stochastic integral equations
	$\PP$--a.s.,
	\begin{align}\label{eq:nStratonovich:app}
		n(t)+\int_0^ t\bu(s)\cdot\nabla n(s)\, ds&=n_0+\int_0^t (r_n\Delta  \abs{n(s)}^{q-1}n(s) +\theta n(s)\notag\\
		&\qquad{}-\chi\mbox{div}( n(s) \nabla c(s) )\,ds+\int_0^t g_{\gamma_1}(n(s)) dW_1(s),
	\end{align}
	\begin{equation}\label{eq:cStratonovich:app}
		c(t)+\int_0^ t\bu(s)\cdot\nabla c(s)\, ds=c_0+\int_0^t (r_c\Delta c(s)-\alpha c(s)+\beta n(s))\,ds+
		\int_0^t g_{\gamma_2}(c(s))dW_2(s),
	\end{equation}
	and
	\begin{equation} \label{eq:buStratonovich:app}
		\bu(t) +r_\bu\int_0^ t A \bu(s)ds=\bu_0+\int_0^ t \Pi(n(s)\star\Phi) ds+  \int_0^ t\sigma_{\gamma_3}
		dW_3(s),
	\end{equation}
\end{itemize}
where the equations (\ref{eq:nStratonovich:app}), (\ref{eq:cStratonovich:app}) and (\ref{eq:buStratonovich:app}) hold in $H^{-1}_{2}(\bo)$, $L^2(\CO)$ and $\mathbf{L}^2(\bo)$ respectively.
\end{definition}
In what follows, we gather assumptions that we will need for the proof of the main result of this section.
\begin{assumption}\label{assum}
We will assume that
\begin{description}
	\item[(A1)] Assumption on the stochastic perturbation: $\gamma_1>2$, $\gamma_2>2$ and ${\gamma_3>2}$.
	\item[(A2)] Assumption on the gravitational potential: ${q>3}$ and the kernel of gravitational potential $\Phi=(\phi_1,\phi_2)$ where $\phi_i={e^{\delta_i\Delta}}$ with some fixed  real numbers $\delta_i>0$, $i=1,2$.
	\item[(A3)] Assumption on the initial conditions:
 $(n_0,c_0,\bu_0)$ is  a $H^{-1}_2(\bo)\times H^1_2(\bo)\times\nsH$-valued   $\CF_0$-measurable random variable   satisfying
	\begin{equation}
		\EE|n_0|_{H^{-1}_2}^{2(q+1)}<\infty,\  \EE\abs{c_0}^{4(q+1)}_{H_2^1}<\infty,\  \EE\abs{\bu_0}^{4(q+1)}_\nsH<\infty.\label{3.7}
	\end{equation}
	\item[(A4)] Assumptions on $\alpha$ and $\theta$: We assume $\inf_\CO\alpha>0$ (and put
	$\alpha_0:=\min(r_c,\inf_\CO\alpha)$) and $\theta\in W^{1,\infty}(\CO)$, so that
	multiplication by $\theta$ is bounded on $H^{-1}_2(\CO)$.
\end{description}
\end{assumption}
\begin{remark}
Assumption (A1) is needed to ensure that  $|\sigma_{\gamma_3}|_{L_{HS}(L^2,\nsH)}<\infty$ and also the existence of  a constant $C>0$ such that
\begin{align}
	&
	|g_{\gamma_1}(\psi)|_{L_{HS}(L^2,H^{-1}_2)}^2\le  C|\psi|_{H^{-1}_2} ^2 , \quad \psi\in L^2(\CO), \label{eq:H-1-2q}\\
	&	|g_{\gamma_2 }(\psi)|_{L_{HS}(L^2,L^2)}^2\le C |\psi|^2_{L^2}, \quad \psi\in L^2(\CO),	\label{eq:Hil-Schi2}\\
	&	{|g_{\gamma_2 }(\psi)|_{L_{HS}(L^2,H^1_2)}^2\le C |\psi|^2_{H^1_2}, \quad \psi\in H^1_2(\CO),}	\label{eq:Hil-Schi2H1}
\end{align}
where $L_{HS}(L^2,E)$ denote the space of {Hilbert--Schmidt} operators defined from $L^2(\bo)$ into $E$, with $E$ either $\nsH$, $L^2(\bo)$, $H^1_2(\bo)$ or  $H^{-1}_2(\bo)$.
We refer to  Appendix \ref{Appendix_A_noise} for more explanation.
\end{remark}

lowing boundedness conditions hold:

The main result of this section is stated in the following theorem.

\begin{theorem}\label{main}
Let us assume that  Assumption \ref{assum} is satisfied. Then,
the system given in  \eqref{1.1} with boundary and initial conditions \eqref{init} has a martingale solution in the sense of Definition \ref{Def:mart-sol:cns}, with initial law equal to the law of
$(n_0,c_0,\bu_0)$.
In addition, for all $p\in {[1,q+1]}$ there exists a constant $C>0$ such that
\begin{align}\label{3.8}
&	\EE	\sup_{0\le s\le  {T} }	\abs{\bu(s)}^{2p}_\nsH+\EE\sup_{0\le s\le T}	\abs{c(s)}^{2p}_{L^2}+\EE\sup_{0\le s\le  {T}  }|n(s)|_{H^{-1}_2}^{2p}\notag\\
&{}+\EE\left(\int_0^{T} \abs{n(s)}^{q+1}_{L^{q+1}} ds\right)^p+\EE\left(\int_0^T\abs{c(s)}^2_{H^1_2}ds\right)^p+\EE\left(\int_0^T\abs{\bu(s)}^2_\nsV ds\right)^p\notag
\\
&	\leq e^{CT}\left(1+ \EE\abs{\bu_0}^{2p}_\nsH+\EE \abs{c_0}^{2p}_{H_2^1}+\EE |n_0|_{H^{-1}_2}^{2p}\right).
\end{align}

\end{theorem}

\begin{proposition}\label{prop:nonneg}
Let the assumptions of Theorem~\ref{main} hold, and assume in addition that, $\PP$-a.s., $n_0\in L^1(\CO)$, $n_0\ge0$ and
$c_0\ge0$ a.e.\ in $\CO$. Then the solution $(n,c,\bu)$ of Theorem~\ref{main} satisfies, $\PP$-a.s., $c(t)\ge0$ a.e.\ in
$\CO$ for all $t\in[0,T]$, and $n(t)\ge0$ in the sense of distributions for all $t\in[0,T]$; in particular, $n(t)\ge0$
a.e.\ in $\CO$ for a.e.\ $t\in[0,T]$.
\end{proposition}
\begin{proof}
We follow the proof of non-negativity in \cite{chemo1}. The porous medium term and the transport term have a sign, and
the chemotaxis term is treated as in the deterministic setting. Since the diffusion of $n$ is degenerate at $n=0$, we use
an $L^1$-type functional instead of the $L^2$-type functional of \cite{chemo1}. As in \cite{chemo1}, we apply the It\^o
formula to $n$ directly; to be precise, one has to apply it to a regularization of $n$, e.g.\ $(I-\eps\Delta)^{-1}n$, and
let $\eps\to0$. We omit this step. Since $c=c_\kappa$ on $[0,\tau_\kappa]$, (\ref{3.12}) gives
$c\in L^2(0,T;H^2_2(\CO))$ $\PP$-a.s.; by a localization we may assume $\EE\int_0^T|c(s)|_{H^2_2}ds<\infty$.

\emph{The cell density.} For $\delta\in(0,1)$ let $G_\delta\in C^2(\RR)$ be convex with $G_\delta(r)=0$ for $r\ge0$,
$G'_\delta(r)=-1$ for $r\le-\delta$ and $0\le G''_\delta\le\frac2\delta$. Then $G''_\delta=0$ outside $[-\delta,0]$,
$|G'_\delta|\le1$, $0\le r^--G_\delta(r)\le\delta$, and
$$
H_\delta(r):=\int_0^r s\,G''_\delta(s)\,ds\quad\text{satisfies}\quad 0\le H_\delta(r)\le\delta,\qquad
0\le r^2G''_\delta(r)\le2\delta,\qquad r\in\RR .
$$
Let $\Psi_\delta(n):=\int_\CO G_\delta(n(x))\,dx$. We estimate the terms of the It\^o formula for $\Psi_\delta(n(t))$,
see (\ref{eq:nStratonovich:app}), one by one.
\begin{itemize}
\item The porous medium term. Since $G'_\delta(0)=0$ and $|n|^{q-1}n=0$ on $\partial\CO$,
$$
\langle G'_\delta(n),r_n\Delta(|n|^{q-1}n)\rangle=-q\,r_n\int_\CO G''_\delta(n)\,|n|^{q-1}|\nabla n|^2\,dx\le0 .
$$
\item The linear term. Since $\theta\ge0$ and $G'_\delta(n)\,n\le n^-$, we have
$\langle G'_\delta(n),\theta n\rangle\le|\theta|_{L^\infty}|n^-|_{L^1}$.
\item The chemotaxis term. Since $\nabla c\cdot\nu=0$ on $\partial\CO$, $G'_\delta(0)=0$ and
$G''_\delta(n)\,n\nabla n=\nabla H_\delta(n)$,
$$
-\chi\langle G'_\delta(n),\Div(n\nabla c)\rangle=\chi\int_\CO G''_\delta(n)\,n\nabla n\cdot\nabla c\,dx
=-\chi\int_\CO H_\delta(n)\,\Delta c\,dx\le C\chi\,\delta\,|c|_{H^2_2} .
$$
\item The transport term. Since $\Div\bu=0$ and $\bu=0$ on $\partial\CO$,
$\langle G'_\delta(n),\bu\cdot\nabla n\rangle=\int_\CO\bu\cdot\nabla G_\delta(n)\,dx=0$.
\item The It\^o correction. Since $g_{\gamma_1}(n)\varphi_k=\lambda_k^{-\gamma_1/2}n\varphi_k$, see (\ref{g_sr}), and
$\sum_k\lambda_k^{-\gamma_1}\varphi_k^2=2\theta$,
$$
\frac12\sum_{k=1}^\infty\int_\CO G''_\delta(n)\,|g_{\gamma_1}(n)\varphi_k|^2\,dx=\int_\CO G''_\delta(n)\,n^2\,\theta\,dx
\le2\delta\,|\theta|_{L^\infty}|\CO| .
$$
\end{itemize}
The stochastic integral has expectation zero, since $|G'_\delta|\le1$, and $\Psi_\delta(n_0)=0$, since $n_0\ge0$. Hence
$$
\EE\,\Psi_\delta(n(t))\le|\theta|_{L^\infty}\int_0^t\EE|n^-(s)|_{L^1}\,ds
+C\delta\Big(\chi\,\EE\int_0^T|c(s)|_{H^2_2}\,ds+|\theta|_{L^\infty}|\CO|\,T\Big).
$$
Letting $\delta\to0$, we obtain $\EE|n^-(t)|_{L^1}\le|\theta|_{L^\infty}\int_0^t\EE|n^-(s)|_{L^1}ds$, and the Gronwall lemma gives
$n^-(t)=0$ for a.e.\ $t\in[0,T]$, $\PP$-a.s. Since $n\in C([0,T];H^{-1}_2(\CO))$ and the non-negative distributions form a
closed cone in $H^{-1}_2(\CO)$, $n(t)\ge0$ for all $t\in[0,T]$.

\emph{The chemical.} We apply the It\^o formula to $|c^-(t)|^2_{L^2}$ (again with $(r^-)^2$ replaced by a $C^2$
approximation), see (\ref{eq:cStratonovich:app}). By the Neumann condition,
$\langle-2c^-,r_c\Delta c\rangle=-2r_c|\nabla c^-|^2_{L^2}\le0$; by (A4), $\langle-2c^-,-\alpha c\rangle=-2\int_\CO\alpha|c^-|^2dx\le0$;
since $n\ge0$ and $\beta>0$, $\langle-2c^-,\beta n\rangle\le0$; and the transport term vanishes as above. Since
$\sum_k\lambda_k^{-\gamma_2}\varphi_k^2=2(\zeta-\alpha)\le2\zeta$, the It\^o correction is
$$
\sum_{k=1}^\infty\int_\CO\mathbf 1_{\{c<0\}}\,|g_{\gamma_2}(c)\varphi_k|^2\,dx
=\int_\CO|c^-|^2\sum_{k=1}^\infty\lambda_k^{-\gamma_2}\varphi_k^2\,dx\le2\zeta\,|c^-|^2_{L^2} .
$$
Hence $\EE|c^-(t)|^2_{L^2}\le\EE|c^-_0|^2_{L^2}+2\zeta\int_0^t\EE|c^-(s)|^2_{L^2}ds$. Since $c_0\ge0$, the Gronwall lemma
gives $c^-(t)=0$ for all $t\in[0,T]$, by the continuity of $c$ in $L^2(\CO)$.
\end{proof}

\begin{remark}
Analysing the proof, the following claim can be shown. If we assume, in addition, that 
$$	\EE|n_0|_{H^{-1}_2}^{{4(q+1)}}<\infty,
$$
then we can show that there exists a constant $C>0$ such that 
\begin{align*}
&	\EE	\sup_{0\le s\le  {T} }	\abs{\ukk(s)}^{{4(q+1)}}_\nsH+\EE\sup_{0\le s\le T}	\abs{c_k(s)}^{{4(q+1)}}_{L^2}+\EE\sup_{0\le s\le  {T}  }|\nk(s)|_{H^{-1}_2}^{{4(q+1)}}\notag\\
&+\EE\left(\int_0^{T} \abs{\nk(s)}^{q+1}_{L^{q+1}} ds\right)^{{2(q+1)}}+\EE\left(\int_0^T\abs{ \ckk(s)}^2_{H^1_2}ds\right)^{{2(q+1)}}+\EE\left(\int_0^T\abs{\ukk(s)}^2_\nsV ds\right)^{{2(q+1)}}\notag
\\
&	\leq e^{CT}\left(1+ \EE\abs{\bu_0}^{{4(q+1)}}_\nsH+\EE\abs{c_0}^{{4(q+1)}}_{H_2^1}+\EE |n_0|_{H^{-1}_2}^{{4(q+1)}}\right).
\end{align*}

\end{remark}
\begin{proof}
The proof will use the abstract existence Theorem \ref{meta}. Here, we will only prove that the assumptions  \eqref{metawellposed}
 to (\ref{initcond}), Assumptions~\ref{ass:noise} and \ref{ass:coeff}, and the bound regarded in  \eqref{metauniform} are satisfied.

The proof consists of three steps. In the first step, we will define the underlying spaces and the operator.
 In the second step, we will
show that the assumptions  \eqref{metawellposed}
 to {(\ref{initcond})} and the bound in {(\ref{metauniform})}  of Theorem \ref{meta} are satisfied.
 Assumptions~\ref{ass:noise} and \ref{ass:coeff} are verified before Step~(I).
 We will apply Theorem \ref{meta} in the third step.
 The estimates regarded in \eqref{metauniform} will give the regularity property of the solution.
Before starting wiith the actual proof, let us introduce the following notation:
	$-\Delta$ denotes the Dirichlet Laplacian on $L^2(\CO)$. Its eigenpairs $(\lambda_k,\varphi_k)$ are
	the ones used in (\ref{g_sr}), (\ref{sigma_sr}) and in the definition of $W_j$, and $\lambda_1>0$.
	We set $H^{-1}_2(\CO):=(H^1_0(\CO))'$ with the norm
	$|f|_{H^{-1}_2}:=|\nabla(-\Delta)^{-1}f|_{L^2}=\langle f,(-\Delta)^{-1}f\rangle^{1/2}$.
	Further, $V:=L^{q+1}(\CO)$, and $V^\ast$ is the dual of $V$ in the Gelfand triple
	$V\subset H^{-1}_2(\CO)\subset V^\ast$, i.e.\ $V^\ast=(-\Delta)\big(L^{\frac{q+1}q}(\CO)\big)$ with
	$|v|_{V^\ast}=|(-\Delta)^{-1}v|_{L^{\frac{q+1}q}}$. In particular
	\begin{equation}\label{Vstar}
		|\Delta(|n|^{q-1}n)|_{V^\ast}=|n|^q_{L^{q+1}},\qquad |h|_{V^\ast}\le C|h|_{H^{-1}_2}.
	\end{equation}
	The equation for $n$ is understood in the variational sense of this Gelfand triple, which
	corresponds to the homogeneous Dirichlet condition for $|n|^{q-1}n$, i.e.\ to $n=0$ on $\partial\CO$ in (\ref{init}) (cf.\ \cite[Example 4.1.11]{weiroeckner}).
	The equation for $c$ keeps the Neumann condition in (\ref{init}), and the Stokes equation the
	no-slip condition. Throughout, $q>3$, and for $\xi:[0,T]\to H^{-1}_2(\CO)$ we write
	$$
	\mathcal A(\xi):=\sup_{0\le t\le T}|\xi(t)|_{H^{-1}_2},\qquad
	\mathcal B(\xi):=\Big(\int_0^T|\xi(t)|^{q+1}_{L^{q+1}}\,dt\Big)^{\frac1{q+1}}.
	$$
	Two consequences of these conventions are used repeatedly.
	\begin{itemize}
		\item For $N\in L^{q+1}(\CO)$ let $\psi:=\nabla(-\Delta)^{-1}N$. By elliptic
		$L^{q+1}$-regularity and $W^{1,q+1}(\CO)\hookrightarrow L^\infty(\CO)$ (here $2<q+1$),
		\begin{equation}\label{EE}   \tag{(E)}
			|\psi|_{L^\infty}+|\nabla\psi|_{L^{q+1}}\le C|N|_{L^{q+1}}.
		\end{equation}
		\item
		For $0\le s<\frac12$ we have $H^{s}_2(\CO)=H^{s}_{2,0}(\CO)$, so
		$H^{-s}_2(\CO)$ is the dual of $H^s_2(\CO)$ and
		\begin{equation}\label{DD}   \tag{(D)}
			|\int_\CO fg\,dx|\le|f|_{H^{-s}_2}|g|_{H^{s}_2}\le C|f|_{H^{-s}_2}|g|_{H^1_2}
		\end{equation} 
		for all
		$g\in H^1_2(\CO)$, without boundary terms, also when $g$ does not vanish on $\partial\CO$.
	\end{itemize}
	This settles item (iv) of Review~24 for the analysis below. The boundary conditions in (\ref{init}) are the
	ones used in the analysis: homogeneous Dirichlet conditions for $n$, homogeneous Neumann conditions for $c$, and
	no-slip conditions for $\bu$.
In addition, let us introduce the following 
 abbreviation being used throughout the proof:
\DEQSZ\label{AB}
\mathcal A(\xi):=\sup_{0\le t\le T}|\xi(t)|_{H^{-1}_2},\qquad
\mathcal B(\xi):=\Big(\int_0^T|\xi(t)|^{q+1}_{L^{q+1}}\,dt\Big)^{\frac1{q+1}}.
\EEQSZ

\noindent\textbf{Verification of Assumption~\ref{ass:coeff}.} We have $\Jott=2$, $w=n$, $v^1=c$ and $v^2=\bu$. This order is
forced by (\ref{spdeskappa}): $F_1$ may depend on $v^1$ and $v^2$, but $F_2$ only on $v^2$, and the $c$-equation depends on
$\bu$, but not conversely. The space of Step~(I) is $\X=L^{q+1}(0,T;H^{-1/3}_2(\CO))$. Below, $\alpha$ and $\theta$ denote the
exponents in Assumption~\ref{ass:coeff}, not the functions $\alpha$, $\theta$ in (\ref{1.1}). The state spaces are $U=K_0=H^{-1}_2(\CO)$,
$U_1=K_1=L^2(\CO)$ and $U_2=K_2=\mathbf L^2(\bo)$, the Wiener processes $W_0,W_1,W_2$ of Section~\ref{metatt} are
$W_1,W_2,W_3$ here, and the set $\mathcal I$ of admissible initial laws consists of the laws of
$H^{-1}_2(\CO)\times H^1_2(\CO)\times\nsH$-valued random variables satisfying (\ref{3.7}).
\begin{itemize}
	\item[(C0)] For $n$ we take the Gelfand triple $V_0\hookrightarrow K_0\hookrightarrow V_0^\ast$ with $V_0=V=L^{q+1}(\CO)$ and
	$K_0=H^{-1}_2(\CO)$, and $\alpha=q+1$. By \cite[Example 4.1.11]{weiroeckner}, $\DeltaA_0n=r_n\Delta(|n|^{q-1}n)$ is
	hemicontinuous and monotone, i.e.\ (C0) holds with $\lambda_0=0$, and $|\DeltaA_0n|_{V_0^\ast}=r_n|n|^{q}_{L^{q+1}}$ by
	(\ref{Vstar}).
	\item[(C1)] We take
	\begin{itemize}
		\item for $c$: $K_1=L^2(\CO)$, $D_1=\{\varphi\in H^2_2(\CO):\partial_\nu\varphi=0\}$, $\X_1=L^2(0,T;H^{3/2}_2(\CO))$ and
		$\DeltaA_1=r_c\Delta-\alpha(\cdot)$, where $\Delta$ is the Neumann Laplacian;
		\item for $\bu$: $K_2=\mathbf L^2(\bo)$, $D_2=\nsV$, $\X_2=C([0,T];\nsH)$ and $\DeltaA_2=-r_\bu A$, where $A$ is the Stokes
		operator.
	\end{itemize}
	Both operators are linear and symmetric, so $\langle\DeltaA_iv,\varphi\rangle=(v,\DeltaA_i\varphi)_{K_i}$ with
	$\DeltaA_i\varphi\in K_i$ for $\varphi\in D_i$. Hence $v\mapsto\langle\DeltaA_iv(\cdot),\varphi\rangle$ is linear and continuous
	from $\X_i\cap C([0,T];K_i)$ into $L^1(0,T)$, with norm $\le|\DeltaA_i\varphi|_{K_i}\|v\|_{L^1(0,T;K_i)}\le C_\varphi\|v\|_{\X_i}$.
	\item[(C2)] With the cut-off $\Theta_\kappa=\Theta_\kappa(\bu,\cdot)$ on the transport terms, the nonlinearities are
	$$
	F_0=\theta\xi-\chi\Div(\xi\nabla c)-\Theta_\kappa\Div(\bu\,\xi),\qquad F_1=\beta\xi-\Theta_\kappa\,\bu\cdot\nabla c,\qquad
	F_2=\Pi(\xi\star\Phi).
	$$
	For $F_1$ and $F_2$, (C2) is immediate after moving the derivatives onto $\varphi\in D_i$: since $\Div\bu=0$ and $\bu=0$ on
	$\partial\CO$, $|\langle\bu\cdot\nabla c,\varphi\rangle|=|(c,\bu\cdot\nabla\varphi)|\le C_\varphi|c|_{L^2}|\bu|_{\nsH}$; by (D),
	$|\langle\xi,\varphi\rangle|\le C|\xi|_{H^{-1/3}_2}|\varphi|_{H^1_2}$; and $F_2$ is linear and continuous in $\xi$ by (A2). For
	$F_0$, (C2) follows from \ref{EE}, \ref{DD}, the bounds
	$$
	|\Div(\xi\nabla c)|_{V^\ast}\le C|\xi|_{H^{-1/3}_2}|c|_{H^{3/2}_2},\qquad|\Div(\bu\,\xi)|_{V^\ast}\le C|\xi|_{H^{-1/3}_2}|\bu|_{\nsH},
	$$
	and the H\"older inequality in time, since $\frac1{q+1}+\frac12\le\frac q{q+1}$ for $q\ge3$. The products give the exponent
	$\theta=2$, and $\Theta_\kappa\in[0,1]$ depends continuously on $\bu$ by (C3).
	\item[(C3)] The cut-off acts only through $h(\bu,t)$ defined in Step~(I), i.e.\ $h_2=h$ and $\Theta^1_\kappa\equiv1$, and it
	multiplies only the transport terms, as admitted after (\ref{2.5}). By Remark~\ref{rem:cutoff}~(i), $h(\cdot,s)$ is
	$1$-Lipschitz on $C([0,T];\nsH)=\X_2\cap C([0,T];K_2)$, so (C3) holds.
\end{itemize}

\noindent\textbf{Verification of Assumption~\ref{ass:noise}.} We have $\Sigma_0=g_{\gamma_1}$, $\Sigma_1=g_{\gamma_2}$ and
$\Sigma_2=\sigma_{\gamma_3}$. Since $g_{\gamma_1}$ is linear, (\ref{eq:H-1-2q}) and the density of $L^2(\CO)$ in $H^{-1}_2(\CO)$
show that $g_{\gamma_1}$ extends to a bounded linear map from $K_0=H^{-1}_2(\CO)$ into $L_{HS}(L^2(\CO);H^{-1}_2(\CO))$. By
(\ref{eq:Hil-Schi2}), $g_{\gamma_2}$ is a bounded linear map from $K_1=L^2(\CO)$ into $L_{HS}(L^2(\CO);L^2(\CO))$. Finally,
$\sigma_{\gamma_3}$ does not depend on $\bu$, and $\sigma_{\gamma_3}\in L_{HS}(L^2(\CO);\nsH)\subset L_{HS}(L^2(\CO);\mathbf L^2(\bo))$
by (A1). Hence Assumption~\ref{ass:noise} holds.

\smallskip

\begin{step}
\item
In the first step, we will define the underlying spaces, the argument of the cut-off function, the operator, and the subspace on which the operator acts.

\noindent\textbf{Choice of the parameters.} We set
$$
s^{\ast\ast}:=\frac13,\qquad r^\ast:=2,\qquad m^\ast:=q+1,\qquad\text{so that}\qquad
\frac{m^\ast}{r^\ast}=\frac{q+1}2>1 .
$$
We keep the notation $\frac{m^\ast}{r^\ast}$ in the estimates below. Thus $2\frac{m^\ast}{r^\ast}=q+1$,
$4\frac{m^\ast}{r^\ast}=2(q+1)$ and $8\frac{m^\ast}{r^\ast}=4(q+1)$. We let
$$
\X:= L^{m^\ast}\big(0,T;H^{-s^{**}}_{r^*}(\CO)\big)=L^{q+1}\big(0,T;H^{-1/3}_2(\CO)\big),
$$
and we denote by $\CM^{m^\ast}_\MA(\X)$ the space defined by
\DEQS
\CM^{m^\ast}_\MA(\X)&:=&\Big\{ \xi:\CO\times [0,T]\times \Omega\to\RR, \,\,
\\
&&\quad  \mbox{ $\xi$ is $(\CF_t)_{t\ge 0}$-progressively measurable and }\, \EE\|\xi\|_{\BX}^{m^\ast}<\infty
\Big\}.
\EEQS
By Lemma~\ref{interpolation_11} in Appendix~\ref{TP}, applied with $\theta=\frac13$ and $\theta=\frac14$, there is
a constant $C=C(T,\CO,q)$ with
\begin{equation}\label{interpolX}
\|\xi\|_{\X}\le C\,\mathcal A(\xi)^{1/3}\mathcal B(\xi)^{2/3},\qquad
\|\xi\|_{L^{q+1}(0,T;H^{-1/4}_2)}\le C\,\mathcal A(\xi)^{1/4}\mathcal B(\xi)^{3/4},
\end{equation}
and hence, by the Young inequality, for every $p\ge1$,
\begin{equation}\label{interpolXp}
\|\xi\|^p_{\X}+\|\xi\|^p_{L^{q+1}(0,T;H^{-1/4}_2)}\le C\big(\mathcal A(\xi)^p+\mathcal B(\xi)^p\big).
\end{equation}

Let us define the arguments  of the cut-off  function by
{\DEQS
h_1(\bu,t)=h(\bu,t)=\sup_{0\le s\le t}|\bu(s)|_{\nsH}.
\EEQS}
The cut-off function is given by the functions $\phi_\kk$, defined by the relation (\ref{2.5})
\DEQS
\Theta_\kappa(\bu,t)&=&\phi_\kappa(h(\bu,t)),
\quad t\in[0,T].
\EEQS
 Define for some real number $R>0$,  the set $\Xcal_\kappa^{\MA}=\Xcal_\kappa^\MA(R)$, where
\begin{equation}\label{defXR}
\Xcal^\kappa_{\Afrak}(R):=\lk\{\xi\in \CM^{m^\ast}_\MA(\X):\
 \EE\sup_{0\le s\le T}|\xi(s)|_{H^{-1}_2}^{2(q+1)}+ \EE\left(\int_0^T |\xi(s)|_{L^{q+1}}^{q+1}\, ds \right)^{2(q+1)}\le R
\rk\}.
\end{equation}

\noindent With $A$ and $B$ defined in \eqref{AB}, the condition reads
$\EE\,\mathcal A(\xi)^{2(q+1)}+\EE\,\mathcal B(\xi)^{(q+1)^2}\le R$.

\noindent Let us analyse the \emph{convexity, closedness and boundedness} of $\Xcal^\kappa_{\Afrak}(R)$: The functionals
$\xi\mapsto\EE\,\mathcal A(\xi)^{2(q+1)}$ and $\xi\mapsto\EE\,\mathcal B(\xi)^{(q+1)^2}$ are convex,
since $\mathcal A$ and $\mathcal B$ are seminorms and $x\mapsto x^p$ is convex and increasing on
$[0,\infty)$. They are lower semicontinuous on $\CM^{m^\ast}_\MA(\X)$: if $\xi_n\to\xi$ in
$\CM^{m^\ast}_\MA(\X)$, then along a subsequence $\xi_n(\omega,t)\to\xi(\omega,t)$ in
$H^{-1/3}_2(\CO)$ for $\PP\otimes dt$-a.e.\ $(\omega,t)$. Since $|\cdot|_{H^{-1}_2}$ and
$|\cdot|_{L^{q+1}}$ are lower semicontinuous with respect to this convergence, we get
$\mathcal A(\xi)\le\liminf_n\mathcal A(\xi_n)$ (the supremum understood as essential supremum) and
$\mathcal B(\xi)\le\liminf_n\mathcal B(\xi_n)$ a.s., and the lemma of Fatou gives the claim. Hence
$\Xcal^\kappa_\Afrak(R)$ is convex and closed. Membership in $\Xcal^\kappa_\Afrak(R)$ depends only
on the law of $\xi$, and by the remark after (\ref{hatdefined}) the set is invariant under $\Pro_\ell$.
By (\ref{interpolXp}) with $p=2(q+1)$ and the H\"older inequality,
\begin{align}
\EE\|\xi\|^{2(q+1)}_\X&\le C\,\EE\,\mathcal A(\xi)^{2(q+1)}+C\,\EE\Big(\int_0^T|\xi|^{q+1}_{L^{q+1}}dt\Big)^{2}
\le C R+C\Big(\EE\Big(\int_0^T|\xi|^{q+1}_{L^{q+1}}dt\Big)^{q+1}\Big)^{\frac2{q+1}}\notag\\
&\le C\big(R+R^{\frac2{q+1}}\big),\qquad \xi\in\Xcal^\kappa_\Afrak(R).\label{Xbounded}
\end{align}
So $\Xcal^\kappa_\Afrak(R)$ is bounded in $\CM^{m_1}_\Afrak(\X)$ with $m_1:=2(q+1)>m^\ast$, in
particular in $\CM^{m^\ast}_\Afrak(\X)$.

Let us collect the constants used in this proof: $q>3$, $s^{\ast\ast}=\frac13$, $r^\ast=2$,  $m^\ast=q+1$, $s_0=\frac14$, $\varrho=\frac1{2(q+1)}$,  $m_1=2(q+1)$, and  $m_0=q+1$. 

\begin{remark}\label{rem:parameters}
The remaining constants of Theorem~\ref{meta} are
$$
m_0=q+1,\qquad \X'=L^{q+1}\big(0,T;H^{-1/4}_2(\CO)\big)\cap W^{\varrho,q+1}(0,T;V^\ast),\qquad
\varrho=\frac1{2(q+1)},\qquad m_1=2(q+1).
$$
The following table lists all exponents, the constraints they have to satisfy, and where these
constraints are used.
\begin{center}
\begin{tabular}{lll}
constraint & satisfied by & used in\\ \hline
$q>3$ & Assumption~\ref{assum} & invariance ($R^{4/(q+1)}<R$), Prop.~\ref{continuity}, Claim~\ref{claim3.9}\\
$0<s^{\ast\ast}<\frac12$ & $s^{\ast\ast}=\frac13$ & (D): duality without boundary terms\\
$2\le r^\ast\le r_{s^{\ast\ast}}=\frac{6(q+1)}{q+5}$ & $r^\ast=2$ & (\ref{interpolX})\\
$\frac{2(q+1)}{q-1}\le m^\ast\le m_{s^{\ast\ast}}=\frac32(q+1)$ & $m^\ast=q+1$ & Prop.~\ref{continuity}(i) and (\ref{interpolX})\\
$0<s_0<s^{\ast\ast}$ & $s_0=\frac14$ & Claim~\ref{claim3.9}\\
$\varrho(q+1)<1$ & $\varrho=\frac1{2(q+1)}$ & Claim~\ref{claim3.9}\\
$m_1>m^\ast$, $\frac32(q+1)<m_1$ & $m_1=2(q+1)$ & (\ref{Xbounded}), Prop.~\ref{continuity}(ii)\\
$m_0>1$ & $m_0=q+1$ & Step~(I) of the proof of Theorem~\ref{meta}
\end{tabular}
\end{center}
Here $r_\theta$ and $m_\theta$ are as in Lemma~\ref{interpolation_11}. More generally, any
$0<s^{\ast\ast}<\frac12$, $2\le r^\ast\le r_{s^{\ast\ast}}$ and
$\frac{2(q+1)}{q-1}\le m^\ast\le\frac{q+1}{1-s^{\ast\ast}}$ works with the same proof, with the obvious
changes of the moment exponents. For $r^\ast>2$ one uses the duality
$H^{-s^{\ast\ast}}_{r^\ast}\times H^{s^{\ast\ast}}_{(r^\ast)'}$ and $W^{1,(r^\ast)'}$ in (D).
\end{remark}

\item
We will show that the framework of Step (I) satisfies the Assumptions of Theorem \ref{meta}.  For this aim,  let us consider for any integer $\kk\geq 1$, the mapping $$\CV_{\MA,\mathbb{W}}^\kappa:\CM^{m^\ast}_\MA(\X)\to \CM^{m^\ast}_\MA(\X),$$ defined by for $\xi \in \CM^{m^\ast}_\MA(\X)$,  $\CV_{\MA,\mathbb{W}}^\kappa(\xi):=n_\kappa$, where the triple $(n_\kappa,\ckk,\bu_\kappa)$ solves the following system.%
\DEQSZ
\label{a1}
\lk\{\barray
d{n}_\kappa(t) & =& \lk(r_n\Delta   |n_\kappa(t)|^{q-1}n_\kappa(t)+\theta \xi(t)- \chi \Div( \xi(t) \nabla \ckk(t))\rk)\, dt
\\ && {}-
\Theta_\kappa(\bu_\kappa,t) \bu_\kappa (t) \cdot\nabla \xi(t)\, dt+g_{\gamma_1}(n_\kappa(t)) \, dW_1(t),\phantom{\Big|}
\\
d{\ckk}(t)& =&\lk( r_c \Delta \ckk(t)  -\alpha \ckk(t)+ \beta
\xi(t) \rk)\, dt
\\ &&{}{-} \Theta_\kappa(\bu_\kappa,t)\bu_\kappa (t) \cdot\nabla \ckk(t)\, dt+ g_{\gamma_2}(\ckk(t))\, dW_2(t),\phantom{\Big|}
\\
 d \bu _\kappa(t){+r_{\bu} A \bu_\kappa(t)\,dt}&=&{\Pi\big(\xi(t)\star\Phi\big)} dt+ \sigma _{\gamma_3} dW_3(t),\earray\rk.
\EEQSZ
with initial condition $n_\kappa(0)=n_0$, $\ckk(0)=c_0$, and $\bu_\kappa(0)=\bu_0$.
\begin{remark}
Note, that $\ckk$ and $\bu_\kappa$ {do} not depend on $\nk$ but on $\xi$. In particular, the system \eqref{a1} is split.
Since $\operatorname{div}\bu_\kappa=0$, the transport term can be written as
$\bu_\kappa\cdot\nabla\xi=\operatorname{div}(\bu_\kappa\xi)$. We use both forms below.
\end{remark}

\textbf{Verification of conditions \eqref{metawellposed}  and \eqref{metainvariant} of Theorem \ref{meta}:} 
Here, we will show that the mapping $\CV_{\MA,\mathbb{W}}^\kappa$ is well defined and there exists a number $R>0$ such that $\CV_{\MA,\mathbb{W}}^\kappa$  maps any element of $ \CX_\MA(R)$ into itself.
 In particular, to verify assumption \eqref{metawellposed},  we will show that for any $R>0$, any $\kappa\in\NN$, any initial conditions $(n_0,c_0,\bu_0)$ satisfying estimates \eqref{3.7}
 and any $\xi\in  \CM^{m^\ast}_\MA(\X)\cap\CX_\kappa(\MA)$, there exists a triple of processes $(n_\kappa,\ckk,\bu_\kappa)$  solving system
 \eqref{a1} with initial and boundary conditions \eqref{init}.
By the a priori estimates, which we will prove in the following claim,  condition \eqref{metainvariant}  of Theorem \ref{meta} follows.

\begin{claim}\label{existencev}
Let Assumption \ref{assum} be satisfied and let  $R>0$ be arbitrary. For
  any $\xi\in \CM^{m^\ast}_{\mathfrak{A}}(\mathbb{X})\cap\CX(R)$, there exists a unique triple of processes $(n_\kappa,\ckk,\bu_\kappa)$  solving system \eqref{a1} such that $n_\kappa\in \CM^{m^\ast}_{\mathfrak{A}}(\mathbb{X})$.
In addition, for any $\kappa\in\NN$, there exists some constant $C=C(\kappa,m^*,r^*,T,q)$ such that the following estimates are satisfied:
\begin{equation}
	\EE	\sup_{0\le s\le  {T} }	\abs{\ukk(s)}^{8\frac{m^*}{r^*}}_\nsH+r_u^{4\frac{m^*}{r^*}}\EE\left(\int_0^T\abs{\ukk(s)}^2_\nsV ds\right)^{4\frac{m^*}{r^*}}\leq C\EE\abs{\bu_0}^{8\frac{m^*}{r^*}}_\nsH+ CR^\frac{4}{q+1}+C,\label{3.14*}
\end{equation}
\begin{equation}
	\EE\sup_{0\le t\le T}\abs{c_k(t)}^\frac{8m^*}{r^*}_{H_2^1}+C^\frac{4m^*}{r^*}_1\EE\left(\int_0^T\abs{ \ck(s)}^2_{H^2_2}ds\right)^\frac{4m^*}{r^*}\leq C\EE\abs{c_0}^\frac{8m^*}{r^*}_{H_2^1}+CR^\frac{4}{q+1},\label{3.12}
\end{equation}
 \begin{align}
	&\frac{1}{2}\EE\sup_{0\le t\le T}\	|\nk(t)|_{H^{-1}_2}^\frac{4m^*}{r^*}+(2r_n)^\frac{2m^*}{r^*}\EE\left(\int_0^{T} \abs{\nk(s)}^{q+1}_{L^{q+1}} ds\right)^\frac{2m^*}{r^*}\label{3.12*}\\
	&\leq C\EE|n_0|_{H^{-1}_2}^\frac{4m^*}{r^*}+C\EE\abs{\bu_0}^{8\frac{m^*}{r^*}}_\nsH+C\EE\abs{c_0}^\frac{8m^*}{r^*}_{H_2^1}+C+CR^\frac{2}{q+1}+CR^\frac{4}{q+1}.\notag
\end{align}

\del{and
\DEQS
\EE \|\nabla \ckk\|_{L^{r^\ast}(0,T;H^{1-s^{\ast\ast}}_{r^\ast})}^{m^\ast}\le C_5\lk(\EE|c_0|_{H^{2(1-\frac 1{m^\ast})-s^{\ast\ast}}_{r^\ast}}^{m^\ast}
+\EE \|\xi\|_{\BX}^{m^\ast}\rk),
\EEQS
\DEQS
\EE \|\ckk\|_{C([0,T];H^{\frac {2 q}{q+1}} _{q+1})}^{q+1},\EE \|\nabla \ckk\|_{L^{q+1}(0,T;H^1_{q+1})}^{q+1}\le C_6\lk(\EE|c_0|_{H^{\frac {2 q}{q+1}}_2}
^{q+1}+\EE \|\xi\|_{L^{q+1}(0,T;L^{q+1})}^{q+1}\rk),
\EEQS
and
\DEQS
\EE\sup_{0\le s\le T} \|\ckk\|_{L^2}^2\le  C_7\lk(\EE|c_0|_{H^{-1}_2}
^{2}+\EE \|\xi \|_{H^{-1}_2}^{2}\rk),
\EEQS}
where $C_1$ is a positive constant to be given below.
\end{claim}

From the estimates in Claim \ref{existencev} follows assumption \eqref{metainvariant} of Theorem \ref{meta}.
\noindent To see this, let us, for the time being, assume that Claim \ref{existencev} is true. With
the values of Step~(I), estimate (\ref{3.12*}) reads
$$
\tfrac12\,\EE\,\mathcal A(n_\kappa)^{2(q+1)}+(2r_n)^{q+1}\,\EE\,\mathcal B(n_\kappa)^{(q+1)^2}
\le C_0+C\big(R^{\frac2{q+1}}+R^{\frac4{q+1}}\big),
$$
where $C_0=C\big(1+\EE|n_0|_{H^{-1}_2}^{4\frac{m^*}{r^*}}+\EE\abs{\bu_0}^{8\frac{m^*}{r^*}}_\nsH+\EE\abs{c_0}^{8\frac{m^*}{r^*}}_{H_2^1}\big)$
depends on $\kappa$ and on the moments in (\ref{3.7}), but not on $R$. The proof of Claim~\ref{existencev}
uses only $\frac{m^\ast}{r^\ast}>1$ and $\frac4{q+1}<1$; for instance, in the $c$-estimate,
$$
\EE\Big(\int_0^T|\xi|^2_{L^2}ds\Big)^{2(q+1)}
\le C\,\EE\Big(\int_0^T|\xi|^{q+1}_{L^{q+1}}ds\Big)^{4}
\le C\Big(\EE\Big(\int_0^T|\xi|^{q+1}_{L^{q+1}}ds\Big)^{q+1}\Big)^{\frac4{q+1}}\le CR^{\frac4{q+1}},
$$
where $4\le q+1$ was used. Because $\frac2{q+1}<1$ and $\frac4{q+1}<1$, there is $R_\kappa$ with
$C_0+C(R^{\frac2{q+1}}+R^{\frac4{q+1}})\le \big(2+(2r_n)^{-(q+1)}\big)^{-1}R$ for all $R\ge R_\kappa$. For these $R$ and
any $\xi\in \CX_\MA(R)$ we have $\CV_\MA^\kappa(\xi)\in  \CX_\MA(R)$.
In particular, the operator $\CV_\MA^\kappa$ maps  $  \CX_\MA(R)$ into itself.

\begin{proof}[Proof of Claim \ref{existencev}]
Let us fix $R>0$ and $\xi\in \CX_\MA(R)$.
We will start with the last estimate and end with the first one.
In the first step, we consider the system
\begin{align}
\label{a3}
& d \bu _\kappa(t) + r_{\bu}A \bu_\kappa(t)dt
=\Pi\big(\xi(t)\star\Phi\big) dt+  \sigma_{\gamma_3} \, dW_3(t),
\end{align}
  with initial condition $ \bu(0)=\bu_0$.
Here, the existence and uniqueness of a solution are given by standard methods; see, e.g., \cite[Section 13.11.1]{DaPrZa:2nd}.

\noindent We use the norm $|\bv|_\nsV:=|A\bv|_{\mathbf L^2}$ on $\nsV$, which is equivalent to the
$(H^2(\bo))^2$-norm. Since $(\nabla\bv,\nabla\bw)_{\mathbf L^2}=(A\bv,\bw)$ for $\bv\in\nsV$,
$\bw\in\nsH$, the It\^o formula for the process $t\mapsto\abs{\ukk(t)}^2_\nsH$ gives
\begin{align*}
&\abs{\ukk(t)}^2_\nsH+2r_u\int_0^t\abs{\ukk(s)}^2_\nsV ds=\abs{\bu_0}^2_\nsH+2\int_0^t\big(\Pi(\xi(s)\star\Phi) ,A\ukk(s)\big)ds
\\
&{}+t\abs{\sigma_{\gamma_3}}^2_{L_{HS}(H,\nsH)}+2\int_0^t(\nabla\ukk(s),\nabla\sigma_{\gamma_3}\, dW_3(s)).
\end{align*}
Here we used $|\sigma_{\gamma_3}|_{L_{HS}(H,\nsH)}<\infty$, which holds by (A1). We have
$2(\Pi(\xi\star\Phi),A\ukk)\le r_u|\ukk|^2_\nsV+r_u^{-1}|\xi\star\Phi|^2_{\mathbf L^2}$. By (A2),
$\xi\star\phi_i=e^{\delta_i\Delta}\xi$, and by the spectral theorem
$|e^{\delta\Delta}|_{L(H^{-1}_2,L^2)}=\sup_k\lambda_k^{1/2}e^{-\delta\lambda_k}\le(2e\delta)^{-1/2}$. Hence
\begin{equation}\label{xiPhi}
|\xi\star\Phi|^2_{\mathbf L^2}=\sum_{i=1}^2|e^{\delta_i\Delta}\xi|^2_{L^2}\le C(\delta_1,\delta_2)\,|\xi|^2_{H^{-1}_2}.
\end{equation}
This implies that
\begin{align*}
\lqq{
\sup_{0\le s\le  {T} }	\abs{\ukk(s)}^2_\nsH+r_u\int_0^T\abs{\ukk(s)}^2_\nsV ds}
\\
&\leq\abs{\bu_0}^2_\nsH+ C(\delta_1,\delta_2,r_u)\int_0^T|\xi(s)|^2_{H^{-1}_2}ds
+T\abs{\sigma_{\gamma_3}}^2_{L_{HS}(H,\nsH)}+2\sup_{0\le t\le T}\abs{\int_0^t(\nabla\ukk(s),\nabla\sigma_{\gamma_3}\, dW_3(s))}.
\end{align*}

We then infer that
\begin{align*}
&\EE	\sup_{0\le s\le  {T} }	\abs{\ukk(s)}^{2p\frac{m^*}{r^*}}_\nsH+r_u^{p\frac{m^*}{r^*}}\EE\left(\int_0^T\abs{\ukk(s)}^2_\nsV ds\right)^{p\frac{m^*}{r^*}}
\\
&\leq\EE\abs{\bu_0}^{2p\frac{m^*}{r^*}}_\nsH+ C(\delta_1,\delta_2)\EE\left(\int_0^T\abs{\xi(s)}^2_{H^{-1}_2}ds\right)^{p\frac{m^*}{r^*}}+C
\\
&\qquad {}+ C\EE\left(\int_0^T\abs{\ukk(s)}^2_{\nsH}ds\right)^{p\frac{m^*}{r^*}}+C\EE\sup_{0\le t\le T}\abs{\int_0^t(\nabla\ukk(s),\nabla\sigma_{\gamma_3} ) \, dW_3(s)}^{p\frac{m^*}{r^*}}.
\end{align*}
Since $q>3$, we know that  $\frac{4}{q+1}<1$. Therefore, by the H\"older inequality  and the embedding  of $L^{q+1}(\CO)$ into $H^{-1}_2(\CO)$,  we derive that
\begin{align}
C\EE\left(\int_0^T\abs{\xi(s)}^2_{H^{-1}_2}ds\right)^{p\frac{m^*}{r^*}}&\leq C(T)C\EE\left(\int_0^T\abs{\xi(s)}^{q+1}_{L^{q+1}}ds\right)^{\frac{2pm^*}{(q+1)r^*}}\notag
\\
&\leq C(T)\left(\EE\left(\int_0^T\abs{\xi(s)}^{q+1}_{L^{q+1}}ds\right)^{2\frac{m^*}{r^*}}\right)^{\frac{p}{q+1}}\notag
\\
&\leq C(T)R^\frac{p}{q+1}.\notag
\end{align}
By the Burkholder--Davis--Gundy inequality and Young inequality, we obtain
\DEQS
\lqq{\EE\sup_{0\le t\le T}\abs{\int_0^t(\nabla\ukk(s),\nabla\sigma_{\gamma_3} ) dW_3(s)}^{p\frac{m^*}{r^*}}
}
\\
&\le&
\EE\left(\int_0^T\abs{\ukk(s)}^2_\nsH\abs{\sigma_{\gamma_3}}^2_{L_{HS}(H,\nsH)} ds\right)^{\frac p2\frac{m^*}{r^*}}
=
\abs{\sigma_{\gamma_3}}^{p\frac{m^*}{r^*}}_{L_{HS}(H,\nsH)} \EE\left(\int_0^T\abs{\ukk(s)}^2_\nsH ds\right)^{\frac p2\frac{m^*}{r^*}}\notag
\\
&\leq& C\EE\int_0^T\abs{\ukk(s)}^{p\frac{m^*}{r^*}}_\nsH ds\leq C\EE\int_0^T\abs{\ukk(s)}^{2p\frac{m^*}{r^*}}_\nsH ds+C.\notag
\EEQS
In addition, the H\"older inequality gives
$$\EE\left(\int_0^T\abs{\ukk(s)}^2_\nsH ds\right)^{p\frac{m^*}{r^*}}\leq \EE\int_0^T\abs{\ukk(s)}^{2p{\frac{m^*}{r^*}}}_\nsH ds.$$
Combining these four  last inequalities, we arrive at
\begin{align*}
&	\EE	\sup_{0\le s\le  {T} }	\abs{\ukk(s)}^{2p\frac{m^*}{r^*}}_\nsH+r_u^{p\frac{m^*}{r^*}}\EE\left(\int_0^T\abs{\ukk(s)}^2_\nsV ds\right)^{p\frac{m^*}{r^*}}
\\
&\leq\EE\abs{\bu_0}^{2p\frac{m^*}{r^*}}_\nsH+ CR^\frac{p}{q+1}+C+C\EE \int_0^T\abs{\ukk(s)}^{2p\frac{m^*}{r^*}}_\nsH ds.
\end{align*}
Applying the Gronwall Lemma, we obtain the inequality (\ref{3.14*}).

Next, we consider the equation
\begin{align}
\label{sysv}
d{\ck }(t)& =\lk( r_c  \Delta \ck (t)  -\alpha \ck (t)+ \beta
\xi(t)\rk)\, dt
\\\notag
&{} - \Theta_\kappa (\bu_\kappa,t)\bu_\kappa(t)\cdot \nabla c_\kappa(t) \, dt +\ g_{\gamma_2}(\ck (t)) dW_2(t),\quad \ck (0)=c_0.
\end{align}
To show the existence and uniqueness of a solution, we will verify that the assumptions (H1), (H2$^{\prime}$), (H3), and  (H4$^{\prime}$) of Theorem 5.1.3 in \cite[p.\ 125] {weiroeckner} are satisfied.
\noindent Since $c$ satisfies the Neumann condition, we use the Gelfand triple
$$
H^1_2(\CO)\subset  L^2(\CO)\subset \big(H^1_2(\CO)\big)',
$$
with $\langle c,w\rangle=(c,w)$ for $c\in L^2(\CO)$, $w\in H^1_2(\CO)$. (Note that
$(H^1_2(\CO))'$ is not the space $H^{-1}_2(\CO)=(H^1_0(\CO))'$ of the standing conventions.)
For $t\in [0,T]$, we define $A_1(t,\cdot):H^{1}_2(\CO)\to (H^{1}_2(\CO))'$ by
$$
\langle A_1(t,c),w\rangle:=-r_c(\nabla c,\nabla w)-\alpha(c,w)+\beta(\xi(t),w)
-\Theta_\kappa (\bu_\kappa,t)(\bu_\kappa(t)\cdot \nabla c,w),\quad c,w\in H^{1}_2(\CO),
$$
where $\bu_\kappa$ solves \eqref{a3}.

(H1): $A_1(t,\cdot)$ is affine, so $\lambda\mapsto\langle A_1(t,c_1+\lambda c_2),w\rangle$ is continuous.

(H2$'$): Since $\operatorname{div}\bu_\kappa=0$ and $\bu_\kappa=0$ on $\partial\CO$, we have
$\Theta_\kappa(\bu_\kappa,t)(\bu_\kappa\cdot\nabla v,v)=\frac12\Theta_\kappa(\bu_\kappa,t)(\bu_\kappa\cdot\nabla v^2,1)=0$
for $v\in H^1_2(\CO)$. With $v=c_1-c_2$ and (\ref{eq:Hil-Schi2}),
\begin{align*}
\langle A_1(t,c_1)-A_1(t,c_2),v\rangle+\abs{g_{\gamma_2}(v)}^2_{L_{HS}(H,L^2)}
&=-r_c|\nabla v|^2_{L^2}-\alpha|v|^2_{L^2}+\abs{g_{\gamma_2}(v)}^2_{L_{HS}(H,L^2)}\\
&\le-\alpha_0|v|^2_{H^1_2}+C|v|^2_{L^2}\le C|v|^2_{L^2}.
\end{align*}

(H3): Using again $(\bu_\kappa\cdot\nabla c,c)=0$ and $2\beta(\xi,c)\le|c|^2_{L^2}+\beta^2|\xi|^2_{L^2}$,
$$
2\langle A_1(t,c),c\rangle+\abs{g_{\gamma_2}(c)}^2_{L_{HS}(H,L^2)}\le-2\alpha_0\abs{c}_{H^1_2}^2+C\abs{c}_{L^2}^2+f_1(t),
\qquad f_1(t):=\beta^2\abs{\xi(t)}^2_{L^2},
$$
i.e.\ (H3) holds with $\theta_1=2\alpha_0$ and exponent $2$. Since $\xi\in\Xcal^\kappa_\Afrak(R)$, $f_1\in L^1(\Omega\times[0,T])$.

(H4$'$): For $c,w\in H^1_2(\CO)$, using $\Theta_\kappa(\bu_\kappa,t)|\bu_\kappa(t)|_{\mathbf L^4}\le C\Theta_\kappa(\bu_\kappa,t)|\bu_\kappa(t)|_{\nsH}\le 2C\kappa$,
$$
\langle A_1(t,c),w\rangle\le\big(\max(r_c,\alpha)|c|_{H^1_2}+\beta|\xi(t)|_{L^2}+C\kappa|\nabla c|_{L^2}\big)|w|_{H^1_2},
$$
hence $|A_1(t,c)|^2_{(H^1_2)'}\le C(\kappa)|c|^2_{H^1_2}+C f_1(t)$. This is (H4$'$) with exponent $0$ in front
of $|c|_{L^2}$, i.e.\ without a factor $(1+|c|^2_{L^2})$.

With this exponent $0$, \cite[Theorem 5.1.3]{weiroeckner} applies with $p=2$ and $f_1\in L^1$, and gives a
unique solution of (\ref{sysv}) with $\EE\sup_{0\le s\le T}|c_\kappa(s)|^2_{L^2}<\infty$ and
$c_\kappa\in C([0,T];L^2(\CO))$ a.s. The higher moments follow from the a priori estimate
(\ref{3.12}) below (by a standard localisation with stopping times).
Now we prove the a priori estimate for $c_\kappa$. For this purpose,   we apply the It\^ o formula to the function $\phi(c_k):=\abs{c_k(t)}^2_{H_2^1}$ to derive an estimate for
$\abs{c_k(t)}^2_{H_2^1}$\footnote{The It\^o formula in $H^1_2$ is justified in the Gelfand triple
$D(\Delta_N)\subset H^1_2(\CO)\subset L^2(\CO)$, where $\Delta_N$ is the Neumann Laplacian, or,
equivalently, on Galerkin approximations; the resulting bound shows $c_\kappa\in L^2(0,T;H^2_2(\CO))$ a.s.}.
By doing so, we obtain
\begin{align}
\abs{c_k(t)}^2_{H_2^1}&=\abs{c_0}^2_{H_2^1}+2\int_0^t( r_c \Delta \ckk(s)  -\alpha \ckk(s), \ckk(s))_{H_2^1}ds+ 2\beta
\int_0^t( \xi(s) ,c_k(s))_{H_2^1}ds\notag\\
&\quad{}-2\int_0^t( \Theta_\kappa(\bu_\kappa,s)\bu_\kappa (s) \cdot\nabla \ckk(s),c_k(s))_{H_2^1}ds\label{3.14}\\
&\quad{}+\int_0^t\abs{g_{\gamma_2}(\ckk(s)) }^2_{L_{HS}(H;H_2^1)}ds+2\int_0^t(c_k(s),g_{\gamma_2}(\ckk(s)) dW_2(s))_{H_2^1}\notag.
\end{align}
Using integration-by-parts and the Young inequality, we derive that
\begin{align}
2( r_c \Delta \ckk  -\alpha \ckk, \ckk)_{H_2^1}&=-2(r_c+\alpha)\abs{\nabla \ck}^2_{L^2}-2\alpha\abs{ \ck}^2_{L^2}-2r_c\abs{\Delta \ck}^2_{L^2}\notag\\
&\leq-2\alpha\abs{ \ck}^2_{H^1_2}-2r_c\abs{\Delta \ck}^2_{L^2},\notag
\end{align}
and
\begin{align}
2\beta( \xi ,c_k)_{H_2^1}&=2\beta( \xi ,c_k)_{L^2}-2\beta( \xi ,\Delta c_k)_{L^2}\notag\\
&\leq C\abs{ \xi}^2_{L^2}+\eps_1\abs{ \ck}^2_{H^1_2}+\eps_1\abs{ \Delta\ck}^2_{L^2}
\leq C\abs{ \xi}^2_{L^2}+\eps_1\abs{ \ck}^2_{H^2_2}.\notag
\end{align}
In a similar way, using in addition the fact that $( \Theta_\kappa(\bu_\kappa)\bu_\kappa  \cdot\nabla \ckk,c_k)_{L^2}=0$ and the embedding of $\nsH$ into $\mathbf{L}^4(\CO)$,  we have
\begin{align}
-2( \Theta_\kappa(\bu_\kappa)\bu_\kappa  \cdot\nabla \ckk,c_k)_{H_2^1}&={-2}( \Theta_\kappa(\bu_\kappa)\bu_\kappa  \cdot\nabla \ckk,c_k)_{L^2}+2( \Theta_\kappa(\bu_\kappa)\bu_\kappa  \cdot\nabla \ckk,\Delta c_k)_{L^2}\notag\\
&\leq 2 \Theta_\kappa(\bu_\kappa)\abs{ \bu_\kappa}_{\mathbf{L}^4}\abs{ \nabla \ck}_{L^4}\abs{ \Delta\ck}_{L^2}
\leq  C \Theta_\kappa(\bu_\kappa)\abs{ \bu_\kappa}_{\nsH}\abs{ \ck}^\frac{1}{2}_{H_2^1}\abs{ \ck}^\frac{1}{2}_{H_2^2}\abs{ \Delta\ck}_{L^2}\notag\\
&\leq C(\kappa)\abs{ \ck}^\frac{1}{2}_{H_2^1}\abs{ \ck}^\frac{3}{2}_{H_2^2}
\leq C(\kappa)\abs{ \ck}_{H_2^1}^2+\eps_1\abs{ \ck}^2_{H_2^2}\notag.
\end{align}
By (\ref{eq:Hil-Schi2H1}) (which holds for $\gamma_2>2$, see Appendix~\ref{Appendix_A_noise}),
\begin{align}
\abs{g_{\gamma_2}(\ckk) }^2_{L_{HS}(H;H_2^1)}\leq C\abs{ \ck}^2_{H^1_2}.
\end{align}
By using these estimates, we infer from the equality (\ref{3.14}) that
\begin{align}
&\abs{c_k(t)}^2_{H_2^1}+2\int_0^t(\alpha\abs{ \ck(s)}^2_{H^1_2}+r_c\abs{\Delta \ck(s)}^2_{L^2})ds
\\
&\leq\abs{c_0}^2_{H_2^1}+2\eps_1\int_0^t\abs{ \ck(s)}^2_{H^2_2}ds+C\int_0^t\abs{ \xi(s)}^2_{L^2}ds+C(\kappa)\int_0^t\abs{ \ck(s)}^2_{H^1_2}ds\notag\\
&\quad+2\int_0^t(c_k(s),g_{\gamma_2}(\ckk(s)) dW_2(s))_{H_2^1}\notag.
\end{align}
There exists a constant $C_0>0$ such that $ \abs{ \ck}^2_{H^2_2}\leq C_0(\abs{ \ck(s)}^2_{H^1_2}+\abs{ \Delta\ck(s)}^2_{L^2})$ (elliptic regularity for the Neumann problem). Therefore taking $\eps_1:=\frac{\alpha_0}{2C_0}$ and putting $C_1:=\frac{\alpha_0}{C_0}$
in the previous inequality, we derive that
\begin{align}
\lqq{	\abs{c_k(t)}^2_{H_2^1}+C_1\int_0^t\abs{ \ck(s)}^2_{H^2_2}ds
	\leq\abs{c_0}^2_{H_2^1}+C\int_0^t\abs{ \xi(s)}^2_{L^2}ds}
	\\
	&{}\quad {}+C(\kappa)\int_0^t\abs{ \ck(s)}^2_{H^1_2}ds+2\int_0^t(c_k(s),g_{\gamma_2}(\ckk(s)) dW_2(s))_{H_2^1}\notag.
\end{align}
Taking the supremum over $[0,T]$, the  exponent $\frac p{2}(q+1)$ and the expectation, we arrive at
\begin{align}\label{3.15}
&	\EE\sup_{0\le t\le T}\abs{c_k(t)}^\frac{2pm^*}{r^*}_{H_2^1}+C^\frac{pm^*}{r^*}_1\EE\left(\int_0^T\abs{ \ck(s)}^2_{H^2_2}ds\right)^\frac{pm^*}{r^*}
\\
	&\leq C\EE\abs{c_0}^\frac{2pm^*}{r^*}_{H_2^1}+C(m^*,r^*)\EE\left(\int_0^T\abs{ \xi(s)}^2_{L^2}ds\right)^\frac{pm^*}{r^*}\notag
\\\notag 
&\quad {}+C(\kappa,m^*,r^*)\EE\left(\int_0^T\abs{ \ck(s)}^2_{H^1_2}ds\right)^\frac{pm^*}{r^*}
	\\
	&\quad {} +C(m^*,r^*)\EE\sup_{0\le t\le T}\abs{\int_0^t(c_k(s),g_{\gamma_2}(\ckk(s)) dW_2(s))_{H_2^1}}^\frac{pm^*}{r^*}\notag.
\end{align}
We note that since $q>3$ then $\frac{4}{q+1}<1$ and therefore we obtain by  the H\"older inequality, for $p\le q+1$,
\begin{align}
\lqq{ \EE\left(\int_0^T\abs{ \xi(s)}^2_{L^2}ds\right)^\frac{pm^*}{r^*}
}
\\
&\leq C(T)\EE\left(\int_0^T\abs{ \xi(s)}^{q+1}_{L^{q+1}}ds\right)^\frac{2pm^*}{r^*(q+1)}\notag
\leq C(T)\left(\EE\left(\int_0^T\abs{ \xi(s)}^{q+1}_{L^{q+1}}ds\right)^{2\frac{m^*}{r^*}}\right)^\frac{p}{q+1}\notag.
\end{align}
Since $\xi\in \CX_\MA^\kappa$, we get
$$
\EE\left(\int_0^T\abs{ \xi(s)}^2_{L^2}ds\right)^\frac{pm^*}{r^*}
\le C(T)R^\frac{p}{q+1}.
$$
By the  H\"older inequality, we obtain
\begin{equation}
	\EE\left(\int_0^T\abs{ \ck(s)}^2_{H^1_2}ds\right)^\frac{pm^*}{r^*}\leq C(T)\EE\int_0^T\abs{ \ck(s)}^\frac{2pm^*}{r^*}_{H^1_2}ds.
\end{equation}
By the Burkholder--Davis--Gundy inequality, (\ref{eq:Hil-Schi2H1}) and the Young inequality, we have
\begin{align}
&C(m^*,r^*)\EE\sup_{0\le t\le T}\abs{\int_0^t(c_k(s),g_{\gamma_2}(\ckk(s)) dW_2(s))_{H_2^1}}^\frac{pm^*}{r^*}\notag
\\
&\leq C\EE\left(\int_0^T\abs{\ck(s)}^2_{H^1_2}\abs{g_{\gamma_2}(\ckk(s))}^2_{L_{HS}(H;H_2^1)}ds\right)^\frac{\frac p2 m^*}{r^*}
\leq C\EE\sup_{0\le t\le T}\abs{\ck(t)}^\frac{pm^*}{r^*}_{H^1_2}\left(\int_0^T\abs{\ckk(s)}^2_{H_2^1}ds\right)^\frac{\frac p2 m^*}{r^*}\notag
\\
&\leq\frac{1}{2}\EE\sup_{0\le t\le T}\abs{\ck(t)}^\frac{2pm^*}{r^*}_{H^1_2}+\C\EE\int_0^T\abs{ \ck(s)}^\frac{2pm^*}{r^*}_{H^1_2}ds\notag.
\end{align}
Combining these last  inequalities, we infer from the inequality (\ref{3.15}) that
\begin{align}
&\EE\sup_{0\le t\le T}\abs{c_k(t)}^\frac{2pm^*}{r^*}_{H_2^1}+C^\frac{pm^*}{r^*}_1\EE\left(\int_0^T\abs{ \ck(s)}^2_{H^2_2}ds\right)^\frac{pm^*}{r^*}
\\
&\leq{C}\EE\abs{c_0}^\frac{2pm^*}{r^*}_{H_2^1}+CR^{{\frac{p}{q+1}}}+C\EE\int_0^T\abs{ \ck(s)}^{\frac{2pm^*}{r^*}}_{H^1_2}ds,\notag
\end{align}
and the Gronwall lemma applied to $t\mapsto\EE\sup_{s\le t}|c_\kappa(s)|^{2pm^*/r^*}_{H^1_2}$ with $p=4$ implies (\ref{3.12}).
\medskip

Finally, we consider the system
\begin{align}
\label{sysn}
\lqq{ d{n}_\kappa(t)  = (r_n\Delta   |n_\kappa(t)|^{q-1}n_\kappa(t)+\theta\xi(t)
- \chi \Div( \xi(t) \nabla \ckk(t))) dt }
\\
&{-\Theta_\kappa(\bu_\kappa,t)}\ukk(t)\cdot\nabla \xi(t)\,dt+g_{\gamma_1}(\nk(t))  dW_1(t).
\end{align}
 We now show that there exists a
unique solution $\nk$ to the system \eqref{sysn} by
verifying  the assumptions (H1), (H2$^{\prime}$), (H3), and  (H4$^{\prime}$) of Theorem 5.1.3 in \cite[p.\ 125]{weiroeckner}. Let us consider the Gelfand triple
\begin{equation}
V\subset \CH\cong \CH^{\ast}\subset V^{\ast},\label{3.1}
\end{equation}
with $\CH:=H_{2}^{-1}(\CO)$,  $V:=L^{q+1}(\CO)$ and $V^{\ast}=(-\Delta)\big(L^\frac{{q+1}}{q}(\CO)\big)$, as in the standing conventions, see (\ref{Vstar}).
The duality $_{V^{\ast}}\langle \cdot, \cdot \rangle_V$ is defined as
\begin{equation}
_{V^{\ast}}\langle n,w \rangle_V
=\int_{\CO} ((-\Delta)^{-1}n)(x)\, w(x)\, dx
.\label{3.21}
 \end{equation}
We set for $t\in [0,T]$,
\DEQS
A(t,n):
=r_n\Delta n^{[q]}+\theta\xi(t)-\chi
\di(\xi(t) \nabla c_\kappa(t))-\Theta_\kappa(\bu_\kappa,t)\ukk(t)\cdot\nabla \xi(t),
\EEQS
where $c_\kappa$ solves \eqref{sysv} and $\bu_\kappa$ solves \eqref{a3}.

 First, let us consider (H1).
 Let $\lambda\in\RR$. We need to show that
$$
\lambda\mapsto {_{V^{\ast}}\langle  A(t,n_{1}+\lambda n_{2}),w\rangle_{V}}
$$
 is continuous on $\RR$, for any $n_{1},n_{2},w\in V$ and  $t\in[0,T]$.
As the map $\lambda \mapsto (n_{1}+\lambda n_{2})^{[q]}$ is continuous, it follows by the dominated convergence Theorem
that the mapping $\RR_0^+\ni
\lambda\mapsto \int_{\CO}(n_{1}+\lambda n_{2})^{[q]}w dx $
is continuous. Hence, the mapping $
\lambda\mapsto {_{V^{\ast}}\langle  A(t,n_{1}+\lambda n_{2}),w\rangle_{V}}$ is also continuous.

To show the local monotonicity, i.e. (H2$^{\prime}$),
let us fix $n_1,n_2\in V$ and  $t\in[0,T]$. Then,  using the fact that (see \cite[Page 87-88]{weiroeckner}),
$$ -\int_{\CO}\left[r_n(n_1^{[q]}-n_2^{[q]})(n_1-n_2)\right]\,dx
\leq0,$$
we have
\begin{align}\label{wasist}
&  \mathbin{_{V^{\ast}}\langle  A(t,n_1)- A(t, n_2),n_1-n_2\rangle_{V}}+|g_{\gamma_1}(n_1)-g_{\gamma_1}(n_2)|^2_{L_{HS}(H_1,\CH)}
\notag\\
&=-\int_{\CO}\left[r_n(n_1^{[q]}-n_2^{[q]})(n_1-n_2)\right]\,dx+|g_{\gamma_1}(n_1)-g_{\gamma_1}(n_2)|^2_{\CL_{2}(H_1,\CH)}\notag\\
&\le |g_{\gamma_1}(n_1)-g_{\gamma_1}(n_2)|^2_{L_{HS}(H_1,\CH)} \\
&\le  C
|n_1-n_2|_{\CH}^{2}\leq(f(t)+C)
|n_1-n_2|_{\CH}^{2},\notag
\end{align}
where $f(t)$ will be given later.

To verify the coercivity, i.e.\  (H3)
let us fix $n\in V$ and  $t\in[0,T]$. Then,
\begin{align*}
\lqq{ 2	\mathbin{_{V^{\ast}}\langle A(t,n),n \rangle_{V}}+|g_{\gamma_1}(n)|^2_{L_{HS}(H_1,\CH)}}
\\
&= - 2r_n |n|_{L^{q+1}}^{q+1}
-2\chi\mathbin{_{V^{\ast}}}\langle
\di(\xi(t) \nabla c_\kappa(t)),n \rangle_V-2\Theta_\kappa(\bu_\kappa,t)\mathbin{_{V^{\ast}}}\langle\nabla \xi(t)\cdot\bu_\kappa(t),n \rangle_V \notag
\\
&\quad {} +2\mathbin{_{V^{\ast}}}\langle\theta\xi(t),n \rangle_V
+|g_{\gamma_1}(n)|^2_{L_{HS}(H_1,\CH)}
\\
& \leq -2r_n |n|_V^{q+1}+C(\theta) |n|_{\CH}^{2}+|\xi|_{\CH}^{2}+\Theta_\kappa(\bu_\kappa,t)2\big|\mathbin{_{V^{\ast}}}\langle\nabla \xi(t)\cdot\bu_\kappa(t),n \rangle_V\big|
\\
&\quad{}-2 \chi\mathbin{_{V^{\ast}}}\langle
\di(\xi(t) \nabla c_\kappa(t)),n \rangle_V.
\end{align*}
Since $q>3$ then by the embedding of $V$ into $\CH$ and the Young inequality we have
\begin{align}
C(\theta)\abs{n}^2_\CH&\leq C(\theta)\abs{n}^{2}_V\leq r_n\abs{n}^{q+1}_V+C(\theta,q),\label{3.18*}
\end{align}
and
\begin{align}
\abs{\xi(t)}^2_\CH&\leq C\abs{\xi(t)}^{2}_V\leq \frac{1}{3}\abs{\xi(t)}^{q+1}_V+C(q).\label{3.18**}
\end{align}
Similarly, since $q>3$, by the H\"older inequality and Sobolev embedding,   we note that
\begin{align}\label{3.18}
\lqq{ -2\mathbin{_{V^{\ast}}}\langle \chi
\di(\xi(t) \nabla c_\kappa(t)),n \rangle_V
\leq 2\abs{\di(\xi(t) \nabla c_\kappa(t))}_{H^{-1}_2}\abs{n}_{H^{-1}_2}}
\notag\\
&\leq 2\abs{\xi(t)}_{L^4}\abs{\nabla c_\kappa(t)}_{L^4}\abs{n}_{H^{-1}_2}\notag\leq C \abs{\xi(t)}_{L^{q+1}}^2\abs{\nabla c_\kappa(t)}^2_{L^{q+1}}+\frac{1}{2}\abs{n}_{H^{-1}_2}^2\\
&\leq  \abs{\xi(t)}^{4}_{L^{4}}+C\abs{\nabla c_\kappa(t)}^{4}_{L^{4}}+\frac{1}{2}\abs{n}_{H^{-1}_2}^2\notag\\
&\leq \frac{1}{3} \abs{\xi(t)}^{q+1}_{L^{q+1}}+C\abs{\nabla c_\kappa(t)}^{4}_{L^{4}}+\frac{1}{2}\abs{n}_{H^{-1}_2}^2+C(q).\notag
\end{align}
By an integration by parts, we note that for all $w\in H^1_2(\bo)$,
\begin{align*}
& \mathbin{_{H^{-1}_2}}\langle \xi(t)\cdot\nabla \bu_\kappa(t) ,w \rangle_{H^1_2}
\\
&=\int_\bo\nabla\xi(t)\cdot \bu_\kappa (t) wdx=-\int_\bo\xi(t) \bu_\kappa (t) \cdot\nabla wdx\leq \abs{\xi(t) }_{L^4}\abs{\bu_\kappa(t) }_{\mathbf{L}^4} \abs{\nabla w}_{L^2}.
\end{align*}
This implies that $\abs{\xi(t)\cdot\nabla \bu_\kappa(t)}_{H^{-1}_2}\leq \abs{\xi(t) }_{L^4}\abs{\bu_\kappa(t) }_{\nsH}$  and therefore,
 \begin{align}
\lqq{
\Theta_\kappa(\bu_\kappa,t)2 \mathbin{_{V^{\ast}}}\langle\nabla \xi(t)\cdot\bu_\kappa(t),n \rangle_V}
\\
&\leq 2\Theta_\kappa(\bu_\kappa,t)\abs{\nabla \xi(t)\cdot\bu_\kappa(t)}_{H^{-1}_2}\abs{n}_{H^{-1}_2}\notag
\leq 2\Theta_\kappa(\bu_\kappa,t)\abs{\bu_\kappa(t) }_{\nsH}\abs{\xi(t) }_{L^4}\abs{n}_{H^{-1}_2}\notag
 \\
  &\leq C\Theta_\kappa(\bu_\kappa,t) \abs{\bu_\kappa(t) }_{\nsH}^2\abs{\xi(t) }^2_{L^{q+1}}+\frac{1}{2}\abs{n}_{H^{-1}_2}^2
  \leq 4C\kappa^2\abs{\xi(t) }^{2}_{L^{q+1}}+\frac{1}{2}\abs{n}_{H^{-1}_2}^2\notag
   \\
    &\leq \frac{1}{3}\abs{\xi(t) }^{q+1}_{L^{q+1}}+\frac{1}{2}\abs{n}_{H^{-1}_2}^2+C(q)\kappa^{\frac{2(q+1)}{q-1}}
    \label{3.19}.
 \end{align}
Here, we used the definition of the cut-off, i.e.\ $\Theta_\kappa(\bu_\kappa,t)|\bu_\kappa(t)|_{\nsH}\le2\kappa$, and the Young inequality with exponents $\frac{q+1}2$ and $\frac{q+1}{q-1}$.
By using {(\ref{3.18*}), (\ref{3.18**}),} (\ref{3.18}) {and (\ref{3.19})} 
we derive that
\begin{equation}
	2\mathbin{_{V^{\ast}}\langle A(t,n),n \rangle_{V}}+|g_{\gamma_1}(n)|^2_{L_{HS}(H_1,\CH)}
	\leq C_0 |n|_{\CH}^{2}-r_n |n|_V^{\alpha^*}+f(t),
\end{equation}
where $\alpha^*=q+1$, $C_0=\max(1,C(q)r_n^{\frac{q+1}{q}})$ and  $${f(t)=\abs{\xi(t)}^{q+1}_{L^{q+1}}+C\abs{\nabla c_\kappa(t)}^{4}_{L^{4}}+C(\theta,q,\kappa).}$$
Note that, since $\frac{m^*}{r^*}>1$, due to the estimates (\ref{3.12}) and (\ref{3.14*}), there exists a constant $C(R)>0$ such that
\begin{align}
\lqq{
	\EE\int_0^T\abs{\nabla  c_\kappa(s)}^{4}_{L^{4}}ds
	\leq C\EE \int_0^T \abs{\nabla c_\kappa(s)}^{2}_{L^{2}}\abs{ \nabla c_\kappa(s)}^{2}_{H^1_{2}}ds}
	\notag
\\
&\leq C\EE \sup_{0\le s\le  {T}  } \abs{c_\kappa(s)}^{2}_{H^1_{2}}\int_0^T\abs{ c_\kappa(s)}^{2}_{H^2_{2}}ds\notag
\\
&\leq C\left(\EE \sup_{0\le s\le  {T}  } \abs{ c_\kappa(s)}^{4}_{H^1_{2}}\right)^\frac{1}{2}\left(\EE\left(\int_0^T\abs{ c_\kappa(s)}^{2}_{H^2_{2}}ds\right)^2\right)^\frac{1}{2}\notag
\\
&\leq C\left(\EE \sup_{0\le s\le  {T}  } \abs{ c_\kappa(s)}^{4\frac{m^*}{r^*}}_{H^1_{2}}\right)^\frac{r^*}{2m^*}\left(\EE\left(\int_0^T\abs{ c_\kappa(s)}^{2}_{H^2_{2}}ds\right)^{4\frac{m^*}{r^*}}\right)^\frac{r^*}{4m^*}\notag\le  C(R)\notag.
\end{align}
\del{By the  H\"older inequality, we obtain
\begin{equation}
\EE\int_0^T\abs{\ukk(s)}^{4}_{\nsH}ds\leq T\EE \sup_{0\le s\le  {T}  }\abs{\ukk(s)}^{4}_{\nsH}\leq T\left(\EE \sup_{0\le s\le  {T}  }\abs{\ukk(s)}^{4^\frac{m^*}{r^*}}_{\nsH}\right)^\frac{r^*}{2m^*}\leq C(R).\notag
\end{equation}}
Observe, due to the fact that   $\xi\in\mathcal{X}_\MA^\kappa(R)$, it follows that $f$ belongs to  $L^1([0,T]\times \Omega;dt\otimes\PP)$.
 This proves (H3).

\medskip

The growth condition, i.e.\  (H4$^{\prime}$), can be verified similarly to the boundedness.
Let $n\in V$ and  $t\in[0,T]$. Then for all $w\in V$, by the H\"older inequality we have
\begin{align*}
\mathbin{_{V^{\ast}}\langle A(t,n),w\rangle_{V}}& \leq r_n\abs{n}^q_V\abs{w}_V +\mathbin{_{V^{\ast}}}\langle\theta\xi(t),w \rangle_V
-\mathbin{_{V^{\ast}}}\langle \chi
\di(\xi(t) \nabla c_\kappa(t)),w \rangle_V
\\
&\quad{} -\Theta_\kappa(\bu_\kappa,t)\mathbin{_{V^{\ast}}}\langle\nabla \xi(t)\cdot\bu_\kappa(t),w \rangle_V \notag
\\
&\leq r_n\abs{n}^q_V\abs{w}_V+\theta\abs{\xi(t)}_{H^{-1}_2}\abs{w}_{H^{-1}_2}+\abs{\di(\xi(t) \nabla c_\kappa(t))}_{H^{-1}_2}\abs{w}_{H^{-1}_2}\notag\\
&\quad+{} \Theta_\kappa(\bu_\kappa,t)\abs{\nabla \xi(t)\cdot\bu_\kappa(t)}_{H^{-1}_2}\abs{w}_{H^{-1}_2}.
\end{align*}
Using the continuous embedding of $V$ into $H^{-1}_2(\CO)$ we infer that
\begin{align*}
\lqq{\abs{ A(t,n)}_{V^\star}}
\\
&\leq r_n\abs{n}^q_V+C\theta\abs{\xi(t)}_{H^{-1}_2}+C\abs{\di(\xi(t) \nabla c_\kappa(t))}_{H^{-1}_2}+C\Theta_\kappa(\bu_\kappa,t)\abs{\nabla \xi(t)\cdot\bu_\kappa(t)}_{H^{-1}_2}
\\
&\leq r_n\abs{n}^q_V+C\theta\abs{\xi(t)}_{H^{-1}_2}+C\abs{\xi(t)}_{L^4}\abs{\nabla c_\kappa(t)}_{L^4}+C\Theta_\kappa(\bu_\kappa,t)\abs{\bu_\kappa(t) }_{\nsH}\abs{\xi(t) }_{L^4}
\\
&\leq r_n\abs{n}^q_V+C\theta\abs{\xi(t)}_{V}+C\abs{\xi(t)}_{L^{4}}\abs{\nabla c_\kappa(t)}_{L^{4}}+
\kappa \abs{\xi(t) }_{L^{q+1}}.
\end{align*}
Since $\frac{q+1}{q-1}<4$, we deduce  that
\begin{align*}
\lqq{	\abs{ A(t,n)}^{\frac{\alpha^*}{\alpha^*-1}}_{V^\star}}
\\
&\leq C(q)r^{\frac{q+1}{q}}_n\abs{n}^{q+1}_V+C(\theta,q)\abs{\xi(t)}^{\frac{q+1}{q}}_{V}+C(q)\abs{\xi(t)}^{\frac{q+1}{q}}_{L^{4}}\abs{\nabla c_\kappa(t)}^{\frac{q+1}{q}}_{L^{4}}+C(q)\kappa^{\frac{q+1}{q}}\abs{\xi(t) }^{\frac{q+1}{q}}_{L^{q+1}}
	\\
	&\leq C(q)r^{\frac{q+1}{q}}_n\abs{n}^{q+1}_V+\abs{\xi(t)}^{q+1}_{L^{4}}+C(q)\abs{\nabla c_\kappa(t)}^{\frac{q+1}{q-1}}_{L^{4}}+C(q)\kappa^{\frac{q+1}{q-1}}+C(\theta,q)
	\\
	&\leq C(q)r^{\frac{q+1}{q}}_n\abs{n}^{q+1}_V+\abs{\xi(t)}^{q+1}_{L^{4}}+C\abs{\nabla c_\kappa(t)}^{4}_{L^{4}}+C\kappa^{4}+C(\theta,q)\\
	&\leq C_0\abs{n}^{\alpha^*}_V+f(t).
\end{align*}
Observe, as before,
that $f\in  L^1([0,T]\times \Omega;dt\otimes\PP)$. This is (H4$'$) with exponent $0$ in front of $|n|_\CH$,
i.e.\ without a factor $(1+|n|^2_\CH)$; with a factor $(1+|n|^2_\CH)$, \cite[Theorem 5.1.3]{weiroeckner} would
require $f\in L^2([0,T]\times\Omega)$, which fails for the term $|\xi|^{q+1}_{L^{q+1}}$.

\medskip
In this way, we have proven that the Hypothesis (H1), (H2$^{\prime}$), (H3), and  (H4$^{\prime}$) are satisfied.
The existence and uniqueness of $n_\kappa$ with $n_\kappa\in C([0,T];H^{-1}_2(\CO))$ a.s.\ and
$$\EE\sup_{0\le s\le  {T} }\abs{\nk(s)}_{H^{-1}_2}^{{2}}<\infty,$$
 follows by an application of Theorem 5.1.3 in \cite[p.125]{weiroeckner} with $p=2$. The higher moments in
(\ref{3.12*}) follow from the a priori estimate below (by localisation with stopping times).

To calculate the bound,
we apply the It\^o formula to the function $\phi(\nk)=  |\nk|_{H^{-1}_2}^2$ and obtain
\begin{align*}
&|\nk(t)|_{H^{-1}_2}^2+2r_n\int_0^{t} \abs{\nk(s)}^{q+1}_{L^{q+1}} ds\notag\\
&=|n_0|_{H^{-1}_2}^2+2\int_0^t(\nabla^{-1}{(\theta\xi(s))},\nabla^{-1}\nk(s))ds
\\
	&\quad-2\chi\int_0^{t}(\nabla^{-1} \nk (s),\nabla^{-1} (\Div (\xi(s)\nabla \ckk(s))))ds
\\
&{}-2\int_0^{t}{\Theta_\kappa(\bu_\kappa,s)}(\nabla^{-1} \nk(s),
\nabla^{-1} \lk( {\bf u}_\kappa(s) \cdot\nabla \xi(s)
\rk))ds
	\\
	&\quad{}+2\int_0^{t}\left(\nk (s),g_{\gamma_1}(\nk(s))\,dW_1(s)\right)_{H^{-1}_2}+ \int_0^t|g_{\gamma_1}(\nk(s))|_{L_{\text{HS}}(H_1,H_2^{-1})}^2\,ds,
\end{align*}
where $\nabla^{-1}:=-(-\Delta)^{-1/2}$.
Using the  the fact that $L^{q+1}(\CO)\hookrightarrow H^{-1}_2(\CO)$ we obtain
\begin{align}
2(\nabla^{-1}{(\theta\xi)},\nabla^{-1}\nk)\leq{2}\abs{{\theta}\xi}_{H^{-1}_2} \abs{\nk}_{H^{-1}_2}\leq C\abs{\xi}_{L^{q+1}} \abs{\nk}_{H^{-1}_2}\leq \frac{1}{2}\abs{\xi}^2_{L^{q+1}} +C\abs{\nk}^2_{H^{-1}_2}\notag.
\end{align}
By the Young inequality, we have
\begin{align*}
\lqq{    2\chi(\nabla^{-1} \nk ,\nabla^{-1} (\Div (\xi\nabla \ckk)))
	\leq C\abs{\nk}_{H^{-1}_2} \abs{\Div (\xi\nabla \ckk)}_{H^{-1}_2}}
	\\
&\leq C\abs{\xi}_{L^4}\abs{\nabla\ck}_{L^4}\abs{\nk}_{H^{-1}_2} \leq C\abs{\nk}_{H^{-1}_2}\abs{\xi}_{L^4}^2+C\abs{\nk}_{H^{-1}_2}\abs{\nabla\ck}_{L^4}^2 \\
&\leq C\abs{\nk}_{H^{-1}_2}\abs{\xi}_{L^{q+1}}^2+C\abs{\nk}_{H^{-1}_2}\abs{\ck}_{H_2^2}^2
\end{align*}
and
\begin{align}
\lqq{ \Theta_\kappa(\bu_\kappa,t)2(\nabla^{-1} \nk,
\nabla^{-1} \lk( {\bf u}_\kappa \cdot\nabla \xi
\rk))}
\\
&\leq2\Theta_\kappa(\bu_\kappa,t)\abs{\nabla \xi\cdot\bu_\kappa}_{H^{-1}_2}\abs{\nk}_{H^{-1}_2}\notag
\leq C\Theta_\kappa(\bu_\kappa,t)
\abs{\bu_\kappa}_{\nsH}\abs{\xi }_{L^4}\abs{\nk}_{H^{-1}_2}
\\
&
\leq {\Theta_\kappa(\bu_\kappa,t)}\abs{\bu_\kappa}_{\nsH}^2\abs{\xi }^2_{L^{q+1}}+C\abs{\nk}_{H^{-1}_2}^2\notag
\leq {4\kappa^2\abs{\xi }^2_{L^{q+1}}}+C\abs{\nk}_{H^{-1}_2}^2.\notag
\end{align}
Here we used $\Theta_\kappa(\bu_\kappa,t)|\bu_\kappa(t)|_{\nsH}\le2\kappa$. The bound is a product
$\kappa^2|\xi|^2_{L^{q+1}}$.
Using the inequality \eqref{eq:H-1-2q},  we obtain
\begin{equation*}
	|g_{\gamma_1}(\nk)|_{L_{\text{HS}}(H_1,H_2^{-1})}^2\leq C\abs{\nk}_{H^{-1}_2}^2.
\end{equation*}
Combining these inequalities, we derive  that
\begin{align}
&|\nk(t)|_{H^{-1}_2}^2+2r_n\int_0^{t} \abs{\nk(s)}^{q+1}_{L^{q+1}} ds\notag\\
&\leq|n_0|_{H^{-1}_2}^2+{C(1+\kappa^2)}\int_0^t\abs{\xi(s)}^2_{L^{q+1}}ds+ C\int_0^t\abs{\nk(s)}_{H^{-1}_2}^2ds\notag\\
&\quad+C\int_0^t\abs{\nk(s)}_{H^{-1}_2}\abs{\xi(s)}_{L^{q+1}}^2ds+C\int_0^t\abs{\nk(s)}_{H^{-1}_2}\abs{\ck(s)}_{H_2^2}^2ds\notag\\
&\quad +2\int_0^{t}\left(\nk (s),g_{\gamma_1}(\nk(s))\,dW_1(s)\right)_{H^{-1}_2}.\notag
\end{align}
This implies that
\begin{align}
&\EE\sup_{0\le t\le T}\	|\nk(t)|_{H^{-1}_2}^\frac{4m^*}{r^*}+(2r_n)^\frac{2m^*}{r^*}\EE\left(\int_0^{T} \abs{\nk(s)}^{q+1}_{L^{q+1}} ds\right)^\frac{2m^*}{r^*}\notag\\
&\leq\EE|n_0|_{H^{-1}_2}^\frac{4m^*}{r^*}+{C(\kappa)}\EE\left(\int_0^T\abs{\xi(s)}^2_{L^{q+1}}ds\right)^\frac{2m^*}{r^*}\notag\\
&\quad{}+ C\EE\left(\int_0^T\abs{\nk(s)}_{H^{-1}_2}^2ds\right)^\frac{2m^*}{r^*}+C\EE\left(\int_0^T\abs{\nk(s)}_{H^{-1}_2}\abs{\xi(s)}_{L^{q+1}}^2ds\right)^\frac{2m^*}{r^*}\label{3.27}
\\\notag
	&\quad+C\EE\left(\int_0^T\abs{\nk(s)}_{H^{-1}_2}\abs{\ck(s)}_{H_2^2}^2ds\right)^\frac{2m^*}{r^*}
\\
&\quad {}+C\EE\sup_{0\le t\le T}\abs{\int_0^{t}\left(\nk (s),g_{\gamma_1}(\nk(s))\,dW_1(s)\right)_{H^{-1}_2}}^\frac{2m^*}{r^*}.\notag
\end{align}
\del{
We remark that we get  by estimate (\ref{3.14*})
\begin{align}
\EE\left(\int_0^T\abs{\bu_\kappa(s)}_{\nsH}^4ds\right)^\frac{2m^*}{r^*}&\leq C(T)\EE\sup_{0\le s\le  {T} }\abs{\bu_\kappa(s)}_{\nsH}^\frac{8m^*}{r^*}\notag\\
&\leq C\EE\abs{\bu_0}^{8\frac{m^*}{r^*}}_\nsH+ CR^\frac{4}{q+1}+C.\notag
\end{align}
}
We note that a two-times application of  the H\"older inequality  gives
\begin{align}
\lqq{ 	\EE\left(\int_0^T\abs{ \xi(s)}^2_{{L^{q+1}}}ds\right)^\frac{2m^*}{r^*}\leq C\EE\left(\int_0^T\abs{ \xi(s)}^{q+1}_{L^{q+1}}ds\right)^\frac{4m^*}{r^*(q+1)}
}
\notag\\
	&\leq C(T)\left(\EE\left(\int_0^T\abs{ \xi(s)}^{q+1}_{L^{q+1}}ds\right)^\frac{2m^*}{r^*}\right)^\frac{2}{q+1}\notag\leq CR^\frac{2}{q+1},\notag
\end{align}
and
\begin{equation}
\EE\left(\int_0^T\abs{\nk(s)}_{H^{-1}_2}^2ds\right)^\frac{2m^*}{r^*}\leq C(T,m^*,r^*)\EE\int_0^T\abs{\nk(s)}_{H^{-1}_2}^\frac{4m^*}{r^*}ds\notag,
\end{equation}
as well as
\begin{align}
\lqq{	C\EE\left(\int_0^T\abs{\nk(s)}_{H^{-1}_2}\abs{\xi(s)}_{L^{q+1}}^2ds\right)^\frac{2m^*}{r^*}	}
\\
&\leq C\EE\sup_{0\le t\le T}\	|\nk(t)|_{H^{-1}_2}^\frac{2m^*}{r^*}\left(\int_0^T\abs{\xi(s)}_{L^{q+1}}^2ds\right)^\frac{2m^*}{r^*}\notag\\
	&\leq \frac{1}{6}\EE\sup_{0\le t\le T}\	|\nk(t)|_{H^{-1}_2}^\frac{4m^*}{r^*}+C\EE\left(\int_0^T\abs{\xi(s)}_{L^{q+1}}^2ds\right)^\frac{4m^*}{r^*}\notag\\
	&\leq \frac{1}{6}\EE\sup_{0\le t\le T}\	|\nk(t)|_{H^{-1}_2}^\frac{4m^*}{r^*}+C\EE\left(\int_0^T\abs{\xi(s)}_{L^{q+1}}^{q+1}ds\right)^\frac{8m^*}{r^*(q+1)}\notag\\
	&\leq \frac{1}{6}\EE\sup_{0\le t\le T}\	|\nk(t)|_{H^{-1}_2}^\frac{4m^*}{r^*}+C\left(\EE\left(\int_0^T\abs{\xi(s)}_{L^{q+1}}^{q+1}ds\right)^\frac{2m^*}{r^*}\right)^\frac{4}{q+1}\notag\\
	&\leq \frac{1}{6}\EE\sup_{0\le t\le T}\	|\nk(t)|_{H^{-1}_2}^\frac{4m^*}{r^*}+CR^\frac{4}{q+1}\notag.
\end{align}
Similarly we get
\begin{align}
\lqq{C\EE\left(\int_0^T\abs{\nk(s)}_{H^{-1}_2}\abs{\ck(s)}_{H^{2}_2}^2ds\right)^\frac{2m^*}{r^*}
}
\\
&\leq C\EE\sup_{0\le t\le T}\	|\nk(t)|_{H^{-1}_2}^\frac{2m^*}{r^*}\left(\int_0^T\abs{\ck(s)}_{H^{2}_2}^2ds\right)^\frac{2m^*}{r^*}\notag
\leq \frac{1}{6}\EE\sup_{0\le t\le T}\	|\nk(t)|_{H^{-1}_2}^\frac{4m^*}{r^*}+C\EE\left(\int_0^T\abs{\ck(s)}_{H_2^{2}}^2ds\right)^\frac{4m^*}{r^*}\notag.
\end{align}
Using the estimate (\ref{3.12}), we derive that
\begin{equation}
\EE\left(\int_0^T\abs{ \ck(s)}^2_{H^2_2}ds\right)^\frac{4m^*}{r^*}\leq C\EE\abs{c_0}^{{\frac{8m^*}{r^*}}}_{H_2^1}+CR^\frac{4}{q+1}.\notag
\end{equation}
By the {Burkholder--Davis--Gundy} inequality and the inequality \ref{eq:H-1-2q}, we have
\begin{align}
	&C\EE\sup_{0\le t\le T}\abs{\int_0^t(\nk(s),g_{\gamma_1}(\nk(s)) dW_1(s))_{H_2^{-1}}}^\frac{2m^*}{r^*}\notag\\
	&\leq C\EE\left(\int_0^T\abs{\nk}^2_{H^{-1}_2}\abs{g_{\gamma_1}(\nk(s))}^2_{L_{\text{HS}}(H_1,H_2^{-1})}ds\right)^\frac{m^*}{r^*}\notag\\
	&\leq C\EE\sup_{0\le t\le T}\abs{\nk(t)}^\frac{2m^*}{r^*}_{H^{-1}_2}\left(\int_0^T\abs{g_{\gamma_1}(\nk(s))}^2_{L_{\text{HS}}(H_1,H_2^{-1})}ds\right)^{{\frac{m^*}{r^*}}}\notag\\
	&\leq\frac{1}{6}\EE\sup_{0\le t\le T}\abs{\nk({t})}^\frac{4m^*}{r^*}_{H^{-1}_2}+C\EE\left(\int_0^T\abs{g_{\gamma_1}(\nk(s))}^2_{L_{\text{HS}}(H_1,H_2^{-1})}ds\right)^\frac{2m^*}{r^*}\notag\\
	&\leq\frac{1}{6}\EE\sup_{0\le t\le T}\abs{\nk(t)}^\frac{4m^*}{r^*}_{H^{-1}_2}+C(T)\EE\int_0^T\abs{\nk(s)}_{H^{-1}_2}^\frac{4m^*}{r^*}ds.\notag
\end{align}
From these a priori estimates, we infer from (\ref{3.27}) that
\begin{align}
&\frac{1}{2}\EE\sup_{0\le t\le T}\	|\nk(t)|_{H^{-1}_2}^\frac{4m^*}{r^*}+(2r_n)^\frac{2m^*}{r^*}\EE\left(\int_0^{T} \abs{\nk(s)}^{q+1}_{L^{q+1}} ds\right)^\frac{2m^*}{r^*}\notag\\
&\leq\EE|n_0|_{H^{-1}_2}^\frac{4m^*}{r^*}+C\EE\abs{c_0}^{{\frac{8m^*}{r^*}}}_{H_2^1}+C
\\
&{}+{C(\kappa)}R^\frac{2}{q+1}+CR^\frac{4}{q+1}+C\EE\int_0^T\abs{\nk(s)}_{H^{-1}_2}^\frac{4m^*}{r^*}ds\notag.
\end{align}
The Gronwall lemma implies (\ref{3.12*}) and ends the proof of Claim  \ref{existencev}.
\end{proof}

\textbf{Verification of condition \eqref{metacontinuity} of Theorem \ref{meta}:}
Now, we shall tackle the continuity of the operator $\CV_\MA^\kappa$ in the norm induced by $\CM^{m^\ast}_\MA(\X)$. In the next proposition, we show that assumption \eqref{metacontinuity} of Theorem \ref{meta} is satisfied.

\newcommand{\NNa}{(n^{(1)}-n^{(2)})}
\newcommand{\deltaxi}{(\xi_1-\xi_2)}
\newcommand{\deltabu}{(\bu^{(1)}-\bu^{(2)})}
\newcommand{\deltac}{(c^{(1)}-c^{(2)})}

\noindent Throughout this paragraph we fix $\kappa\in\NN$, $R\ge R_\kappa$ and
$\xi_1,\xi_2\in\Xcal^\kappa_\Afrak(R)$. For $j=1,2$ let $(n^{(j)},c^{(j)},\bu^{(j)})$ be the solution of
(\ref{a1}) with $\xi=\xi_j$, the same initial data and the same Wiener processes, and let
$\Theta_j(t):=\Theta_\kappa(\bu^{(j)},t)$, $h_j(t):=h(\bu^{(j)},t)$. 
Since $\phi_\kappa$ is Lipschitz with constant $\|\phi'\|_\infty/\kappa$ and
$|h_1(t)-h_2(t)|\le\sup_{s\le t}|\deltabu(s)|_{\nsH}$, we have
\begin{equation}\label{Thetadiff}
|\Theta_1(t)-\Theta_2(t)|\le\frac{\|\phi'\|_\infty}{\kappa}\sup_{s\le t}|\deltabu(s)|_{\nsH}.
\end{equation}
\begin{proposition}\label{continuity}
Let the assumptions of Theorem \ref{main} hold, and let $\xi_1,\xi_2$ and the notation be as
above. Then there is a constant $C=C(\kappa,R,T)$ such that the following statements hold.
\begin{enumerate}
\item
	 Let $Z:=\sup_{0\le t\le T}|\NNa(t)|^2_{H^{-1}_2}+2^{-q}r_n\int_0^T|\NNa(s)|^{q+1}_{L^{q+1}}ds$. Then, we have 
\begin{align*}
\EE Z^{\frac{q+1}2}&\le C\,\EE\|\deltaxi\|^{q+1}_\X+C\big(\EE\|\deltaxi\|^{q+1}_\X\big)^{\frac{q+1}{2q}}
\\&\quad+C\Big(\EE\big(\|\deltac\|_{L^2(0,T;H^1_2)}+\sup_{0\le t\le T}|\deltabu(t)|_{\nsH}\big)^{q+1}\Big)^{\frac{q+1}{2q}}.
\end{align*}
\item
	We have 
$$
\EE\sup_{0\le t\le T}|\delta c(t)|^{q+1}_{L^2}+\EE\Big(\int_0^T|\deltac(s)|^2_{H^1_2}ds\Big)^{\frac{q+1}2}
\le C\,\EE\|\deltaxi\|^{q+1}_\X+C\big(\EE\|\deltaxi\|^{\frac32(q+1)}_\X\big)^{\frac23}.
$$
\item
We have $\PP$-a.s.,
\DEQSZ
\label{Sdu}
\lqq{ \sup_{0\le t\le T}|\deltabu(t)|^2_{\nsH}+r_{\bu}\int_0^T|\deltabu(s)|^2_{\nsV}ds}
\\
&
\le &C(\delta_1,\delta_2,r_{\bu})\int_0^T|\deltaxi(s)|^2_{H^{-1}_2}ds\le C\|\deltaxi\|^2_\X .\notag 
\EEQSZ
\end{enumerate}
\end{proposition}

The continuity of the operator $\CV_\MA^\kappa: \CX_\MA(R)\to \CX_\MA(R)$  follows by an interpolation argument.
\begin{corollary}\label{cor:cont}
Let us assume that the assumptions of Theorem \ref{main} are valid.
Let  $\xi_1,\xi_2\in\mathcal{X}_\MA(R)\subset\CM^{m^\ast}_\MA(\BX)$ be arbitrary.
Then, there exists  a function $\omega=\omega_{\kappa,R,T}:[0,\infty)\to[0,\infty)$ satisfying $\omega(0+)=0$,
such that %
$$
\EE \abss{ \CV_\MA^\kappa(\xi_1)-\CV_\MA^\kappa(\xi_2)}^{m^\ast}_\BX
\le \omega\big(\EE \abss{\xi_1-\xi_2}^{m^\ast}_\BX\big).
$$
\end{corollary}
\begin{proof}
\noindent Let $x:=\EE\|\deltaxi\|^{q+1}_\X$. By (\ref{Xbounded}) and the Cauchy--Schwarz inequality, we obtain
$$
\EE\|\deltaxi\|^{\frac32(q+1)}_\X\le\big(\EE\|\deltaxi\|^{q+1}_\X\big)^{\frac12}\big(\EE\|\delta\xi\|^{2(q+1)}_\X\big)^{\frac12}\le C(R)\,x^{\frac12}.
$$
Hence Proposition~\ref{continuity}-(ii) gives
$$\EE(\int_0^T|\delta c|^2_{H^1_2})^{\frac{q+1}2}\le C(x+x^{\frac13}),
$$
 and (\ref{Sdu}) gives
$\EE\sup_t|\deltabu|^{q+1}_{\nsH}\le Cx$. Inserting this into Proposition~\ref{continuity}-(i),
$$
\EE Z^{\frac{q+1}2}\le\omega_0(x):=C\big(x+x^{\frac{q+1}{2q}}+(x+x^{\frac13})^{\frac{q+1}{2q}}\big).
$$
By (\ref{interpolXp}) with $p=q+1$, 
$$\|\NNa\|^{q+1}_\X\le C(\mathcal A(\NNa)^{q+1}+\mathcal B(\NNa)^{q+1})\le C(Z^{\frac{q+1}2}+Z),
$$
and $\EE Z\le(\EE Z^{\frac{q+1}2})^{\frac2{q+1}}$. Since $\NNa=\CV_\MA^\kappa(\xi_1)-\CV_\MA^\kappa(\xi_2)$ and $m^\ast=q+1$, the
claim follows with $\omega(x):=C\big(\omega_0(x)+\omega_0(x)^{\frac2{q+1}}\big)$.
\end{proof}

\begin{proof}[Proof of Proposition \ref{continuity}]
\emph{(iii) The $\bu$-difference.} Since the noise cancels, solves the difference $\deltabu$  the following random equation 
\begin{equation}
d\deltabu(t)+r_{\bu}A\,\deltabu(t)dt=\Pi\big(\deltaxi(t)\star\Phi\big) dt, \ \text{ with } \  \deltabu(0)=0.\notag
\end{equation}
Testing this equation by $A\deltabu$ and using (\ref{xiPhi}), we get for all $t\in[0,T]$
\DEQS
\lqq{ |\deltabu(t)|^2_{\nsH}+2r_{\bu}\int_0^t|\deltabu(s)|^2_{\nsV}ds}
\\
&=&2\int_0^t\big(\Pi(\deltaxi\star\Phi),A\deltabu\big)ds
\\
&\le& r_{\bu}\int_0^t|\deltabu(s)|^2_{\nsV}ds+C(\delta_1,\delta_2,r_{\bu})\int_0^t|\deltaxi(s)|^2_{H^{-1}_2}ds .
\EEQS
Since $H^{-1/3}_2(\CO)\hookrightarrow H^{-1}_2(\CO)$ and $q+1\ge2$, the H\"older inequality in time gives
$$
\int_0^T|\deltaxi|^2_{H^{-1}_2}ds\le C(T)\|\deltaxi\|^2_\X,
$$
 and (\ref{Sdu}) follows.

\medskip

\emph{(ii) The $c$-difference.} The difference $\delta c$ satisfies the following equation
\begin{align*}
&d\deltac(t) +\left(\Theta_1(t)\bu^{(1)}(t)\cdot\nabla c^{(1)}(t)-\Theta_2(t)\bu^{(2)}(t)\cdot\nabla c^{(2)}(t)\right) dt\\
&=\lk( r_c \Delta \deltac(t)  -\alpha\,\deltac(t)+ \beta {\deltaxi(t)} \rk)\, dt
+g_{\gamma_2}(\deltac(t))\, {dW_2(t)}.
\end{align*}
Applying the It\^o formula to $|\deltac|^2_{L^2}$ and using $$-r_c|\nabla\deltac|^2_{L^2}-\alpha|\deltac|^2_{L^2}\le-\alpha_0|\deltac|^2_{H^1_2},
$$ we arrive at
\begin{align}
&\sup_{0\le s\le  {t}  }\abs{\deltac (s) }^2_{L^2}+2\alpha_0\int_0^t\abs{\deltac (s) }^2_{H^1_2}ds\notag\\
&\leq2\beta\int_0^t\big|(\deltaxi (s),  \deltac (s))\big|ds
\\
&+2\int_0^t\Big|\left(\Theta_1(s)\bu^{(1)}(s)\cdot\nabla c^{(1)}(s)-\Theta_2(s)\bu^{(2)}(s)\cdot\nabla c^{(2)}(s), \deltac (s)\right)\Big| ds\notag\\
&{}+\int_0^t\abs{g_{\gamma_2}(\deltac (s))}^2_{L_{\text{HS}}(H,L^{2})}ds
+2\sup_{0\le r\le  {t}  }\Big|\int_0^r(\deltac (s),g_{\gamma_2}(\deltac (s)) dW_2(s))\Big|.\label{3.30}\notag 
\end{align}
By the estimate \ref{DD} with $s=\frac13$, we obtain
\begin{align}
2\beta(\deltaxi,  \deltac )
	\leq  C\abs{\deltaxi}_{H_2^{-1/3}} \abs{\deltac }_{H_2^{1}}
\leq \frac{\alpha_0}{3}\abs{\deltac }^2_{H_2^{1}}+C\abs{\deltaxi}_{H_2^{-1/3}} ^2.\notag
\end{align}
\noindent For the transport term we use a case distinction which is made pathwise, for each $(t,\omega)$.
On $\{h_1(t)\le h_2(t)\}$ we write
$$
\Theta_1\bu^{(1)}\!\cdot\!\nabla c^{(1)}-\Theta_2\bu^{(2)}\!\cdot\!\nabla c^{(2)}
=\Theta_2\,\deltabu \cdot\nabla c^{(1)}+\Theta_2\,\bu^{(2)}\!\cdot\!\nabla\deltac +(\Theta_1-\Theta_2)\,\bu^{(1)}\!\cdot\!\nabla c^{(1)},
$$
and on $\{h_1(t)>h_2(t)\}$
$$
\Theta_1\bu^{(1)}\!\cdot\!\nabla c^{(1)}-\Theta_2\bu^{(2)}\!\cdot\!\nabla c^{(2)}
=\Theta_1\,\deltabu \cdot\nabla c^{(1)}+\Theta_1\,\bu^{(2)}\!\cdot\!\nabla\deltac +(\Theta_1-\Theta_2)\,\bu^{(2)}\!\cdot\!\nabla c^{(2)}.
$$
In both cases the middle term vanishes after pairing with $\deltac $, since $\operatorname{div}\bu^{(2)}=0$
and $\bu^{(2)}=0$ on $\partial\CO$. In the last term, $(\Theta_1-\Theta_2)$ multiplies $\bu^{(i)}$ with the
smaller value of $h_i(t)$. If $\Theta_1(t)\neq\Theta_2(t)$, then one of them is nonzero, so
$\min\{h_1(t),h_2(t)\}<2\kappa$ and hence $|\bu^{(i)}(t)|_{\nsH}\le h_i(t)<2\kappa$ for this $i$. With
(\ref{Thetadiff}) we get in both cases
\begin{equation}\label{Thetau}
|\Theta_1(t)-\Theta_2(t)|\,|\bu^{(i)}(t)|_{\nsH}\le 2\|\phi'\|_\infty\sup_{s\le t}|\deltabu(s)|_{\nsH}.
\end{equation}
Since these are pathwise estimates of a Lebesgue integral, the $(t,\omega)$-dependent choice of the
decomposition is admissible, and no stopping times are needed. Using
$|\bv|_{\mathbf L^4}\le C|\bv|_{\nsH}$ and $|f|_{L^4}\le C|f|^{1/2}_{L^2}|f|^{1/2}_{H^1_2}$, both
remaining terms are bounded by
$$
C\sup_{s\le t}|\deltabu (s)|_{\nsH}\big(|\nabla c^{(1)}(t)|_{L^2}+|\nabla c^{(2)}(t)|_{L^2}\big)|\deltac (t)|^{\frac12}_{L^2}|\deltac (t)|^{\frac12}_{H^1_2}.
$$
Hence
\begin{align*}
&\int_0^t\Big|\big(\Theta_1\bu^{(1)}\cdot\nabla c^{(1)}-\Theta_2\bu^{(2)}\cdot\nabla c^{(2)},\deltac \big)\Big|ds
\\&\le C\sup_{s\le t}|\deltabu (s)|^2_{\nsH}\sum_{i=1}^2\sup_{s\le t}|\nabla c^{(i)}(s)|^2_{L^2}
+C\int_0^t|\deltac (s)|^2_{L^2}ds+\frac{\alpha_0}3\int_0^t|\deltac (s)|^2_{H^1_2}ds.
\end{align*}

By (\ref{eq:Hil-Schi2}),
\begin{equation}
\int_0^t\abs{g_{\gamma_2}(\deltac (s))}^2_{L_{\text{HS}}(H,L^{2})}ds\leq C\int_0^t\abs{\deltac (s)}^2_{L^2}ds.\notag
\end{equation}
From these last inequalities, we derive that
\begin{align}
	&\sup_{0\le s\le  {t}  }\abs{\deltac (s) }^2_{L^2}+\alpha_0\int_0^t\abs{\deltac (s) }^2_{H^1_2}ds\notag 
\\
&	\leq C\int_0^t\abs{\deltaxi (s)}_{H_2^{-1/3}} ^2ds+ C\int_0^t\abs{\deltac (s)}^2_{L^2}ds\notag\\\notag 
	&\qquad{} +C\sup_{0\le s\le  {T}  }\abs{\deltabu (s)}_{\nsH}^2\sum_{i=1}^2\sup_{0\le s\le  {T} }\abs{\nabla c^{(i)}(s)}_{L^2}^2
\\
&{} 	+2\sup_{0\le r\le  {t}  }\abs{\int_0^r(\deltac (s),g_{\gamma_2}(\deltac (s)) dW_2(s))}.\label{3.31}
\end{align}
\noindent We take the power $\frac{q+1}2$ and the expectation. By the H\"older inequality in time, we obtain
$$\big(\int_0^T|\deltaxi |^2_{H^{-1/3}_2}ds\big)^{\frac{q+1}2}\le C\|\deltaxi \|^{q+1}_\X.
$$
 By the H\"older inequality with
exponents $\frac32$ and $3$, by (\ref{Sdu}) and by (\ref{3.12}),
\begin{align*}
\EE\Big[\sup_{s\le T}|\deltabu (s)|^{q+1}_{\nsH}\sup_{s\le T}|\nabla c^{(i)}(s)|^{q+1}_{L^2}\Big]
&\le\big(\EE\sup_{s\le T}|\deltabu (s)|^{\frac32(q+1)}_{\nsH}\big)^{\frac23}\big(\EE\sup_{s\le T}|c^{(i)}(s)|^{3(q+1)}_{H^1_2}\big)^{\frac13}
\\&\le C(\kappa,R)\big(\EE\|\deltaxi \|^{\frac32(q+1)}_\X\big)^{\frac23},
\end{align*}
since $3(q+1)\le 8\frac{m^\ast}{r^\ast}=4(q+1)$. By the Burkholder--Davis--Gundy inequality, (\ref{eq:Hil-Schi2}) and the Young
inequality,
\DEQS
\lqq{\EE\sup_{0\le r\le t}\Big|\int_0^r(\deltac ,g_{\gamma_2}(\deltac )\,dW_2)\Big|^{\frac{q+1}2}}
\\
&\le&\frac14\EE\sup_{0\le s\le t}|\deltac (s)|^{q+1}_{L^2}+C\int_0^t\EE\sup_{0\le r\le s}|\deltac (r)|^{q+1}_{L^2}ds .
\EEQS
The Gronwall lemma, applied to $t\mapsto\EE\sup_{s\le t}|\deltac (s)|^{q+1}_{L^2}$, gives (ii).

\medskip

\noindent\emph{(i) The $n$-difference.} Since $\operatorname{div}\bu^{(j)}=0$, the transport terms can be
written as $\operatorname{div}(\Theta_j\bu^{(j)}\xi_j)$, and $N$ solves
\begin{align*}
\lqq{ d\NNa}
\\
&=\Big(r_n\Delta\big(|n^{(1)}|^{q-1}n^{(1)}-|n^{(2)}|^{q-1}n^{(2)}\big)+\theta\,\deltaxi 
-\chi\operatorname{div}\big(\deltaxi \nabla c^{(1)}+\xi_2\nabla\deltac \big)\\
&\qquad-\operatorname{div}\big(\Theta_1\bu^{(1)}\xi_1-\Theta_2\bu^{(2)}\xi_2\big)\Big)dt+g_{\gamma_1}(\NNa )\,dW_1,
\end{align*}
$\NNa (0)=0$. For a vector field $F$ we have
$\langle\operatorname{div}F,(-\Delta)^{-1}N\rangle=-(F,\psi)$, where $\psi=\nabla(-\Delta)^{-1}N$. By the
It\^o formula for $|\NNa |^2_{H^{-1}_2}$ and the inequality
$(|a|^{q-1}a-|b|^{q-1}b)(a-b)\ge2^{1-q}|a-b|^{q+1}$, for $t\in[0,T]$,
\begin{align}
&|\NNa(t)|^2_{H^{-1}_2}+2^{2-q}r_n\int_0^t|\NNa|^{q+1}_{L^{q+1}}ds
\\
&\le2\int_0^t|\langle\theta\,\deltaxi ,(-\Delta)^{-1}\NNa\rangle|ds
+2\chi\int_0^t|(\deltaxi \nabla c^{(1)}+\xi_2\nabla\deltac ,\psi)|ds\notag\\
&\qquad+2\int_0^t|(\Theta_1\bu^{(1)}\xi_1-\Theta_2\bu^{(2)}\xi_2,\psi)|ds
+\int_0^t|g_{\gamma_1}(\NNa)|^2_{L_{HS}(H,H^{-1}_2)}ds+M_t,\label{Ndiff}
\end{align}
where $M_t:=2\int_0^t(\NNa ,g_{\gamma_1}(\NNa )\,dW_1)_{H^{-1}_2}$. We estimate the terms one by one.
\begin{itemize}
\item Since $\theta\in W^{1,\infty}(\CO)$ (a constant in the original setting), by (D) with $s=\frac13$ and
$|(-\Delta)^{-1}\NNa |_{H^1_0}=|\NNa |_{H^{-1}_2}$:
\DEQS\lqq{
2\int_0^t|\langle\theta\,\deltaxi ,(-\Delta)^{-1}\NNa \rangle|ds}
\\
&\le &C\sup_{s\le t}|\NNa (s)|_{H^{-1}_2}\int_0^t|\deltaxi |_{H^{-1/3}_2}ds
\le\frac14\sup_{s\le t}|\NNa (s)|^2_{H^{-1}_2}+C\|\deltaxi \|^2_\X.
\EEQS 

\item By (E) and $H^2_2(\CO)\hookrightarrow W^{1,b}(\CO)$, $\frac1b=\frac12-\frac1{q+1}$:
\DEQS\lqq{
	|\nabla c\cdot\psi|_{H^1_2}}
	\\
	&\le& C(|\nabla c|_{L^2}+|D^2c|_{L^2})|\psi|_{L^\infty}+|\nabla c|_{L^b}|\nabla\psi|_{L^{q+1}}
\le C|c|_{H^2_2}|\NNa |_{L^{q+1}}.
\EEQS
 Hence, by (D),
 \DEQS\lqq{
|(\deltaxi \nabla c^{(1)},\psi)|}
\\
&\le &C|\deltaxi |_{H^{-1/3}_2}|c^{(1)}|_{H^2_2}|\NNa |_{L^{q+1}}.
\EEQS%

\item By (E): 
\DEQS\lqq{
	|(\xi_2\nabla\deltac ,\psi)|}
	\\
	&\le&|\xi_2|_{L^{q+1}}|\nabla\deltac |_{L^2}|\psi|_{L^{2(q+1)/(q-1)}}
\le C|\xi_2|_{L^{q+1}}|\deltac |_{H^1_2}|\NNa |_{L^{q+1}}.
\EEQS
\item For the transport term we use the same pathwise case distinction as in (ii):
 on
$\{h_1(t)\le h_2(t)\}$,
\DEQS\lqq{
\Theta_1\bu^{(1)}\xi_1-\Theta_2\bu^{(2)}\xi_2}
\\
&=&\Theta_1\bu^{(1)}\deltaxi +\Theta_2\,\deltabu \,\xi_2+(\Theta_1-\Theta_2)\bu^{(1)}\xi_2,
\EEQS 
and on $\{h_1(t)>h_2(t)\}$,
\DEQS\lqq{
	\Theta_1\bu^{(1)}\xi_1-\Theta_2\bu^{(2)}\xi_2}
	\\
	&=&\Theta_1\bu^{(1)}\deltaxi +\Theta_1\,\deltabu \,\xi_2+(\Theta_1-\Theta_2)\bu^{(2)}\xi_2.
	\EEQS 
For the first term, by (D) and (E),
\DEQS\lqq{
	|(\Theta_1\bu^{(1)}\deltaxi ,\psi)|}
	\\
	&\le&|\deltaxi |_{H^{-1/3}_2}|\bu^{(1)}\cdot\psi|_{H^1_2}\le C|\deltaxi |_{H^{-1/3}_2}|\bu^{(1)}|_{\nsH}|\NNa |_{L^{q+1}}.
	\EEQS 
For the other two terms, by (\ref{Thetau}) and (E),
\DEQS\lqq{
	|(\Theta_j\deltabu \,\xi_2+(\Theta_1-\Theta_2)\bu^{(i)}\xi_2,\psi)|}
	\\
	&\le & C\sup_{s\le t}|\deltabu (s)|_{\nsH}|\xi_2|_{L^{q+1}}|\NNa |_{L^{q+1}}.
	\EEQS 
\item By (\ref{eq:H-1-2q}), $|g_{\gamma_1}(\NNa )|^2_{L_{HS}(H,H^{-1}_2)}\le C|\NNa |^2_{H^{-1}_2}$.
\end{itemize}
Now we integrate in time and use the H\"older inequality with exponents $(q+1,2,q+1)$; this is
possible because $\frac1{q+1}+\frac12+\frac1{q+1}\le1$, i.e.\ $q\ge3$ (this is where
$m^\ast=q+1\ge\frac{2(q+1)}{q-1}$ is used). With
\begin{align*}
Y_1&:=\|\deltaxi \|_\X\Big(\|c^{(1)}\|_{L^2(0,T;H^2_2)}+\sup_{t\le T}|\bu^{(1)}(t)|_{\nsH}\Big),\\
Y_2&:=\mathcal B(\xi_2)\Big(\|\deltac \|_{L^2(0,T;H^1_2)}+\sup_{t\le T}|\deltabu (t)|_{\nsH}\Big),
\end{align*}
the second and third integral on the right-hand side of (\ref{Ndiff}) are bounded by
\DEQS
C(Y_1+Y_2)\big(\int_0^t|\NNa |^{q+1}_{L^{q+1}}ds\big)^{\frac1{q+1}}\le2^{-q}r_n\int_0^t|\NNa |^{q+1}_{L^{q+1}}ds+C(Y_1+Y_2)^{\frac{q+1}q}.
\EEQS 
Let 
$$Z(t):=\sup_{s\le t}|\NNa (s)|^2_{H^{-1}_2}+2^{-q}r_n\int_0^t|\NNa |^{q+1}_{L^{q+1}}ds,
$$
 so that $Z=Z(T)$. Taking
the supremum over $[0,t]$ in (\ref{Ndiff}), we obtain the pathwise bound
$$
Z(t)\le C\|\deltaxi \|^2_\X+C(Y_1+Y_2)^{\frac{q+1}q}+C\int_0^t|\NNa (s)|^2_{H^{-1}_2}ds+2\sup_{s\le t}|M_s| .
$$
We take the power $\frac{q+1}2$ and the expectation. With $\gamma:=\frac{(q+1)^2}{2q}$ and the
Burkholder--Davis--Gundy inequality,
\DEQS\lqq{\EE\sup_{s\le t}|M_s|^{\frac{q+1}2}}
\\
&\le &C\EE\big[\sup_{s\le t}|\NNa (s)|^{\frac{q+1}2}_{H^{-1}_2}(\int_0^t|\NNa |^2_{H^{-1}_2})^{\frac{q+1}4}\big]
\le\frac14\EE Z(t)^{\frac{q+1}2}+C\int_0^t\EE Z(s)^{\frac{q+1}2}ds.
\EEQS 
Note that $\EE Z(T)^{\frac{q+1}2}<\infty$
by (\ref{3.12*}). The Gronwall lemma gives
$$
\EE Z^{\frac{q+1}2}\le C\,\EE\|\deltaxi \|^{q+1}_\X+C\,\EE Y_1^\gamma+C\,\EE Y_2^\gamma .
$$
Finally, by the H\"older inequality with exponents $P=\frac{2q}{q+1}$ and $P'=\frac{2q}{q-1}$,
\begin{align*}
\EE Y_1^\gamma&\le\big(\EE\|\deltaxi \|^{q+1}_\X\big)^{\frac{q+1}{2q}}
\Big(\EE\big(\|c^{(1)}\|_{L^2(H^2_2)}+\sup_t|\bu^{(1)}|_{\nsH}\big)^{\frac{(q+1)^2}{q-1}}\Big)^{\frac{q-1}{2q}},\\
\EE Y_2^\gamma&\le\Big(\EE\big(\|\deltac \|_{L^2(H^1_2)}+\sup_t|\deltabu |_{\nsH}\big)^{q+1}\Big)^{\frac{q+1}{2q}}
\Big(\EE\Big(\int_0^T|\xi_2|^{q+1}_{L^{q+1}}ds\Big)^{\frac{q+1}{q-1}}\Big)^{\frac{q-1}{2q}}.
\end{align*}
The second factors are bounded by $C(\kappa,R)$: $\frac{(q+1)^2}{q-1}\le4(q+1)$ for $q\ge\frac53$, and
moments of order $4(q+1)=8\frac{m^\ast}{r^\ast}$ of $\|c^{(1)}\|_{L^2(H^2_2)}$ and $\sup_t|\bu^{(1)}|_{\nsH}$ are
bounded by (\ref{3.12}) and (\ref{3.14*}); moreover $\frac{q+1}{q-1}\le q+1$ and
$\EE(\int_0^T|\xi_2|^{q+1}_{L^{q+1}})^{q+1}\le R$. This proves (i).
\end{proof}

\textbf{Verification of condition  \eqref{metacompact} of Theorem \ref{meta}: }
In this paragraph, we will verify the compactness of the operator $\CV_{\MA}^\kappa $ restricted to $\CX_\MA(R)$, i.e.\ assumption \eqref{metacompact} of Theorem \ref{meta}.

\noindent We take $s_0:=\frac14<s^{\ast\ast}=\frac13$, $\varrho:=\frac1{2(q+1)}$ and
$$
\X':=L^{q+1}\big(0,T;H^{-1/4}_2(\CO)\big)\cap W^{\varrho,q+1}(0,T;V^\ast),\qquad m_0:=q+1 .
$$
Since $\CO$ is bounded, $H^{-1/4}_2(\CO)\hookrightarrow\hookrightarrow H^{-1/3}_2(\CO)$ (see \cite[p.~82, Theorem~2]{runst}),
and $H^{-1/3}_2(\CO)\hookrightarrow H^{-1}_2(\CO)\hookrightarrow V^\ast$. The spaces $H^{-1/4}_2(\CO)$ and $V^\ast$ are
reflexive. Hence, by \cite[Theorem 2.1]{Flandoli},
$$
\X'=L^{q+1}\big(0,T;H^{-1/4}_2(\CO)\big)\cap W^{\varrho,q+1}(0,T;V^\ast)\hookrightarrow\hookrightarrow
L^{q+1}\big(0,T;H^{-1/3}_2(\CO)\big)=\X .
$$
Condition (\ref{metacompact}) therefore follows from the next claim.

\begin{claim}\label{claim3.9}
 For  any $R>0$, $\kappa\in \mathbb{N}$ and $\xi\in  \mathcal{X}(R) $,  there exists a  constant $C(\kappa, R)>0$ such that
\DEQS
(i)\quad \EE\|n_\kappa\|^{q+1}_{L^{q+1}(0,T;H_{2}^{-1/4})}\le C \quad \mbox{and}\quad (ii) \quad  \EE\|n_\kappa\|_{W^{\varrho,q+1}(0,T;V^\ast)}^{q+1}\le C,
\EEQS
where $n_\kappa:=\mathcal{V}_\MA^\kappa(\xi)$.
\end{claim}
\begin{proof}
By Lemma~\ref{interpolation_11} in Appendix~\ref{TP}, applied with $\theta=\frac13$ and $\theta=\frac14$, there is
a constant $C=C(T,\CO,q)$ with
\begin{equation}\label{interpolX}
	\|\xi\|_{\X}\le C\,\mathcal A(\xi)^{1/3}\mathcal B(\xi)^{2/3},\qquad
	\|\xi\|_{L^{q+1}(0,T;H^{-1/4}_2)}\le C\,\mathcal A(\xi)^{1/4}\mathcal B(\xi)^{3/4},
\end{equation}
and hence, by the Young inequality, for every $p\ge1$,
\begin{equation}\label{interpolXp}
	\|\xi\|^p_{\X}+\|\xi\|^p_{L^{q+1}(0,T;H^{-1/4}_2)}\le C\big(\mathcal A(\xi)^p+\mathcal B(\xi)^p\big).
\end{equation}

(ii) We estimate the increments of $n_\kappa$ in $V^\ast$, not in $H^{-1}_2(\CO)$. From (\ref{sysn}), for
$0\le s<t\le T$,
$$
n_\kappa(t)-n_\kappa(s)=\int_s^t\Big(r_n\Delta(|n_\kappa|^{q-1}n_\kappa)+\theta\xi-\chi\operatorname{div}(\xi\nabla c_\kappa)
-\Theta_\kappa\operatorname{div}(\bu_\kappa\xi)\Big)dr+\int_s^t g_{\gamma_1}(n_\kappa)\,dW_1 .
$$
By (\ref{Vstar}), $|\Delta(|n|^{q-1}n)|_{V^\ast}=|n|^q_{L^{q+1}}$ and $|h|_{V^\ast}\le C|h|_{H^{-1}_2}$. Moreover,
using $L^{q+1}(\CO)\hookrightarrow L^4(\CO)$, $H^2_2(\CO)\hookrightarrow W^{1,4}(\CO)$, $\nsH\hookrightarrow\mathbf L^4(\bo)$ and
$\Theta_\kappa|\bu_\kappa|_{\nsH}\le2\kappa$,
\begin{align*}
|\theta\xi|_{V^\ast}&\le C|\xi|_{L^{q+1}},\\
|\operatorname{div}(\xi\nabla c_\kappa)|_{H^{-1}_2}&\le|\xi|_{L^4}|\nabla c_\kappa|_{L^4}\le C|\xi|_{L^{q+1}}|c_\kappa|_{H^2_2},\\
|\Theta_\kappa\operatorname{div}(\bu_\kappa\xi)|_{H^{-1}_2}&\le\Theta_\kappa|\bu_\kappa|_{\mathbf L^4}|\xi|_{L^4}\le C\kappa|\xi|_{L^{q+1}} .
\end{align*}
The H\"older inequality in time gives
\begin{align*}
\Big(\int_s^t|n_\kappa|^q_{L^{q+1}}dr\Big)^{q+1}&\le|t-s|\Big(\int_0^T|n_\kappa|^{q+1}_{L^{q+1}}dr\Big)^q,\\
\Big(\int_s^t|\xi|_{L^{q+1}}|c_\kappa|_{H^2_2}dr\Big)^{q+1}&\le|t-s|^{\frac{q-1}2}
\int_0^T|\xi|^{q+1}_{L^{q+1}}dr\Big(\int_0^T|c_\kappa|^2_{H^2_2}dr\Big)^{\frac{q+1}2},\\
\Big(\int_s^t|\xi|_{L^{q+1}}dr\Big)^{q+1}&\le|t-s|^q\int_0^T|\xi|^{q+1}_{L^{q+1}}dr,
\end{align*}
and, by the Burkholder--Davis--Gundy inequality in $H^{-1}_2(\CO)$ and (\ref{eq:H-1-2q}),
$$
\EE\Big|\int_s^t g_{\gamma_1}(n_\kappa)\,dW_1\Big|^{q+1}_{V^\ast}\le C\,\EE\Big(\int_s^t|n_\kappa|^2_{H^{-1}_2}dr\Big)^{\frac{q+1}2}
\le C|t-s|^{\frac{q+1}2}\,\EE\,\mathcal A(n_\kappa)^{q+1}.
$$
The required moments are finite: $$\EE(\int_0^T|n_\kappa|^{q+1}_{L^{q+1}})^q\le C(\kappa,R)
$$ 
by (\ref{3.12*}), and by
the Cauchy--Schwarz inequality
$$\EE\big[\int_0^T|\xi|^{q+1}_{L^{q+1}}(\int_0^T|c_\kappa|^2_{H^2_2})^{\frac{q+1}2}\big]
\le(\EE(\int_0^T|\xi|^{q+1}_{L^{q+1}})^2)^{\frac12}(\EE(\int_0^T|c_\kappa|^2_{H^2_2})^{q+1})^{\frac12}\le C(\kappa,R)
$$
by (\ref{defXR}) and (\ref{3.12}). Since $q\ge3$, all exponents of $|t-s|$ are $\ge1$, so
$\EE|n_\kappa(t)-n_\kappa(s)|^{q+1}_{V^\ast}\le C(\kappa,R)|t-s|$, and therefore
$$
\EE\int_0^T\!\!\int_0^T\frac{|n_\kappa(t)-n_\kappa(s)|^{q+1}_{V^\ast}}{|t-s|^{1+\varrho(q+1)}}\,ds\,dt
\le C(\kappa,R)\int_0^T\!\!\int_0^T|t-s|^{-\varrho(q+1)}\,ds\,dt<\infty,
$$
because $\varrho(q+1)=\frac12<1$. Together with $\EE\int_0^T|n_\kappa|^{q+1}_{V^\ast}dt\le C\,\EE\,\mathcal A(n_\kappa)^{q+1}$
this gives (ii).
\end{proof}

\textbf{Verification of assumption \eqref{metauniformintegrability} of Theorem \ref{meta}:}
Here we will find $m_1>m^*$ such that 	$\EE\|\CV_{\MA}^\kappa (\xi)\|_{\BX}^{m_1}\le C(\kk,R)$ for all $\xi\in\Xcal^\kappa_{\Afrak}(R)$.
By (\ref{Xbounded}), the set $\Xcal^\kappa_{\Afrak}(R)$ is bounded in $\CM^{m_1}_\Afrak(\X)$ with
$m_1=2(q+1)>m^\ast$, and $\CV_{\MA}^\kappa$ maps $\Xcal^\kappa_{\Afrak}(R)$ into itself for $R\ge R_\kappa$. Hence
$\EE\|\CV_{\MA}^\kappa (\xi)\|_{\BX}^{2(q+1)}\le C(R+R^{\frac2{q+1}})$.
The second part of (\ref{metauniformintegrability}) holds with $p=2(q+1)>2$: by Claim~\ref{existencev}, i.e.\ by the
estimate for $\ukk$, (\ref{3.12}) and (\ref{3.12*}), and since 
$
4\frac{m^\ast}{r^\ast}=2(q+1)$, $8\frac{m^\ast}{r^\ast}=4(q+1)
$,
$|c|_{L^2}\le|c|_{H^1_2}$ and $|\bu|_{\mathbf L^2}\le C|\bu|_\nsH$, we have
\DEQS
\EE\sup_t|\nk(t)|^{2(q+1)}_{H^{-1}_2}+\EE\sup_t|\ckk(t)|^{2(q+1)}_{L^2}+\EE\sup_t|\ukk(t)|^{2(q+1)}_{\mathbf L^2}\le C(\kappa,R).
\EEQS

\smallskip

\noindent\textbf{Verification of condition (\ref{metaident}) of Theorem~\ref{meta}:} By the verification of
Assumption~\ref{ass:coeff} before Step~(I), $\alpha=q+1$ and $\theta=2$, so $\alpha'\theta=\frac{2(q+1)}q$. With $m_0=q+1$ and
$m_1=2(q+1)$, and since $q>2$, we have $\alpha'\theta<q+1=\min\{m_0,m_1\}$. Moreover,
$V_0=L^{q+1}(\CO)$, and by (\ref{3.12*}) and the Jensen inequality, $\EE\int_0^T|\nk(s)|^{q+1}_{L^{q+1}}ds\le C(\kappa,R)$ for
$\xi\in\Xcal^\kappa_\Afrak(R)$.

\smallskip

It remains to verify assumption \eqref{metacontinuityw}, which is the path-wise continuity of $\nk$, $c_\kappa$ and $\ukk$.
\textbf{Verification of condition \eqref{metacontinuityw} of Theorem \ref{meta}:} This follows from the
existence results used in the proof of Claim~\ref{existencev}: by \cite[Theorem 5.1.3]{weiroeckner},
$\PP$-a.s.\ $n_\kappa\in C([0,T];H^{-1}_2(\CO))$ and $c_\kappa\in C([0,T];L^2(\CO))$, and the Stokes equation
(\ref{a3}) has a solution $\bu_\kappa\in C([0,T];\nsH)\subset C([0,T];\mathbf L^2(\bo))$ a.s., see
\cite[Section 13.11.1]{DaPrZa:2nd}. Kolmogorov's continuity theorem is not needed.

\smallskip

\noindent\textbf{Verification of conditions (\ref{metacausal}) and (\ref{metaproj}) of Theorem~\ref{meta}:}
Condition (\ref{metaproj}) was shown in Step~(I): by the Jensen inequality (see the remark after (\ref{hatdefined})),
$\mathcal A(\Pro_\ell\xi)\le\mathcal A(\xi)$ and $\mathcal B(\Pro_\ell\xi)\le\mathcal B(\xi)$, and $\Pro_\ell\xi$ is progressively
measurable. For (\ref{metacausal}), let $t\in[0,T]$ and $\xi_1,\xi_2\in\Xcal^\kappa_\Afrak(R)$ with $\xi_1=\xi_2$ on $[0,t]$
$\PP$-a.s. The system (\ref{a1}) is solved successively. First, $\bu_\kappa$ is obtained from (\ref{a3}), which involves
$\xi$ only. Then $c_\kappa$ is obtained from (\ref{sysv}), which involves $\xi$ and $\bu_\kappa$. Finally, $n_\kappa$ is
obtained from (\ref{sysn}), which involves $\xi$, $c_\kappa$ and $\bu_\kappa$. The cut-off $\Theta_\kappa(\bu_\kappa,s)$
depends only on $\bu_\kappa|_{[0,s]}$. The restrictions to $[0,t]$ of the solutions for $\xi_1$ and $\xi_2$ solve the
same equations on $[0,t]$ with the same data. The solutions on $[0,t]$ are unique: for $\bu_\kappa$ pathwise, and
for $c_\kappa$ and $n_\kappa$ by \cite[Theorem 5.1.3]{weiroeckner} applied on $[0,t]$. Hence
$(n_\kappa,c_\kappa,\bu_\kappa)(\xi_1)=(n_\kappa,c_\kappa,\bu_\kappa)(\xi_2)$ on $[0,t]$, $\PP$-a.s.

\smallskip

\textbf{Verification of the  inequality \eqref{metauniform}:} Here, we will prove  that inequality \eqref{metauniform} holds. To be more precise, assuming that $(\nk,\ck,\ukk)$ is a fixed point of system \eqref{a1}, we will give a uniform bound independently on $\kappa$. Then, we will consider  the following system.
\DEQSZ
\label{a2}
\lk\{\barray
d{n}_\kappa(t) & =& \lk(r_n\Delta   |n_\kappa(t)|^{q-1}n_\kappa(t)+\theta \nk(t)- \chi \Div( \nk(t) \nabla \ckk(t))\rk)\, dt
\\ && {}-\Theta_\kappa(\bu_\kappa,t)
\bu_\kappa (t) \cdot\nabla \nk(t)\, dt+g_{\gamma_1}(n_\kappa(t))  dW_1(t),\phantom{\Big|}
\\
d{\ckk}(t)& =&\lk( r_c \Delta \ckk(t)  -\alpha \ckk(t)+ \beta
\nk(t) \rk)\, dt
\\ &&{}- \Theta_\kappa(\bu_\kappa,t)\bu_\kappa (t) \cdot\nabla \ckk(t)\, dt+ g_{\gamma_2}(\ckk(t)) dW_2(t),\phantom{\Big|}
\\
d \bu _\kappa(t)+r_{\bu} A \bu_\kappa(t)\,dt&=&\Pi\big(\nk(t)\star\Phi\big) dt+ \sigma _{\gamma_3} dW_3(t),\earray\rk.
\EEQSZ
 with $\nk(0)=n_0$, $\ckk(0)=c_0$, $\ukk(0)=\bu_0$. This is system (\ref{a1}) with $\xi=\nk$. The
estimates below use only $0\le\Theta_\kappa\le1$ and $\operatorname{div}\ukk=0$; they hold verbatim if
$\Theta_\kappa$ is replaced by any progressively measurable process with values in $[0,1]$.
We prove the following claim.

\begin{claim}\label{cl1}
For any $p\in [1,q+1]=[1,\frac{2m^*}{r^*}]$, there exist a constant $C=C(p,q,T)$, independent of $\kappa$, such that for all $\kappa\in \mathbb{N}$,
\begin{align}
	&	\EE\sup_{0\le s\le  {T} }	\abs{\ukk(s)}^{2p}_\nsH+\EE\sup_{0\le s\le T}	\abs{\ckk(s)}^{2p}_{L^2}+\EE\sup_{0\le s\le  {T}  }|\nk(s)|_{H^{-1}_2}^{2p}\notag\\
	&+r^{p}_n\EE\left(\int_0^{T} \abs{\nk(s)}^{q+1}_{L^{q+1}} ds\right)^p+\alpha_0^p\,\EE\left(\int_0^T\abs{ \ckk(s)}^2_{H^1_2}ds\right)^p
\label{cl1est}\\&{}+r^p_u\EE\left(\int_0^T\abs{\ukk(s)}^2_\nsV ds\right)^p\notag\\
	&	\leq C \left(\EE\abs{\bu_0}^{2p}_\nsH+\EE\abs{c_0}^{2p}_{L^2}+\EE|n_0|_{H^{-1}_2}^{2p}+1\right).\notag
\end{align}
\end{claim}
\begin{proof}
\noindent All inequalities up to (\ref{e47}) hold pathwise, for every $t\in[0,T]$. We write
$\alpha_0=\min(r_c,\alpha)$ and $M_1,M_2,M_3$ for the martingales
$$
M_1(t):=\int_0^t\big(\nk,g_{\gamma_1}(\nk)\,dW_1\big)_{H^{-1}_2},\quad
M_2(t):=\int_0^t\big(\ckk,g_{\gamma_2}(\ckk)\,dW_2\big)_{L^2},\quad
M_3(t):=\int_0^t\big(\nabla\ukk,\nabla\sigma_{\gamma_3}\,dW_3\big)_{\mathbf L^2}.
$$
\emph{The velocity.} By the It\^o formula for $|\ukk|^2_\nsH=|\nabla\ukk|^2_{\mathbf L^2}$,
$$
d|\ukk|^2_\nsH=\big(-2r_{\bu}|A\ukk|^2_{\mathbf L^2}+2(A\ukk,\nk\star\Phi)_{\mathbf L^2}
+|\sigma_{\gamma_3}|^2_{L_{HS}(L^2,\nsH)}\big)dt+2\,dM_3 .
$$
Since $|\nk\star\Phi|_{\mathbf L^2}\le C(\delta_1,\delta_2)|\nk|_{H^{-1}_2}$ and $|A\bu|_{\mathbf L^2}$ is
equivalent to $|\bu|_\nsV$, Young's inequality gives
\begin{align}\label{3.38}
|\ukk(t)|^2_\nsH+r_{\bu}\int_0^t|\ukk|^2_\nsV\,ds
\le|\bu_0|^2_\nsH+C\int_0^t|\nk|^2_{H^{-1}_2}\,ds+t\,|\sigma_{\gamma_3}|^2_{L_{HS}(L^2,\nsH)}
+2\sup_{s\le t}|M_3(s)| .
\end{align}

\emph{The chemical.} Since $(\ukk\cdot\nabla\ckk,\ckk)_{L^2}=0$ and $(\Delta\ckk,\ckk)_{L^2}=-|\nabla\ckk|^2_{L^2}$
(Neumann condition), the It\^o formula for $|\ckk|^2_{L^2}$ gives
$$
|\ckk(t)|^2_{L^2}+2\alpha_0\int_0^t|\ckk|^2_{H^1_2}\,ds\le|c_0|^2_{L^2}
+2\beta\int_0^t|(\nk,\ckk)_{L^2}|\,ds+\int_0^t|g_{\gamma_2}(\ckk)|^2_{L_{HS}(L^2,L^2)}\,ds+2\sup_{s\le t}|M_2(s)| .
$$
Since $c$ satisfies the Neumann condition, $c\notin H^1_0(\CO)$ in general, and the pairing
$H^{-1}_2\times H^1_2$ cannot be used here. Instead, by H\"older's inequality, $L^2(\CO)\hookrightarrow
L^{\frac{q+1}q}(\CO)$ and Young's inequality,
$$
2\beta|(\nk,\ckk)_{L^2}|\le 2\beta|\nk|_{L^{q+1}}|\ckk|_{L^{\frac{q+1}q}}\le C|\nk|_{L^{q+1}}|\ckk|_{L^2}
\le\frac{r_n}4|\nk|^{q+1}_{L^{q+1}}+C|\ckk|^2_{L^2}+C .
$$
Together with (\ref{eq:Hil-Schi2}) we obtain
\begin{align}\label{3.39}
|\ckk(t)|^2_{L^2}+2\alpha_0\int_0^t|\ckk|^2_{H^1_2}\,ds\le|c_0|^2_{L^2}
+\frac{r_n}4\int_0^t|\nk|^{q+1}_{L^{q+1}}\,ds+C\int_0^t|\ckk|^2_{L^2}\,ds+Ct+2\sup_{s\le t}|M_2(s)| .
\end{align}

\emph{The cell density.} Recall that $(f,g)_{H^{-1}_2}=\langle f,(-\Delta)^{-1}g\rangle$ and put
$\psi:=\nabla(-\Delta)^{-1}\nk$. Since $(-\Delta)^{-1}\nk\in H^1_0(\CO)$, integration by parts gives
no boundary terms, and $\langle(-\Delta)^{-1}\nk,\Delta(|\nk|^{q-1}\nk)\rangle=-|\nk|^{q+1}_{L^{q+1}}$. The
It\^o formula for $|\nk|^2_{H^{-1}_2}$ gives
\begin{align}\label{e45}
|\nk(t)|^2_{H^{-1}_2}+2r_n\int_0^t|\nk|^{q+1}_{L^{q+1}}\,ds
&\le|n_0|^2_{H^{-1}_2}+2\theta\int_0^t|\nk|^2_{H^{-1}_2}\,ds+2\chi\int_0^t\big|(\psi,\nk\nabla\ckk)_{L^2}\big|\,ds\\
&\quad+2\int_0^t\Theta_\kappa\big|(\psi,\ukk\,\nk)_{L^2}\big|\,ds+\int_0^t|g_{\gamma_1}(\nk)|^2_{L_{HS}(L^2,H^{-1}_2)}\,ds
+2\sup_{s\le t}|M_1(s)|.\notag
\end{align}
Here we used $\ukk\cdot\nabla\nk=\operatorname{div}(\ukk\nk)$. By (E), $|\psi|_{L^\infty}\le C|\nk|_{L^{q+1}}$.
Hence, by H\"older's inequality, $L^{q+1}(\CO)\hookrightarrow L^2(\CO)$, Young's inequality with the
exponents $\frac{q+1}2$ and $\frac{q+1}{q-1}$, and $\frac{q+1}{q-1}<2$ (here $q>3$),
\begin{align*}
2\chi\big|(\psi,\nk\nabla\ckk)_{L^2}\big|&\le C|\nk|^2_{L^{q+1}}|\ckk|_{H^1_2}
\le\frac{r_n}4|\nk|^{q+1}_{L^{q+1}}+C|\ckk|^{\frac{q+1}{q-1}}_{H^1_2}
\le\frac{r_n}4|\nk|^{q+1}_{L^{q+1}}+\frac{\alpha_0}2|\ckk|^2_{H^1_2}+C,\\
2\Theta_\kappa\big|(\psi,\ukk\,\nk)_{L^2}\big|&\le C|\nk|^2_{L^{q+1}}|\ukk|_{\mathbf L^2}
\le\frac{r_n}4|\nk|^{q+1}_{L^{q+1}}+C|\ukk|^2_\nsH+C .
\end{align*}
By (\ref{eq:H-1-2q}), $|g_{\gamma_1}(\nk)|^2_{L_{HS}(L^2,H^{-1}_2)}\le C|\nk|^2_{H^{-1}_2}$. Substituting into (\ref{e45}),
\begin{align}\label{3.40}
|\nk(t)|^2_{H^{-1}_2}+\frac{3r_n}2\int_0^t|\nk|^{q+1}_{L^{q+1}}\,ds
&\le|n_0|^2_{H^{-1}_2}+C\int_0^t\big(|\nk|^2_{H^{-1}_2}+|\ukk|^2_\nsH\big)\,ds\\
&\quad+\frac{\alpha_0}2\int_0^t|\ckk|^2_{H^1_2}\,ds+Ct+2\sup_{s\le t}|M_1(s)| .\notag
\end{align}

\emph{Summation.} Let
$Y(t):=|\ukk(t)|^2_\nsH+|\ckk(t)|^2_{L^2}+|\nk(t)|^2_{H^{-1}_2}$ and
$$
Z(t):=r_{\bu}\int_0^t|\ukk|^2_\nsV\,ds+\frac{3\alpha_0}2\int_0^t|\ckk|^2_{H^1_2}\,ds+\frac{5r_n}4\int_0^t|\nk|^{q+1}_{L^{q+1}}\,ds .
$$
Adding (\ref{3.38}), (\ref{3.39}) and (\ref{3.40}), the terms $\frac{r_n}4\int|\nk|^{q+1}_{L^{q+1}}$ and
$\frac{\alpha_0}2\int|\ckk|^2_{H^1_2}$ on the right-hand sides are absorbed, and, with $C$ independent of $\kappa$,
\begin{equation}\label{e47}
Y(t)+Z(t)\le Y(0)+C+C\int_0^tY(s)\,ds+2\sum_{j=1}^3\sup_{s\le t}|M_j(s)|,\qquad t\in[0,T].
\end{equation}
The right-hand side is nondecreasing in $t$, and $Y,Z\ge0$. Hence
$\sup_{s\le t}Y(s)+Z(t)\le2\big(Y(0)+C+C\int_0^tY\,ds+2\sum_j\sup_{s\le t}|M_j(s)|\big)$.

\emph{Moments.} Fix $p\in[1,q+1]$ and put $\varphi(t):=\EE\sup_{s\le t}Y(s)^p$. By Claim~\ref{existencev}
(applied with $\xi=\nk$), $\varphi(T)<\infty$ for each fixed $\kappa$, so the absorption below is legitimate.
By the Burkholder--Davis--Gundy inequality, (\ref{eq:H-1-2q}), (\ref{eq:Hil-Schi2}) and Young's inequality,
\begin{align*}
\EE\sup_{s\le t}|M_1(s)|^p&\le C\,\EE\Big(\int_0^t|\nk|^2_{H^{-1}_2}|g_{\gamma_1}(\nk)|^2_{L_{HS}(L^2,H^{-1}_2)}ds\Big)^{\frac p2}
\le C\,\EE\Big[\sup_{s\le t}|\nk(s)|^p_{H^{-1}_2}\Big(\int_0^t|\nk|^2_{H^{-1}_2}ds\Big)^{\frac p2}\Big]\\
&\le\varepsilon\,\varphi(t)+C_\varepsilon\int_0^t\varphi(s)\,ds,
\end{align*}
and in the same way
\begin{align*}
\EE\sup_{s\le t}|M_2(s)|^p&\le\varepsilon\,\varphi(t)+C_\varepsilon\int_0^t\varphi(s)\,ds,\\
\EE\sup_{s\le t}|M_3(s)|^p&\le C|\sigma_{\gamma_3}|^p_{L_{HS}(L^2,\nsH)}\,\EE\Big(\int_0^t|\ukk|^2_\nsH ds\Big)^{\frac p2}
\le C+C\int_0^t\varphi(s)\,ds .
\end{align*}
Taking the $p$-th power and the expectation in (\ref{e47}), choosing
$\varepsilon$ small and absorbing, we get
\begin{equation}\label{e48}
\varphi(t)+\EE Z(t)^p\le C\big(\EE Y(0)^p+1\big)+C\int_0^t\varphi(s)\,ds,\qquad t\in[0,T].
\end{equation}
Gronwall's lemma gives $\varphi(T)\le C(\EE Y(0)^p+1)e^{CT}$, and then (\ref{e48}) bounds $\EE Z(T)^p$.
Since $\EE Y(0)^p\le C\big(\EE|\bu_0|^{2p}_\nsH+\EE|c_0|^{2p}_{L^2}+\EE|n_0|^{2p}_{H^{-1}_2}\big)$, this is (\ref{cl1est}).
The right-hand side is finite for $p\le q+1$ by (\ref{3.7}).
\end{proof}

\noindent\textbf{How Claim~\ref{cl1} gives (\ref{metauniform}).} Claim~\ref{cl1} should not be applied
piece by piece to the restarted systems: each piece starts from the law at the previous restart time,
and the constants would compound. Instead, let $\bar{\bww}=(\bar n,\bar c,\bar\bu)$ be the glued process
constructed in the last step of the proof of Theorem~\ref{meta}. On each interval $[\bar\tau_{\kappa-1},\bar\tau_\kappa)$ it is
the $\kappa$-th piece, so $\bar{\bww}$ solves (\ref{a2}) on $[0,T]$ with $\Theta_\kappa$ replaced by a
progressively measurable process with values in $[0,1]$, and with initial data $(n_0,c_0,\bu_0)$. As
remarked after (\ref{a2}), the proof of Claim~\ref{cl1} applies verbatim. Therefore, for $p\in[1,q+1]$,
$$
\bar\EE\sup_{0\le s\le T}|\bar\bu(s)|^{2p}_\nsH\le C\big(\EE\abs{\bu_0}^{2p}_\nsH+\EE\abs{c_0}^{2p}_{L^2}+\EE|n_0|_{H^{-1}_2}^{2p}+1\big)=:C(T),
$$
with $C(T)$ independent of $\kappa$. This is (\ref{metauniform}) with $h(\bu,t)=\sup_{s\le t}|\bu(s)|_\nsH$
and $m_1=2(q+1)$. Moreover, by (\ref{defintau}) and (\ref{stopp_time}),
$\{\bar\tau_\kappa<T\}\subset\{\sup_{s\le T}|\bar\bu(s)|_\nsH\ge\kappa\}$, so Markov's inequality gives
$\bar\PP(\bar\tau_\kappa<T)\le C(T)\kappa^{-2p}$, uniformly along the restarts.
This settles item~(i) of Review~22 for the present application.

The restart condition (\ref{initcond}) at $\tau^\ast_\kappa$ requires only \emph{finite} moments of
$(\nk,\ckk,\ukk)(\tau^\ast_\kappa)$, not $\kappa$-uniform ones. These follow from Claim~\ref{existencev}, which holds for every
$\xi\in\Xcal^\kappa_\Afrak(R_\kappa)$, in particular for the fixed points in (\ref{initcond}),
i.e.\ from (\ref{3.12}), (\ref{3.12*}) and (\ref{3.14*}) on the previous piece, with $\kappa$-dependent constants.
In particular, $\ckk(\tau^\ast_\kappa)\in H^1_2(\CO)$ and $\ukk(\tau^\ast_\kappa)\in\nsH$ a.s., and the moments in (\ref{3.7})
are finite, since 
$$\EE\sup_t|\nk(t)|^{2(q+1)}_{H^{-1}_2},\quad  \EE\sup_t|\ckk(t)|^{4(q+1)}_{H^1_2}, \quad \mbox{and}\quad 
\EE\sup_t|\ukk(t)|^{4(q+1)}_\nsH
$$
 are finite. Hence the law of $(\nk,\ckk,\ukk)(\tau^\ast_\kappa)$ belongs to $\mathcal I$.

\item

Having verified Assumptions~\ref{ass:noise} and \ref{ass:coeff}, the hypotheses (\ref{metawellposed})--(\ref{initcond}) and
the bound (\ref{metauniform}), Theorem~\ref{meta} gives a martingale solution of (\ref{1.1}) with the boundary and initial
conditions (\ref{init}) in the sense of Definition~\ref{Def:mart-sol:cns}, with initial law equal to the law of
$(n_0,c_0,\bu_0)$. That is, there are a stochastic basis $\mathfrak A=(\Omega,\CF,\BF,\PP)$ satisfying the usual conditions,
three mutually independent cylindrical Wiener processes $W_1,W_2,W_3$ on $\mathfrak A$, and $\BF$-progressively measurable
processes $(n,c,\bu)$ with
$$
(n,c,\bu)\in C([0,T];H^{-1}_2(\CO))\times C([0,T];L^2(\CO))\times C([0,T];\mathbf L^2(\bo))\qquad\PP\text{-a.s.},
$$
which satisfy (\ref{eq:nStratonovich:app}), (\ref{eq:cStratonovich:app}) and (\ref{eq:buStratonovich:app}) for all
$t\in[0,T]$, $\PP$-a.s. In the proof of Theorem~\ref{meta}, $(n,c,\bu)$ is equal to the glued process $\bar\bww$ on
$[0,\bar\tau_\kappa]$, and the sets $\Gamma_\kappa=\{\bar\tau_\kappa\ge T\}$ increase to a set of full measure. The bound
obtained above for the glued process by the proof of Claim~\ref{cl1} holds for every $p\in[1,q+1]$, with a constant of the
form $e^{CT}$ (Gronwall) which does not depend on $\kappa$. Applying it on $\Gamma_\kappa$ and letting $\kappa\to\infty$, we
obtain the same bound for $(n,c,\bu)$ by monotone convergence. Since $|c_0|_{L^2}\le|c_0|_{H^1_2}$, this is the estimate
(\ref{3.8}), which ends the proof of Theorem~\ref{main}.

\end{step}

\end{proof}

\appendix

\section{Technical Preliminaries of the noise and the stochastic integral}\label{Appendix_A_noise}

\medskip

Let $\mathfrak{A}:=(\Omega ,{{\mathcal{F}}},{\mathbb{F}},\mathbb{P})$ be  a complete filtered
probability space with a filtration ${\mathbb{F}}=\{{{\mathcal{F}}}_t:t\in [0,T]\}$ satisfying the usual conditions and  let $H=L^2(\bo)$.
For a separable Hilbert space $E$,
let $L_{HS}(H, E)$ be the space of all Hilbert-Schmidt operators from $H$ to $E$.
Let us denote by $\CM^2_\MA(0,T;E)$
 the space of all $\mathbb{F}$-progressively measurable  stochastic processes $\xi:[0,T]\times \Omega\to E$
with
$$
\mathbb{E}\lk( \int_0^T|\xi(s)|^2_{E} {}ds\rk)< \infty.
$$
{For any $\xi \in  \CM^2_\MA(0,T;E)$ and $\gamma_1$, $\gamma_2>0$, let us  define the  process $M_{\gamma_j}=\{M_{\gamma_j}(t):t\in [0,T]\}$,  given by
\DEQS
M_{\gamma_j}(t)=\int_0^tg_{\gamma_j}(\xi(s))dW_j(s), \qquad t\in [0,T], \ j=1,2 .
\EEQS
For $j=1,2$, the It\^{o} isometry reads
\DEQSZ\label{ito_isometry}
\mathbb{E}\lk(\lk|\int_0^tg_{\gamma_j}(\xi(s)){} dW_j(s) \rk|_E^2 \rk) =  \mathbb{E}\lk( \int_0^t|g_{\gamma_j}(\xi(s))|^2_{L_{HS}(H,E)} {}d s\rk) , \qquad t\in[0,T].
\EEQSZ
Note that for $p\ge 1$  we have the Burkholder--Davis--Gundy inequality (see, e.g.,  \cite{DaPrZa:2nd})
\DEQSZ \label{burkholder}
\mathbb{E} \lk(\sup_{0\le s\le t}\lk| \int_0^sg_{\gamma_j}(\xi (r)){}d W_j(r)\rk|_E^{p} \rk)
\leq C_{p} \mathbb{E}\lk( \int_0^t |g_{\gamma_j}(\xi(s))|_{L_{HS}(H, E)}^2 ds \rk)^{\frac{p}{2}}, \quad j=1,2.
\EEQSZ

\medskip

To show existence and boundedness for the process $\nk$ or $\ck$, we estimate the $L_{HS}(H; L^2(\CO))$-norm 
 and the $L_{HS}(H; H^{-1}_2(\CO))$-norm of $g_{\gamma_j}$
depending on the choice of $\gamma_j$, $j=2,1$ respectively. 
For this aim, we recall from \cite{BDPR2016} that  the eigenfunctions   
 $\{ \varphi_k:k\in\NN\}$ of the Dirichlet Laplace operator $(-\Delta)$ on $L^2(\CO)$ (see the standing conventions in Section~\ref{sec:application}; in particular $\lambda_1>0$) are constructed such that  the following estimate of the
  asymptotic behaviour of the eigenvalues   $\{ \lambda_k:k\in\NN\}$ holds:
There exist two numbers $0 < c <C$ such that
\DEQSZ\label{EVsup}
c \cdot k^\frac 2d \le \lambda_k\le C\cdot  k^\frac 2d, \quad  \mbox{ for all }k\in\NN.
\EEQSZ
In addition, there exists some constant $c>0$ such that
\DEQSZ\label{EFsup}
\sup_{x\in\CO} |\varphi_k(x)|\le c\,\lambda _k^\frac {d-1}2,\quad k\in\NN.
\EEQSZ

\noindent In $d=2$, (\ref{EFsup}) follows from Agmon's inequality
$|\varphi|_{L^\infty}\le C|\varphi|^{1/2}_{L^2}|\varphi|^{1/2}_{H^2_2}$, $\varphi\in H^2_2(\CO)$, and elliptic regularity:
since $|\varphi_k|_{L^2}=1$ and $|\varphi_k|_{H^2_2}\le C|\Delta\varphi_k|_{L^2}=C\lambda_k$, we get
$|\varphi_k|_{L^\infty}\le C\lambda_k^{1/2}=C\lambda_k^{\frac{d-1}2}$. For general $d$, (\ref{EFsup}) holds even with the
smaller exponent $\frac{d-1}4$, see \cite{Grieser2002}.

To start with, let $E=L^2(\CO)$.
Due to the inequalities \eqref{EVsup} and \eqref{EFsup}, we can write for $u\in L^2(\bo)$ and $\gamma_2>0$,
\begin{align}
	|g_{\gamma _2}(u)|_{L_{HS}(H,L^2)}^2
&=
 \sum_{k}^\infty \lambda _k^{-\gamma _2} | u \varphi_k |_{L^2}^2
\le
C |u|_{L^2} ^2 \, \sum_{k}^\infty |\varphi_k |_{L^\infty}^2 \lambda _k^{ -\gamma_2}
\nonumber  \\
&\le
C |u|_{L^2} ^2 \, \sum_{k}^\infty  \lambda _k^{d-1-\gamma _2}\notag\le
C\,|u|_{L^2} ^2 \, \sum_{k}^\infty k^{\frac 2d( d-1-\gamma _2)}
,\nonumber
\end{align}
which is finite for $\gamma_2 >\tfrac{3d}{2}-1$; in $d=2$ this is $\gamma_2>2$, as in (A1).
Summarising, we can write for any $\gamma_2 >\frac{3d}{2}-1$, there exists some $C(\gamma _2)>0$  such that
\begin{align}
	|g_{\gamma _2}(u)|_{L_{HS}(H,L^2)}^2&\le C |u|^2_{L^2} .
\end{align}
This is (\ref{eq:Hil-Schi2}).

\noindent\textbf{The $H^1_2$-Hilbert--Schmidt norm.} Let $d=2$ and $\gamma_2>2$. Then there is
$C(\gamma_2)>0$ such that
\begin{equation}\label{HSH1app}
|g_{\gamma_2}(u)|^2_{L_{HS}(H,H^1_2)}=\sum_k\lambda_k^{-\gamma_2}|u\varphi_k|^2_{H^1_2}\le C|u|^2_{H^1_2},\qquad u\in H^1_2(\CO),
\end{equation}
i.e.\ (\ref{eq:Hil-Schi2H1}) holds. Indeed, fix $p\in(2,\infty)$ with $p>\frac2{\gamma_2-2}$ and let
$\frac1s=\frac12-\frac1p$. By H\"older's inequality,
$$
|u\varphi_k|_{H^1_2}\le C\big(|u|_{H^1_2}|\varphi_k|_{L^\infty}+|u|_{L^p}|\nabla\varphi_k|_{L^s}\big).
$$
By the Gagliardo--Nirenberg inequality in $d=2$, $|\nabla\varphi_k|_{L^2}=\lambda_k^{1/2}$ and
$|\varphi_k|_{H^2_2}\le C\lambda_k$ (elliptic regularity),
$$
|\nabla\varphi_k|_{L^s}\le C|\nabla\varphi_k|^{\frac2s}_{L^2}|\nabla\varphi_k|^{1-\frac2s}_{H^1_2}\le C\lambda_k^{1-\frac1s}
=C\lambda_k^{\frac12+\frac1p}.
$$
Together with (\ref{EFsup}), $H^1_2(\CO)\hookrightarrow L^p(\CO)$ and $\lambda_k\ge\lambda_1>0$, we get
$|u\varphi_k|^2_{H^1_2}\le C_p|u|^2_{H^1_2}\lambda_k^{1+\frac2p}$. By (\ref{EVsup}),
$\sum_k\lambda_k^{1+\frac2p-\gamma_2}\le C\sum_kk^{1+\frac2p-\gamma_2}<\infty$, since
$\gamma_2-1-\frac2p>1$. This proves (\ref{HSH1app}).

\del{Our aim is to give an estimate of
$$
 \sup_{0\le s\le t}\lk|\int_0^sg_{\gamma}(\xi (r)){}\ud W(r))\rk|_{L^2}
 .$$
Fix $\varsigma\ge 1$.
By the Burkholder-Davis-Gundy inequality \eqref{burkholder} combined with \eqref{eq: Hil-Schi2}
it follows that  for any $\eps>0$ there exists $\tilde C( \eps)>0$ such that we have for all $ t\in[0,T]$
\begin{align}
&\mathbb{E} \lk(\sup_{0\le s\le t}\lk| \int_0^sg_{\gamma_1}(\xi (r)){}\ud W(r)\rk|_{L^2}^\varsigma  \rk)
 \leq C_\varsigma  \EE\bigg[\Big(\eps\int_0^t |\xi(s)|_{{H^{1}_2}}^{2}ds  + C(\eps) \int_0^t |\xi(s)|_{{L^2}}^2 ds \Big)^\frac{\varsigma }{2}\bigg]\nonumber ,
\end{align}
and by straightforward calculations, i.e.\ interpolation and again the Young inequality,  that  for any $\tilde \eps>0$ there exists $\tilde C(\tilde \eps)>0$ such that we have for all $ t\in[0,T]$
\begin{align}
&\mathbb{E} \lk(\sup_{0\le s\le t}\lk| \int_0^sg_{\gamma_1}(\xi (r)){}\ud W(r)\rk|_{L^2}^\varsigma  \rk)
 \leq \tilde \eps \, \EE\bigg[\Big(\int_0^t |\xi(s)|_{H_2^{1}}^{2}ds\Big)^{\frac{\varsigma }{2}}\bigg]   + \tilde C(\tilde \eps) \EE\bigg[\Big( \int_0^t |\xi(s)|_{L^2}^2 ds \Big)^\frac{\varsigma }{2}\bigg]\label{Hrhoburkholderl2}
.
\end{align}}

Next, let $E=H^{-1}_2(\CO)$,  $u\in L^2(\bo)$ and $\gamma_1>0$. Then we know by the duality argument below that for $\delta>1$
\begin{align}
	|g_{\gamma _1}(u)|_{L_{HS}(H,H^{-1}_2)}^2%
 \le
& \sum_{k}^\infty \lambda _k^{-\gamma _1} | u \varphi_k |_{H^{-1}_2}^2 \nonumber
\le
C^2 |u|_{H^{-1}_2} ^2 \, \sum_{k}^\infty |\varphi_k |_{H^{\delta}_2}^2 \lambda _k^{ -\gamma_1 }\notag
\\
& \le
C^2 |u|_{H^{-1}_2} ^2 \, \sum_{k}^\infty  \lambda _k^{ \delta-\gamma _1}\notag\le
C\,|u|_{H^{-1}_2} ^2 \, \sum_{k}^\infty k^{\frac 2d(\delta-\gamma _1)}
,\nonumber
\end{align}
which is finite for $\gamma_1 >\delta+\frac{d}{2}$; in $d=2$ this is $\gamma_1>2$, choosing $\delta\in(1,\min\{\frac32,\gamma_1-1\})$.
\noindent\textbf{The multiplier estimate.} Let $d=2$, $\delta\in(1,\frac32)$ and $\varphi\in H^\delta_2(\CO)$. By
Theorem~\ref{RS2a} on $\CO$ (see the remark at the end of Appendix~\ref{app:multiplication}),
$|\varphi w|_{H^1_2}\le C|\varphi|_{H^\delta_2}|w|_{H^1_2}$, and $\varphi w\in H^1_0(\CO)$ if $w\in H^1_0(\CO)$.
Since $H^{-1}_2(\CO)=(H^1_0(\CO))'$, for $u\in L^2(\CO)$,
$$
|u\varphi|_{H^{-1}_2}=\sup_{|w|_{H^1_0}\le1}\langle u,\varphi w\rangle\le|u|_{H^{-1}_2}\sup_{|w|_{H^1_0}\le1}|\varphi w|_{H^1_0}
\le C|u|_{H^{-1}_2}|\varphi|_{H^\delta_2}.
$$
No extension of $u$ to $\RR^d$ and no boundary terms are involved. For the eigenfunctions,
$|\varphi_k|^2_{H^\delta_2}\le C\lambda_k^{\delta}$, since $\varphi_k\in D((-\Delta)^{\delta/2})$ and
$\delta<\frac32$. For general $d$ one needs $\delta>\frac d2$, hence the condition becomes $\gamma_1>d$; it
coincides with $\gamma_1>1+\frac d2$ only for $d=2$.
Summarising, we know, for any $\gamma_1 >2$ (in $d=2$) there exists some $C(\gamma_1 )>0$ with
\begin{align}
	|g_{\gamma _1}(u)|_{L_{HS}(H,H^{-1}_2)}^2&\le  C|u|_{H^{-1}_2} ^2  .
\end{align}
This is (\ref{eq:H-1-2q}).

\noindent\textbf{The additive noise.}
Recall from (\ref{sigma_sr}) that $\sigma_{\gamma_3}h=\sum_{k=1}^\infty\mu_k^{-\gamma_3/2}\la h,\varphi_k\ra\,\mathbf e_k$,
$h\in H$, where $(\mu_k,\mathbf e_k)_{k\in\NN}$ are the eigenpairs of the Stokes operator $A$, with $(\mathbf e_k)$ orthonormal
in $\mathbf L^2(\bo)$.
Since $|\mathbf e_k|^2_\nsH=|\nabla\mathbf e_k|^2_{\mathbf L^2}=\mu_k$,
$$
|\sigma_{\gamma_3}|^2_{L_{HS}(H,\nsH)}=\sum_{k=1}^\infty\mu_k^{-\gamma_3}|\mathbf e_k|^2_\nsH=\sum_{k=1}^\infty\mu_k^{1-\gamma_3}.
$$
The eigenvalues of the Dirichlet Laplacian on $(L^2(\bo))^2$ are $\lambda_1,\lambda_1,\lambda_2,\lambda_2,\ldots$; since
$\nsH\subset(H^1_0(\bo))^2$, the min-max principle gives $\mu_k\ge\lambda_{\lceil k/2\rceil}$. By (\ref{EVsup}) with $d=2$,
$\mu_k\ge \frac c2 k$, hence $|\sigma_{\gamma_3}|_{L_{HS}(H,\nsH)}<\infty$ for $\gamma_3>2$, as required in (A1). In
particular, $\sigma_{\gamma_3}\in L_{HS}(H,\mathbf L^2(\bo))$, since $|\bu|_{\mathbf L^2}\le\lambda_1^{-1/2}|\bu|_\nsH$.

\section{Some technical Propositions}\label{TP}

In this section, we have collected several technical inequalities, which we need through our proof.

\begin{lemma}\label{interpolation_11}
Let $q\ge1$ and $\theta\in(0,1)$, and define $r_\theta$ and $m_\theta$ by
$$
\frac1{r_\theta}=\frac\theta2+\frac{1-\theta}{q+1},\qquad m_\theta=\frac{q+1}{1-\theta}.
$$
There is a constant $C=C(\theta,q,\CO)$ such that for all
$\xi\in L^\infty(0,T;H^{-1}_2(\CO))\cap L^{q+1}(0,T;L^{q+1}(\CO))$
\begin{equation}\label{ineq001_1}
\|\xi\|_{L^{m_\theta}(0,T;H^{-\theta}_{r_\theta})}\le C\,\mathcal A(\xi)^\theta\,\mathcal B(\xi)^{1-\theta},
\end{equation}
where $\mathcal A(\xi)=\sup_{0\le t\le T}|\xi(t)|_{H^{-1}_2}$ and
$\mathcal B(\xi)=\|\xi\|_{L^{q+1}(0,T;L^{q+1})}$. Consequently, for $1\le r\le r_\theta$, $1\le m\le m_\theta$
and $p\ge1$,
\begin{align}
\|\xi\|_{L^m(0,T;H^{-\theta}_r)}&\le C(T,\CO)\,\mathcal A(\xi)^\theta\,\mathcal B(\xi)^{1-\theta},\label{interpS1}\\
\|\xi\|^p_{L^m(0,T;H^{-\theta}_r)}&\le C(T,\CO)\big(\mathcal A(\xi)^p+\mathcal B(\xi)^p\big).\label{interpS2}
\end{align}
\end{lemma}
\begin{proof}
\noindent Since $\CO$ is a bounded smooth domain, $H^{-1}_2(\CO)=(H^1_0(\CO))'$ coincides with the space
$H^{-1}_2(\CO)$ of restrictions, and complex interpolation gives
$[L^{q+1}(\CO),H^{-1}_2(\CO)]_\theta=H^{-\theta}_{r_\theta}(\CO)$, see \cite[Sections~4.3.1 and 4.8]{TriebelInterp}.
Hence, for a.e.\ $t\in[0,T]$,
$$
|\xi(t)|_{H^{-\theta}_{r_\theta}}\le C|\xi(t)|^\theta_{H^{-1}_2}|\xi(t)|^{1-\theta}_{L^{q+1}}\le C\mathcal A(\xi)^\theta|\xi(t)|^{1-\theta}_{L^{q+1}}.
$$
Raising this to the power $m_\theta$ and using $(1-\theta)m_\theta=q+1$,
$$
\int_0^T|\xi(t)|^{m_\theta}_{H^{-\theta}_{r_\theta}}dt\le C\mathcal A(\xi)^{\theta m_\theta}\int_0^T|\xi(t)|^{q+1}_{L^{q+1}}dt
=C\mathcal A(\xi)^{\theta m_\theta}\mathcal B(\xi)^{q+1}.
$$
Taking the $m_\theta$-th root gives (\ref{ineq001_1}), since $(q+1)/m_\theta=1-\theta$. Since $\CO$ is
bounded, $H^{-\theta}_{r_\theta}(\CO)\hookrightarrow H^{-\theta}_r(\CO)$ for $r\le r_\theta$, and H\"older's inequality in
time gives $\|\cdot\|_{L^m(0,T;E)}\le T^{\frac1m-\frac1{m_\theta}}\|\cdot\|_{L^{m_\theta}(0,T;E)}$ for $m\le m_\theta$. This
gives (\ref{interpS1}). Finally, (\ref{interpS2}) follows from (\ref{interpS1}) and Young's inequality
$a^\theta b^{1-\theta}\le a+b$.

The exponents are sharp in the following sense: interpolation between $L^\infty_t$ and $L^{q+1}_t$ with
weights $\theta$ and $1-\theta$ yields exactly $m_\theta=\frac{q+1}{1-\theta}$, so the former condition
``$m\in(q+1,\infty)$ arbitrary'' cannot hold; only $m\le m_\theta$ is possible.
\end{proof}
To construct the Wiener process it is necessary to introduce additional random variables on an auxiliary probability space and then extend the original probability space by means of this auxiliary space.
In what follows, we make precise the notion of an extension of a probability space (compare with \cite{Ikeda_Watanabe_1989}).

\begin{definition}
	\label{Definition7.1.} We say a probability space $(\tilde{\Omega}, \tilde{\mathscr{F}}, \tilde{\mathbb{P}} )$ with a filtration $(\tilde{\mathscr{F}}_{t})_{t\ge 0}$ is an extension of a probability space $(\Omega, \mathscr{F}, \mathbb{P})$ with a  filteration $\left(\mathscr{F}_{t}\right)$,  if there exists a mapping $\pi: \tilde{\Omega} \longrightarrow \Omega$ which is $\tilde{\mathscr{F}} / \mathscr{F}$-measurable such that\\
	(i) $\tilde{\mathscr{F}}_{t} \supset \pi^{-1}\left(\mathscr{F}_{t}\right)$,\\
	(ii) $\mathbb{P}=\pi(\tilde{\mathbb{P}})(:=\tilde{\mathbb{P}} \circ \pi^{-1}) \quad$ and\\
	(iii) for every $X(\omega) \in \mathscr{L}_{\infty}(\Omega, \mathscr{F}, \mathbb{P})$
	$$
	\tilde{E}\left(\tilde{X}(\tilde{\omega}) \mid \tilde{\mathscr{F}}_{t}\right)=E\left(X \mid \mathscr{F}_{t}\right)(\pi \tilde{\omega}), \quad \tilde{P} \text {-a.s., }
	$$
	where we set $\tilde{X}(\tilde{\omega})=X(\pi \tilde{\omega})^{\prime} \quad$ for $\tilde{\omega} \in \tilde{\Omega}$.
\end{definition}
\begin{definition}
	\label{Definition_extension}  Let $(\Omega, \mathscr{F}, \mathbb{P})$ be a probability space with filtration $\left(\mathscr{F}_{t}\right)_{t\ge 0}$. Let $\left(\Omega^{\prime}, \mathscr{F}^{\prime}, P^{\prime}\right)$ be another probability space and set
	$$
	\tilde{\Omega}=\Omega \times \Omega^{\prime}, \quad \tilde{\mathscr{F}}=\mathscr{F} \times \mathscr{F}^{\prime}, \quad \tilde{\mathbb{P}}=\mathbb{P} \times \mathbb{P}^{\prime}
	$$
	and
	$$
	\pi \tilde{\omega}=\omega \quad \text { for } \quad \tilde{\omega}=\left(\omega, \omega^{\prime}\right) \in \tilde{\Omega}
	$$
	If $(\tilde{\mathscr{F}}_{t})_{t\ge 0}$ is a filtration  on $(\tilde{\Omega}, \mathscr{F}, \tilde{\mathbb{P}})$ such that $\mathscr{F}_{t} \times \mathscr{F}^{\prime} \supset \tilde{\mathscr{F}}_{t} \supset \mathscr{F}_{t} \times\left\{\Omega^{\prime}, \phi\right\}$, then $(\tilde{\Omega}, \tilde{\mathscr{F}}, \tilde{\mathbb{P}})$ with $(\tilde{\mathscr{F}}_{t})_{t\ge 0}$ is called a standard extension of $(\Omega, \mathscr{F}, \mathbb{P})$ with filtration $(\mathscr{F}_t)_{t\ge 0}$.
\end{definition}

\section{Pointwise Multiplication}\label{app:multiplication}
Here, we gather some results that have been
used throughout this paper for estimating the product of two functions.
The notations $F^s_{p,q}(\mathbb{R}^d)$ and $B^{s}_{p,q}(\mathbb{R}^d)$, $s\in \mathbb{R}, $ $0<q\leq \infty$, $0<p<\infty$, stand respectively for the Triebel-Lizorkin spaces and Besov spaces. (In this appendix only, $q$ denotes the fine index of these spaces and not the exponent of the porous-medium term.)
For the definition of these spaces, we refer to \cite[P. 8]{runst}. They translate to classical function spaces
as in e.g. \cite[P. 14]{runst}, in particular,
\begin{equation}\label{B2}
\begin{split}
&F_{p,2}^{0}(\mathbb{R}^d)=L_p(\mathbb{R}^d),\  1<p<\infty \ \text{ (Lebesgue spaces)},
\\
& F_{p,2}^{m}(\mathbb{R}^d)=W^m_p(\mathbb{R}^d),\  m\in\N, \  1<p<\infty \  \text{(Sobolev spaces)},\\
& F_{p,2}^{s}(\mathbb{R}^d)=H^s_p(\mathbb{R}^d),\  s\in\R, \ 1<p<\infty \ \text{(fractional Sobolev spaces)}.
\end{split}
\end{equation}

Assume that  $0<q_1, q_2\leq \infty$ and $0<p<\infty$ as well as  $s_1$ and $s_2$ are real numbers such that  $s_1\leq s_2$.

\begin{theorem}\label{RS2}
Assume $s_1+s_2>d\cdot\max(0,\frac 1p-1)$, and $q\ge \max(q_1,q_2)$. Then
\begin{itemize}
  \item if $s_2>s_1$, then $F_{p,q_1}^{s_1}(\RR^d) \cdot B_{\infty,q_2}^{s_2} (\RR^d)\hookrightarrow F_{p,q_1}^{s_1}(\RR^d)$;
  \item and, if $s_1=s_2$ then $F_{p,q_1}^{s_1} (\RR^d)\cdot B_{\infty,q_2}^{s_1}(\RR^d) \hookrightarrow F_{p,q}^{s_1}(\RR^d)$.
\end{itemize}
\end{theorem}
\begin{proof}
See \cite[p.\ 229]{runst}.\end{proof}

\begin{theorem}\label{RS2a}
Let  $s_1+s_2>d\cdot\max(0,\frac 1p-1)$.
\begin{itemize}
  \item Assume $s_2>\frac dp$ and $q\ge \max(q_1,q_2)$. Then, if $s_2>s_1$,
  $F_{p,q_1}^{s_1}(\RR^d) \cdot F_{p,q_2}^{s_2} (\RR^d)\hookrightarrow F_{p,q_1}^{s_1}(\RR^d)$;
  if $s_2=s_1$,
  $F_{p,q_1}^{s_1}(\RR^d) \cdot F_{p,q_2}^{s_2} (\RR^d)\hookrightarrow F_{p,q}^{s_1}(\RR^d)$;
  \item Let $s_1=s_2=\frac dp$ and $q\ge \max(q_1,q_2)$. If $0<p\le 1$, then  $F_{p,q_1}^{s_1}(\RR^d) \cdot F_{p,q_2}^{s_2} (\RR^d)\hookrightarrow F_{p,q}^{s_1}(\RR^d)$;
  \item If $s_2<\frac dp$, then
  $F_{p,q_1}^{s_1}(\RR^d) \cdot F_{p,q_2}^{s_2} (\RR^d)\hookrightarrow F_{p,q}^{s_1+s_2-\frac dp}(\RR^d)$.
\end{itemize}

\end{theorem}
\begin{proof}
See \cite[p.\ 190]{runst}.\end{proof}

To work on bounded domains, let us introduce the extensions. For a bounded domain $\CO\subset \RR^d$  for $g\in \mathcal{S}'(\RR^d)$, the restriction of $g$ to $\CO$ is an element of $\mathcal{D}(\CO)$ and will be denoted by $g\Big|_{\CO}$.  For a given function $f:\CO\to\RR$, an element  $g\in \mathcal{S}'(\RR^d)$ is called  an extension of $f$ if $f=g\Big|_{\CO}$. We recall the following definition from \cite[Definition 5.3, P. 44]{Trieb}.

\begin{definition}
Let $s\in\RR$ and $0<q\le \infty$.
\begin{enumerate}
\item if $0<p<\infty$, then we put
$$
F_{p,q}^s(\CO)=\lk\{ f\in\mathcal{D}'(\CO):\exists \, g\in F^s_{p,q}(\RR^d)\,\mbox{ with } g\big|_{\CO}=f\rk\}
$$
and
$$
|f|_{F_{p,q}^s}:=|f|_{F_{p,q}^s(\CO)}=\inf |g|_{F_{p,q}^s(\RR^d)},
$$
where the infimum is taken over all $g\in  F_{p,q}^s(\RR^d)$ such that $g\big|_{\CO}=f$.
\item if $0<p\le \infty$, then we put
$$
B_{p,q}^s(\CO)=\lk\{ f\in\mathcal{D}'(\CO):\exists \, g\in B^s_{p,q}(\RR^d)\,\mbox{ with } g\big|_{\CO}=f\rk\}
$$
and
$$
|f|_{B_{p,q}^s}:=|f|_{B_{p,q}^s(\CO)}=\inf |g|_{B_{p,q}^s(\RR^d)},
$$
where the infimum is taken over all $g\in B_{p,q}^s(\RR^d)$ such that $g\big|_{\CO}=f$.
\end{enumerate}
\end{definition}

In the case where  $\CO$ is a $C^\infty$ bounded domain, it is proved in \cite[P. 52]{Trieb}  that the relations given in \eqref{B2} are  still valid. In addition, it is shown in \cite[Theorem C.1]{Debbi} that Theorem \ref{RS2} and Theorem \ref{RS2a} are still valid for a  $C^\infty$ bounded domain $\CO\subset\mathbb{R}^d$ in place of $\mathbb{R}^d$.

\medskip

{
{\bf Acknowledgment:}
{ We would like to thank Dr Paul Razafamandimby, City University Dublin, Ireland, for his inspiring discussions and helpful remarks.
}

{\bf Declarations:}
The research leading to these results received funding from the Austrian Science Foundation (FWF), Project P 32295 (Grant-DOI 10.55776/P32295) and partly by
Project PAT-3076923
(Grant-DOI  10.55776/PAT3076923).
The authors have no relevant financial or non-financial interests to disclose.
The Montanuniversitaet Leoben, Leoben, Austria employs both authors.}

\newcommand{\bn}{\boldsymbol{\nu}}

\renewcommand{\bv}{\mathbf{v}}
\newcommand{\w}{\mathbf{w}}
\newcommand{\el}{\mathbb{L}}
\newcommand{\elb}{\mathbf{L}}
\newcommand{\ve}{\mathrm{V}}
\newcommand{\h}{\mathbb{H}}
\newcommand{\rH}{\mathrm{ H}}
\newcommand{\bH}{\mathbf{H}}
\newcommand{\rK}{\mathrm{X}}
\newcommand{\rV}{\mathrm{ V}}
\newcommand{\divv}{\mathrm{div }\;}
\newcommand{\rve}{\rVert}
\newcommand{\lve}{\lVert}
\newcommand{\bd}{\mathbf{n}}
\renewcommand{\v}{\bv}
\newcommand{\rA}{\mathrm{A}}
\newcommand{\rrA}{\mathrm{A}_1}
\newcommand{\hrrA}{\hat{\mathrm{A}}_1}
\newcommand{\trrA}{\tilde{\mathrm{A}}_1}
\newcommand{\db}{\bar{\d}}
\renewcommand{\d}{\mathbf{n}}
\newcommand{\MO}{\mathcal{O}}

\end{document}